\documentclass{article}
\usepackage[%
journal=FoCM,    
lang=british,   
]{ems-journal}

\usepackage{longtable}
\usepackage{tikz-cd}

\allowdisplaybreaks

\newtheorem{theorem}{Theorem}[section]

\newtheorem{lemma}[theorem]{Lemma}

\newtheorem{remark}[theorem]{Remark}
\theoremstyle{definition}

\def\u{{\bf u}}

\def\d{{\mathrm d}}

\def\R{{\mathbb R}}

\def\D{{\mathbb D}}

\def\ud{\underline{D}}
\def\md{\partial_{t}^\bullet}

\def\Gm{{\Gamma^m}}

\def\Ghm{{\Gamma^m_h}}
\def\GhM{{\Gamma^{m+1}_h}}
\def\Ghsm{{\hat\Gamma^m_{h, *}}}
\def\GhsM{{\hat\Gamma^{m+1}_{h, *}}}

\def\Ghsf{{\Gamma_{h,\rm f}^0}}
\def\Ghso{{\hat\Gamma_{h,*}^0}}

\def\Hm{{H^m}}

\def\nm{{n^m}}

\def\nhm{{n^m_h}}

\def\nsm{{n^m_{*}}}
\def\nsM{{n^{m+1}_{*}}}

\def\Tsm{{T^m_{*}}}
\def\TsM{{T^{m+1}_{*}}}

\def\Nsm{{N^m_{*}}}
\def\NsM{{N^{m+1}_{*}}}

\def\mum{{\mu_m}}
\def\musm{{\mu_{*,m}}}
\def\num{{\nu_m}}
\def\nusm{{\nu_{*,m}}}

\def\uhm{{u_h^m}}

\def\khm{{\kappa_h^m}}

\def\phm{{p_h^m}}

\def\Xhm{{X_h^m}}
\def\XhM{{X_h^{m+1}}}

\def\eum{{e_u^m}}

\def\epm{{e_p^m}}

\def\ekm{{e_\kappa^m}}

\def\exm{{e_x^m}}
\def\exM{{e_x^{m+1}}}
\def\ehxm{{\hat e_x^m}}
\def\ehxM{{\hat e_x^{m+1}}}

\def\Om{{\Omega^{m}}}
\def\OhM{{\Omega_{h}^{m+1}}}
\def\Ohm{{\Omega_{h}^{m}}}

\def\Ohmpm{{\Omega_{h,\pm}^{m}}}

\def\Ohso{{\hat\Omega_{h,*}^{0}}}
\def\Ohsop{{\hat\Omega_{h,*,+}^{0}}}
\def\Ohsom{{\hat\Omega_{h,*,-}^{0}}}

\def\Ohsf{{\Omega_{h,{\rm f}}^{0}}}
\def\Ohsfp{{\Omega_{h,{\rm f},+}^{0}}}
\def\Ohsfm{{\Omega_{h,{\rm f},-}^{0}}}
\def\Ohsfpm{{\Omega_{h,{\rm f},\pm}^{0}}}

\def\Ohsm{{\hat\Omega_{h,*}^{m}}}
\def\OhsM{{\hat\Omega_{h,*}^{m+1}}}

\numberwithin{equation}{section}

\begin{document}

\title{A convergent finite element method for two-phase Stokes flow driven by surface tension}
\titlemark{A convergent FEM for two-phase Stokes flow}


%

\emsauthor{1}{
	\givenname{Genming}
	\surname{Bai}
	\mrid{}
	\zblid{}
	\orcid{}}{G.~Bai}
\emsauthor{2}{
	\givenname{Harald}
	\surname{Garcke}
	\mrid{}
	\zblid{}
	\orcid{}}{H.~Garcke}
\emsauthor{3}{
	\givenname{Shravan}
	\surname{Veerapaneni}
	\mrid{}
	\zblid{}
	\orcid{}}{S.~Veerapaneni}

\Emsaffil{1}{
	\department{Department of Mathematics and Statistics}
	\organisation{Old Dominion University}
	\rorid{01a2bcd34}
	\address{4700 Elkhorn Avenue}
	\zip{VA 23508}
	\city{Norfolk}
	\country{USA}
	\affemail{gbai@odu.edu}
	
	\department{2}{Department of Mathematics} 
	\organisation{2}{University of Michigan}%
	\rorid{2}{}
	\address{2}{530 Church Street}%
	\zip{2}{MI 48109}
	\city{2}{Ann Arbor}
	\country{2}{USA} 
	\affemail{2}{}}
\Emsaffil{2}{
	\department{Department of Mathematics}
	\organisation{University of Regensburg}
	\rorid{00f54p054}
	\address{Universit\"atsstra\ss e 31}
	\zip{93053}
	\city{Regensburg}
	\country{Germany}
	\affemail{harald.garcke@ur.de}}
\Emsaffil{3}{
	\department{Department of Mathematics}
	\organisation{University of Michigan}
	\rorid{00jmfr291}
	\address{530 Church Street}
	\zip{MI 48109}
	\city{Ann Arbor}
	\country{USA}
	\affemail{shravan@umich.edu}}

\classification[65M12]{65M60}


\communicated{Endre S\"uli}

\makeatletter
\ems@print@communicatedtrue
\makeatother

\begin{abstract}
We present the first convergence proof for an iso-parametric finite element discretization of two-phase Stokes flow in $\Omega \subset \mathbb{R}^d$, $d=2,3$, with interface dynamics governed by mean curvature.
The proof relies on a crucial discrete coupled parabolicity structure of the error system and a powerful iso-parametric framework of convergence analysis where we do not strictly discriminate consistency and stability. This new mixing idea leads to a non-trivial construction of the bulk mesh in the consistency analysis.
The techniques and analysis developed in this paper provide fundamental numerical analysis tools for general curvature-driven free boundary problems.
\end{abstract}

\maketitle


\section{Introduction}

Surface tension-driven interface problems describe the evolution of an interface (or free boundary) between two distinct regions or phases, where the motion of the interface is governed by its curvature. Such problems commonly arise in mathematical and physical models of phase transitions, grain growth, and materials science. Due to their practical significance, curvature-driven interface problems have attracted substantial interest from both the partial differential equations (PDE) and numerical analysis communities over the years. Notable examples include the Mullins–Sekerka free boundary problem, the Stefan problem, and curvature-driven two-phase flows; for a comprehensive overview, see \cite{Garcke13}.

In this paper, we study the flow of viscous, incompressible, immiscible two-fluid systems in a low Reynolds number regime, where inertia can be neglected.
Let $\Omega$ be a closed bounded domain in $\R^d$, $d=2,3$, containing one single closed and smooth $(d-1)$-dimensional interface denoted by $\Gamma(t)$ which evolves according to Newton's second law, balancing bulk stress forces and interfacial tension. We define two open sets $\Omega_+(t)$ and $\Omega_-(t)$ to be the exterior and interior components of $\Omega\backslash(\partial\Omega\cup\Gamma(t))$, and define the union open set $\Omega_\pm(t):=\Omega_+(t)\cup \Omega_-(t)=\Omega\backslash(\partial\Omega\cup\Gamma(t))$.
The governing equations for the velocity $u$ and pressure $p$ are described by the stationary Stokes equations:
\begin{subequations}
	\begin{align}
		-\nu \Delta u + \nabla p &= f
		\quad\mbox{in }\Omega_\pm(t)
		,
		\label{subeq:PDE1}\\
		\nabla\cdot u &= 0 
		\quad\mbox{in }\Omega_\pm(t)
		,
		\label{subeq:PDE2}
	\end{align}
	where the positive viscosity coefficient $\nu$ is piecewise constant, $f$ is a possible external force (e.g., gravitational force), and the first equation can be equivalently written as
	\begin{align}
		\nabla\cdot \sigma + f = 0
		\notag
	\end{align}
	with the viscous shear stress tensor
	$\sigma = - p I + 2\nu \D u := - p I + \nu (\nabla u + (\nabla u)^\top)$.
	Across the interface, we enforce the continuity of the velocity:
	\begin{align}\label{subeq:PDE3}
		[\![u]\!] = u|_{\Gamma^+} - u|_{\Gamma^-} = 0 
		\quad\mbox{on }\Gamma(t)
		.
	\end{align}
	Furthermore, the force balance at the interface gives rise to the second jump condition:
	\begin{align}\label{subeq:PDE4}
		[\![\sigma n]\!] = \gamma_0 H n
		\quad\mbox{on }\Gamma(t)
		,
	\end{align}
	where $n$ denotes the outward unit normal vector, $H$ is the mean curvature, and $\gamma_0 > 0$ represents the surface tension coefficient.
	Lastly, the evolution of the interface is governed by the velocity equation
	\begin{align}\label{subeq:PDE5}
		\md X = u \quad\mbox{in }\Omega,
	\end{align}
	where $\md$ denotes the material derivative of the position $X$ along the velocity field $u$.
\end{subequations}

The local and global well-posedness theories for the system \eqref{subeq:PDE1}--\eqref{subeq:PDE5} are well developed and have been studied by many authors, including Beale \cite{Beale81,Beale84}, Lynn \& Sylvester \cite{LS90}, Tani \& Tanaka \cite{TK96}, and Abels \cite{Abels07}. In contrast, to date, there are no known provable convergent numerical schemes for \eqref{subeq:PDE1}--\eqref{subeq:PDE5}, despite the system’s significance in practical applications.

The significance of this paper lies in the construction of the first
convergent finite element scheme of arbitrarily high order and the establishment of a comprehensive numerical analysis framework for the (iso-)parametric discretization where the non-local interface curvature dynamics dominates.

We note that, despite the absence of convergence proofs, unconditionally stable schemes for curvature-driven interface models have been successfully developed in many previous works, including \cite{DLY22,BGN15a,BGN15b,BGN16,Bansch01,NST16,GTZ2026}.
For readers interested in scientific computing approaches, we refer to \cite{VGZB09,VRBZ11,XBS13,ST18} for fast spectral methods based on potential theory; \cite{FZ23,OQS22,GNZ23} for cut-FEMs; \cite{ZWT10,Fu20,GNZ24} for arbitrary Lagrangian–Eulerian (ALE) methods; and \cite{GRR06,OK2005,OKZ2007} for level-set methods.

In \cite{LQ25}, the authors consider an arbitrary Lagrangian–Eulerian interface tracking method for two-phase Navier–Stokes flow without surface tension.
In this case,  no curvature terms need to be dealt with.
They establish the $H^1$-optimal convergence for the semi-discrete iso-parametric finite element method, provided that the finite element degree satisfies $k\geq2$.
We also mention that, in the recent work \cite{EKL24}, a convergence proof was established for a bulk-boundary finite element method for a tumor growth model \cite{EKS19}. It is important to point out, however, that the free-boundary model problem considered in \cite{EKL24} possesses a favorable local parabolicity structure, allowing the application of the standard (iso-)parametric finite element techniques as developed in \cite{DDE05,DE13,KLL17,KLL19,KLL-Willmore,ER21}. In contrast, for the curvature-driven interface problem \eqref{subeq:PDE1}--\eqref{subeq:PDE5}, the associated differential operator is non-local and pseudo-differential (cf. \cite{Garcke13}), resulting in a much more complicated stability structure.

The main contribution of this paper is to develop the discrete numerical stability and convergence 
of the system \eqref{subeq:PDE1}--\eqref{subeq:PDE5} for both cases $d=2,3$. Without loss of generality, we may assume the external force $f=0$ and the surface tension coefficient $\gamma_0=1$.
\begin{enumerate}
	\item[{\footnotesize$\bullet$}]
	The most important analytic feature of the system \eqref{subeq:PDE1}--\eqref{subeq:PDE5} is the following energy identity which is obtained by testing \eqref{subeq:PDE1} with $u$ and applying integration by parts:
	\begin{align}
		\int_\Omega 2\nu \D_\Omega u \cdot \D_\Omega u 
		+
		\int_{\Gamma(t)} Hn \cdot u
		=
		0
		.
		\notag
	\end{align}
	Moreover, by further integration by parts, invoking Lemma \ref{lemma:ud}, and using the identity $\nabla_{\Gamma(t)} {\rm id}_{\Gamma(t)}=I-n\otimes n$, the second term can be rewritten as
	\begin{align}
		\hspace{-10pt}
		\int_{\Gamma(t)} Hn \cdot u
		=
		\int_{\Gamma(t)} \nabla_{\Gamma(t)}\cdot u
		=
		\frac{\d}{\d t}
		\int_{\Gamma(t)} 1
		=
		\frac{1}{d-1}
		\frac{\d}{\d t}
		\int_{\Gamma(t)} \nabla_{\Gamma(t)} {\rm id}_{\Gamma(t)}
		\cdot
		\nabla_{\Gamma(t)} {\rm id}_{\Gamma(t)}
		.
		\notag
	\end{align}
	Combining these two identities yields the key energy identity:
	\begin{align}\label{eq:energy-decay}
		\int_\Omega 2\nu \D_\Omega u \cdot \D_\Omega u 
		+
		\frac{1}{d-1}
		\frac{\d}{\d t}
		\int_{\Gamma(t)} \nabla_{\Gamma(t)} {\rm id}_{\Gamma(t)}
		\cdot
		\nabla_{\Gamma(t)} {\rm id}_{\Gamma(t)}
		=
		0
		.
	\end{align}
	Inspired by \eqref{eq:energy-decay}, one of the main objectives of this paper is to establish the following formal discrete energy decay estimate (see Section \ref{sec:parab}):
	\begin{align}\label{eq:energy-decay-disc}
		\int_\Omega 2\nu \D_\Omega e_u \cdot \D_\Omega e_u 
		+
		\int_{\Gamma(t)} \nabla_{\Gamma(t)} \Big(\frac{\partial}{\partial t} e_x \Big)
		\cdot
		\nabla_{\Gamma(t)}
		e_x
		\leq
		{\rm consistency\,\,errors}
		,
	\end{align}
	where $e_u$ and $e_x$ denote the formal errors of velocity and position in the bulk regions. We refer to \eqref{eq:energy-decay-disc} as the discrete coupled parabolicity, as its left-hand side provides positive definiteness in both the bulk and interface regions.
	This coupling reflects the nonlocal nature of the system.
	
	\item[{\footnotesize$\bullet$}]
	However, the discrete coupled parabolicity \eqref{eq:energy-decay-disc} is not sufficient to conclude the convergence.
	It is well-known that the bulk and surface discrepancy errors are of order $\nabla_\Omega e_x$ and $\nabla_\Gamma e_x |_\Gamma$. Unfortunately, these quantities are not stable and cannot be directly controlled by the left-hand side of \eqref{eq:energy-decay-disc}.
	Previous analyses \cite{BL24,BL24FOCM,BGV26} indicate the interface discrepancy error $\nabla_\Gamma e_x |_\Gamma$ can be resolved using the projection error with the help of super-convergence results, but how to handle the bulk discrepancy error $\nabla_\Omega e_x$ is still unclear.
	
	To address this issue, another major contribution of this paper is the recovery of stability for the bulk discrepancy error $\nabla_\Omega e_x$ through a non-trivial construction of the consistency mesh $\Ohsm$ (see Section \ref{sec:bulk-mesh}). The construction of $\Ohsm$ depends explicitly on the numerical mesh $\Ohm$ (cf. \eqref{eq:Ohsm-def}). In this approach, we intertwine the notions of consistency and stability: While we sacrifice the stability of consistency errors, we gain significantly stronger control over the bulk discrepancy error $\nabla_\Omega e_x$ (see Section \ref{sec:Eh}). It is the subtle balance of this trade-off that enables us to close the stability loop and establish convergence.
\end{enumerate}
Meanwhile, several new techniques are introduced to facilitate the main numerical stability arguments described above.
\begin{enumerate}
	\item[{\footnotesize$\bullet$}]
	The construction of the discrete trace extension operator $E_h$ in Section \ref{sec:Eh} enables us to smoothly extend the interface information into the bulk domain. In view of the super-approximation results of the interface quantities (cf. Appendix \ref{sec:super}), the employment of $E_h$ is expected to lead to improved error estimates in the bulk regions.
	The construction of $E_h$ (cf. \eqref{eq:Eh}) is based on stable interpolations for rough data on iso-parametric triangulations, such as the Scott--Zhang interpolation and Cl\'ement-type interpolation; see Remark \ref{rmk:SZ}.
	
	\item[{\footnotesize$\bullet$}]
	In Section \ref{sec:Lenoir}, we introduce a Lenoir-type lift operator $\Phi^m : \Ohsm \rightarrow \Omega$ tailored for the two-phase flow setting. As discussed above, since the consistency bulk mesh $\Ohsm$ is constructed based on the numerical bulk mesh $\Ohm$, the Lenoir-type lift operator $\Phi^m$ is not purely a consistency quantity. Its approximation properties depend on the quality of the bulk mesh triangulation and are not guaranteed to be uniformly bounded a priori. Therefore, the operator $\Phi^m$ must also be interpreted as a part of the overall stability.
	
	\item[{\footnotesize$\bullet$}]
	A new approximation property of the Lenoir-type lift operator $\Phi^m$ is established in Lemma \ref{lemma:Phi-approx}, where we obtain an additional $h^{1/2}$ order of approximation compared to Lenoir's original result \cite[Lemma 5]{Lenoir86}. This improvement is achieved through the size relations for pointwise evaluations and (semi-)norms, together with a scaling argument. As a result, we gain an extra factor of $h^{1/2}$ in the bulk perturbation estimate (see Lemma \ref{lemma:geo-perturb-b}). This additional $h^{1/2}$ is crucial for guaranteeing the overall numerical stability and convergence of the method.
	
	\item[{\footnotesize$\bullet$}]
	Unlike the continuous PDE or spatially semi-discrete discretizations, where the divergence-free condition is naturally satisfied, in the fully discrete scheme, the discrepancy between the domains $\Ohsm$ and $\Ohm$ generally results in the velocity error not meeting the weak discrete divergence-free condition,
	\begin{align}\label{eq:div-free}
		\int_\Ohsm \nabla_\Ohsm \cdot e_u^m \psi_h \neq 0,
	\end{align}
	where $\psi_h$ is an arbitrary finite element test function.
	This can lead to an unresolvable entanglement between $u$ and $p$ in the error equation -- in this case, the standard inf-sup argument no longer applies. 
	To address this, we develop a new $H^{-1/2}$ estimate for the pressure error $\epm$ in Section \ref{sec:epm}. Combined with the mapping properties of the discrete trace extension operator $E_h$, the $H^{-1/2}$ estimate ensures the stability of the bilinear form \eqref{eq:div-free} in the error analysis.
	
	\item[{\footnotesize$\bullet$}]
	In Lemma \ref{lemma:e-blinear}, we establish a new bilinear error estimate for the mass bilinear form, whose upper bound does not involve any gradients due to the intrinsic orthogonality of the projection error $\ehxm$. This result represents a significant improvement over previously known results (cf. \cite{Kov18,KLL19}). The absence of gradient terms in the upper bound is crucial for ensuring the stability of the numerical errors.
	
	\item[{\footnotesize$\bullet$}]
	Time-marching estimates developed in Section \ref{sec:marching} provide important a priori information at the next time level.
	Lemma \ref{lemma:e-diff-W1inf} helps us avoid the cubic nonlinearity of errors which cannot be controlled by the induction hypothesis and Gr\"onwall's inequality.
	
	\item[{\footnotesize$\bullet$}]
	Section \ref{sec:traj-est} presents a comprehensive discussion of the convergence of trajectory errors. This analysis demonstrates that the solution-dependent consistency concepts introduced in Section \ref{sec:cons_anal} are indeed compatible with the traditional framework, thereby validating the effectiveness of our novel approach of intertwining consistency with stability.
	
	\item[{\footnotesize$\bullet$}]
	A new nonlinear super-approximation estimate is established in Lemma \ref{lemma:super_conv-nonlinear} which is particularly useful for handling errors arising from function compositions in Lemma \ref{lemma:hatX-W1inf} and \eqref{eq:e-sharp-i}.
\end{enumerate}

The rest of this paper is organized as follows. Basic notation (e.g., function spaces, (iso-)parametric finite element spaces and their corresponding Lagrange interpolations) are introduced in Section \ref{sec:notations}. We then state the numerical scheme and the main result in Section \ref{sec:main-rerults}. Section \ref{sec:cons_anal} handles consistency errors, including the construction of the discrete extension operator $E_h$, the bulk consistency mesh $\Ohsm$ and the Lenoir-type lift operator $\Phi^m$,
all of which are crucial components of the overall proof.
The main stability argument is carried out in Section \ref{sec:stab-anal}, where the discrete coupled parabolicity serves as the dominant stability structure.
Trajectory estimates and the a priori boundedness of the shape regularity constants are discussed in Section \ref{sec:shape-reg}.
The proof of the main convergence theorem (Theorem \ref{thm:main}) is completed in Section \ref{sec:conv-err}, drawing upon the preceding analyses.
Several well-known results concerning calculus on moving domains and super-approximation estimates are collected in Appendices \ref{sec:notation} through \ref{sec:super}. Finally, detailed proofs of three technical results, namely, Eq.~\eqref{eq:uhm-epm}, Lemma~\ref{lemma:NT_stab_ref} and Lemma~\ref{lemma:e-convert}, are recorded in Appendices \ref{sec:H3/2}, \ref{sec:appndix_tan_stab} and \ref{sec:e-convert}, respectively.

\section{Basic notation}
\label{sec:notations}

\subsection{Continuous weak formulation}

Let $D$ be either a Lipschitz domain or a closed Lipschitz hypersurface. For $s\geq0$ and $q\geq 1$, we denote by $W^{s,q}(D)$ the fractional-order Sobolev–Slobodeckij spaces; see \cite[Sections 2.3 and 2.4]{SS10book} and \cite[Chapter 1]{LM2012book} for detailed definitions. If $D$ is a Lipschitz domain, we denote $W_0^{s,q}(D):= \overline{C_c^\infty(D)}^{\| \cdot \|_{W^{s,q}(D)}}$; if $D$ is a closed Lipschitz hypersurface, we set $W_0^{s,q}(D):= W^{s,q}(D)$. Moreover, we introduce the dual spaces $W^{-s,q}(D) := (W_0^{s, q'}(D))^*$, and $W^{-s,q}_0(D) := (W^{s, q'}(D))^*$, where $1/q + 1/q' = 1$. For convenience, we also write $H^s(D):=W^{s,2}(D)$ for all $s\in\R$.
Note that $L^2(\Omega)= L^2(\Omega_\pm(t))$ but $H^1(\Omega)\neq H^1(\Omega_\pm(t))$ since $\Gamma(t)$ is a crack in $\Omega$.
We define the following subspaces encoding boundary conditions:
\begin{align}
	L_0^2(\Omega) &= 
	\{ f\in L^2(\Omega): \int_\Omega f = 0 \} ,
	\notag\\
	H_0^1(\Omega) &= 
	\{ f\in H^1(\Omega): f|_{\partial\Omega} = 0 \} .
	\notag
\end{align}
Then, the weak formulation \eqref{subeq:PDE1}--\eqref{subeq:PDE5} with parameters $(\gamma_0, f) = (1, 0)$ reads as follows:
For almost every $t\in(0,T)$, find $(u(\cdot, t), p(\cdot, t), X(\cdot, t), \kappa(\cdot, t))\in (H_0^1(\Omega)^d\times L_0^2(\Omega)\times H^1(\Omega)^d\times L^2(\Gamma(t))^d)$ such that
\begin{subequations}\label{subeq:2phase}
	\begin{align}
		\int_{\Omega} 2\nu \D_\Omega u \cdot \D_\Omega \phi 
		&=
		\int_\Omega p \nabla_\Omega \cdot \phi
		-
		\int_{\Gamma(t)} \kappa \cdot \phi
		\quad\forall  \phi\in H_0^1(\Omega)^d
		\label{eq:2phase1} , \\
		\int_\Omega \nabla_\Omega \cdot u \psi &= 0  
		\quad\forall  \psi\in L_0^2(\Omega)
		, \label{eq:2phase2}
		\\
		\partial_t X &= u \circ X \quad\mbox{in $\Omega$} , \label{eq:2phase3}
		\\
		\int_{\Gamma(t)} \nabla_{\Gamma(t)} {\rm id} \cdot \nabla_{\Gamma(t)} \eta  
		&=
		\int_{\Gamma(t)} \kappa \cdot \eta
		\quad\forall  \eta\in H^1(\Gamma(t))^d
		, \label{eq:2phase4}
	\end{align}
\end{subequations}
where we denote the parameterized flow map along $u$ by $X(t):\Omega\rightarrow \Omega$, which is the solution of the velocity equation
\begin{align}
	\partial_t X(\cdot, t) = u(X(\cdot, t), t)  \mbox{ in } \Omega \mbox{ for a.e. } t \in  [0,T],
	\notag
\end{align}
satisfying the initial condition $X(x, 0) = x$ for all $x\in \Omega$.
The evolution of the interface is described by $\Gamma(t) = \{X(x, t):x\in \Gamma(0)\}$.
From \eqref{eq:2phase4}, we observe that, in fact, $\kappa = -\Delta_{\Gamma(t)} {\rm id} = Hn$, where the second equality is a basic identity in differential geometry; cf. \cite[Proposition 11.11]{QM2013book}.

\subsection{(Iso-)parametric finite element spaces and their identifications}\label{sec:FEM}

To achieve high-order interface and boundary approximation, we employ the iso-parametric elements of Lenoir's type (cf. \cite{Lenoir86}). At $t=0$, we start from an initial triangulation $\mathcal T_{h,\rm f}^0$ of $\Omega$ which consists of shape-regular and quasi-uniform $d$-dimensional simplices with flat $(d-1)$-dimensional facets and maximal mesh size $h$.
We denote by $\Omega_{h,{\rm f}}^0$ the underlying closed domain of $\mathcal T_{h,\rm f}^0$, whose boundary $\partial\Omega_{h,{\rm f}}^0$ is piecewise flat.
We assume $\mathcal T_{h,\rm f}^0$ is fitted to the initial interface $\Gamma^0:=\Gamma(0)$ in a way that there is a canonical piecewise flat discrete interface $\Ghsf$ which is the union of some $(d-1)$-dimensional facets of $\mathcal T_{h,\rm f}^0$ and whose vertices are all on $\Gamma^0$. 
$\Ghsf$ naturally induces a discrete domain decomposition $\Ohsf = \Ohsfp \cup \Ghsf \cup \Ohsfm$ into three closed sets satisfying $ \Ohsfp \cap \Ohsfm = \Ghsf$.

Now consider a general (iso-)parametric triangulation $\mathcal T$ of a closed bulk domain/interface $D(\mathcal T)$ where we require that each piece $K\in \mathcal T$ can be parametrized by an $\R^d$-valued polynomial of degree 
less than or equal to $k$, i.e., $\mathbb P^k(\hat K)^d \ni F_K:\hat K\rightarrow K$, where $\hat K$ is the $d$ or $(d-1)$-dimensional reference element.
We denote the set of nodes of $\mathcal T$ by $\mathcal N(\mathcal T)$.
By pulling back via $F_K$, we can canonically define the (iso-)parametric finite element function space on $\mathcal T$:
\begin{align}\label{eq:ShT}
	S_h^k(\mathcal T)
	:=
	\{
	f\in H^1(D(\mathcal T)): f|_K\circ F_K\in \mathbb P^k(\hat K)\quad\forall K\in\mathcal T
	\}.
\end{align}
Given a triangulation $\mathcal T$, we use $W_h^{s,q}(\mathcal T)$ to denote the piecewise Sobolev spaces.
The curved triangulation $\mathcal T$ satisfying the local parametrization requirement above is not difficult to construct. For example, any $W^{1,\infty}$ homeomorphism $f_h\in S_h^k(\mathcal T_{h,\rm f}^0)^d$ determines a canonical (iso-)parametric triangulation in the following way:
\begin{align}
	\cup_{K\in \mathcal T_{h,\rm f}^0} \{(f_h|_{K})(K)\} .
	\notag
\end{align}
\begin{remark}\upshape
	If $\mathcal T$ is the $(d-1)$-dimensional curved interface triangulation satisfying the local parametrization requirement above, then $S_h^k(\mathcal T)$ coincides with the standard parametric finite element function spaces defined in \cite{DDE05,DE13,BGN20}.
\end{remark}
We say that two order-$k$ (iso-)parametric triangulations $\mathcal T_1$ and $\mathcal T_2$ are equivalent if $|\mathcal T_1| = | \mathcal T_2 |$ and there exists $f_h\in S_h^k(\mathcal T_1)^d$ such that for any $K_2\in\mathcal T_2$ there is a unique $K_1\in\mathcal T_1$ such that $(f_h|_{K_1})(K_1) = K_2$. Equivalently, $\mathcal T_2$ is parametrized by $f_h$ on $\mathcal T_1$.
It is easy to see that if $\mathcal T_1$ and $\mathcal T_2$ are related by $f_h$, then we have a one-to-one map between $S_h^k(\mathcal T_1)$ and $S_h^k(\mathcal T_2)$ via composing $f_h$. In other words, we are free to identify $g_h\in S_h^k(\mathcal T_1)$ as $g_h \in S_h^k(\mathcal T_2)$ up to a composition of the parametrization map $f_h$. 
One can also think of it as the canonical identification of finite element functions by the nodal vector: For any given $g_h\in S_h^k(\mathcal T_1)$, we use $\mathbf g$ to denote its nodal vector. Since $\mathcal T_1$ and $\mathcal T_2$ are equivalent, there is a canonical substantiation of $\mathbf g$ to $g_h\in S_h^k(\mathcal T_2)$.
In the rest of this article, we will use this identification extensively among equivalent triangulations.
In addition, we will simply write the domain in place of triangulation when the underlying triangulation can be easily read off from the context. For instance, since the triangulations of $\Ohsfp$ and $\Ghsf$ are clear, we can simply write $S_h^k(\Ohsfp)$ and $S_h^k(\Ghsf)$ instead of $S_h^k(\mathcal T(\Ohsfp))$ and $S_h^k(\mathcal T(\Ghsf))$ without any ambiguity.

We use the notation $S_h^k(\Ohsfpm)$ for the space of finite element functions which are not necessarily continuous across the discrete interface.
Note that $S_h^k(\Ohsfpm) \neq S_h^k(\Omega_{h,\rm f}^0)$. The broken space $S_h^k(\Ohsfpm)$ serves as the approximation space for pressure. 

\subsection{Lagrange interpolations}

Since the finite element space $S_h^k(\mathcal T)$ in \eqref{eq:ShT} is a local definition, one can define the Lagrange interpolation $I_h(\mathcal T): C^0(\mathcal T)\rightarrow S_h^k(\mathcal T)$ via the pullback to the reference element $\hat K$. This process can also be thought of as nodal vector identification. One can view $I_h(\mathcal T)$ as a covariant functor on the category of equivalent triangulations, and it satisfies the following commutative diagram if $\mathcal T_2$ is parametrized by $f_h$ on $\mathcal T_1$:
\[ \begin{tikzcd}
	C^0(\mathcal T_1) \arrow{r}{(f_h^{-1})^*} \arrow[swap]{d}{I_h(\mathcal T_1)} & C^0(\mathcal T_2) \arrow{d}{I_h(\mathcal T_2)} \\%
	S_h^k(\mathcal T_1) \arrow{r}{(f_h^{-1})^*}& S_h^k(\mathcal T_2)
\end{tikzcd}
\]
Similarly, we will omit the triangulation variable and simply write $I_h$ when both the underlying domain and the corresponding triangulation are clear from the context.

From Lenoir's approach (\!\!\cite[Section 3]{Lenoir86}), one can consistently deform each element in $\mathcal T_{h,\rm f}^0$ to construct an equivalent $k$th-order iso-parametric triangulation $ \hat {\mathcal T}_{h,*}^0:={\rm Lenoir}(\mathcal T_{h,\rm f}^0)$. 
Its discrete interface, exterior and interior regions are denoted by $\Ghso$, $\Ohsop$ and $\Ohsom$ respectively.
By Lenoir's construction, the initial curved interface $\Ghso$ satisfies the interpolation relation: 
$\hat X_{h,*}^0  = I_h (a|_\Ghsf)$
where $\hat X_{h,*}^0: \Ghsf\rightarrow\Ghso$ is the polynomial parametrization map of $\Ghso$.

The approximation properties of $I_h(\mathcal T)$ are closely related to the shape regularity constants.
Suppose that we have an (iso-)parametric triangulation $\mathcal T$ parametrized by $F(\mathcal T)\in S_h^k(\mathcal{T}_0)^d$ over some canonical reference triangulation $\mathcal{T}_0$. Typically, $\mathcal{T}_0$ is chosen as the flat triangulation $\Omega_{h,{\rm f}}^0$ for bulk regions, and as $\Gamma_{h,{\rm f}}^0$ for interfaces. 
The associated shape regularity quantities are defined as
\begin{align}\label{P}
	\begin{aligned}
		\omega(\mathcal T) 
		&:= \| F(\mathcal T) \|_{W^{k-1,4}_h(D(\mathcal T_0))} + 
		\| F(\mathcal T) \|_{W^{k-2,\infty}_h(D(\mathcal T_0))} 
		+ \| F(\mathcal T)^{-1}\|_{W^{1,\infty}(D(\mathcal T))}
		,\\
		\omega_*(\mathcal T) 
		&:= \| F(\mathcal T) \|_{H_h^{k}(D(\mathcal T_0))} 
		.
	\end{aligned}
\end{align}
The following standard Lagrange interpolation error estimates hold due to the local approximation property on the reference triangulation $\mathcal T_0$ and the chain rule for differentiation; cf. \cite[Eqs. (3.2)--(3.5)]{BL24FOCM}.
\begin{lemma}\label{lemma:Ih}
	For any function $f\in W_h^{k+1,\infty}(D(\mathcal T)) \cap C^0(D(\mathcal T))$ where $D(\mathcal T)$ is the underlying domain of the (iso-)parametric triangulation $\mathcal T$ with mesh size $O(h)$, we have
	\begin{align}
		&\| (1-I_h) f \|_{L^{2}(D(\mathcal T))}
		+
		h
		\| \nabla_{D(\mathcal T)} (1-I_h) f \|_{L^{2}(D(\mathcal T))}
		\notag\\
		&\quad
		\leq
		C_{\omega(\mathcal T)} (1 + \omega_*(\mathcal T))
		h^{k+1}\| f \|_{W_h^{k+1,\infty}(D(\mathcal T))} .
		\notag
	\end{align}
\end{lemma}

\begin{remark}
	When $k=3$, the chain rule yields
	\begin{align}
		\| \nabla_{D(\mathcal T_0)}^4 (f\circ F(\mathcal T)) \|_{L^2(D(\mathcal T_0))}
		&\leq
		C \| f \|_{W_h^{4,\infty}(D(\mathcal T))}
		\| \nabla_{D(\mathcal T_0)}^2 F(\mathcal T) \|_{L^{4}(D(\mathcal T_0))}^2
		+
		\cdots
		\,
		.
		\notag
	\end{align}
	To control the first term on the right-hand side, we incorporate the $W_h^{k-1,4}$ norm into the definition \eqref{P}.
\end{remark}


\section{Numerical scheme and the main results}
\label{sec:main-rerults}

\subsection{Numerical scheme}
With the notation and identifications introduced in Section \ref{sec:FEM}, we are ready to state the fully discrete finite element method for the continuous weak formulation \eqref{subeq:2phase}.

At time level $m,m=0,\cdots,[T/\tau]$, with $T>0$ and $\tau>0$ being the final time and the uniform time step, respectively, suppose we know, from the previous time level, the discrete triangulation $\Ohm = \Ohmpm \cup \Ghm$ which is equivalent to $\Ohsfpm \cup \Ghsf$ and parametrized by $\Xhm: \Ohsf\rightarrow\Ohm$.
Then, we define the following finite element subspaces which encode boundary conditions:
\begin{align}
	V_h^k(\Ohm) &= S_h^k(\Ohm) \cap H_0^1(\Ohm) ,
	\notag\\
	Q_h^k(\Ohm) &= S_h^k(\Ohmpm) \cap L_0^2(\Ohm) .
	\notag
\end{align}
Recall that $S_h^k(\Ohmpm)$ denotes the broken finite element space.
We aim to find the finite element solution $(\uhm, \phm, \XhM, \khm)\in (V_h^k(\Ohm)^d \times Q_h^{k-1}(\Ohm)\times S_h^k(\Ohm)^d\times S_h^k(\Ghm)^d)$ such that the following discrete weak formulation holds for all test functions $(\phi_h, \psi_h, \eta_h)\in (V_h^k(\Ohm)^d \times Q_h^{k-1}(\Ohm)\times S_h^k(\Ghm)^d)$: 
\begin{subequations}\label{subeq:2phase-h}
	\begin{align}
		\int_{\Ohm} 2\nu \D_\Ohm \uhm \cdot \D_\Ohm \phi_h
		&=
		\int_\Ohm \phm \nabla_\Ohm \cdot \phi_h
		-
		\int_\Ghm \khm \cdot \phi_h
		\label{eq:2phase-h1} , \\
		\int_\Ohm \nabla_\Ohm \cdot \uhm \psi_h &= 0  , \label{eq:2phase-h2}
		\\
		\frac{\XhM-\Xhm}{\tau} &= \uhm \qquad\mbox{at finite element nodes} , \label{eq:2phase-h3}
		\\
		\int_\Ghm \nabla_\Ghm \XhM \cdot \nabla_\Ghm \eta_h &=
		\int_\Ghm \khm \cdot \eta_h
		. \label{eq:2phase-h4}
	\end{align}
\end{subequations}
The finite element solution $\XhM: \Ohsf\rightarrow\OhM$ determines the position of the discrete domain $\OhM$ and interface $\GhM$ in the next time step.

Since mixed finite elements are used, we adopt the following convention: $I_h$ denotes the Lagrange interpolation of order $k-1$ when applied to pressure-related quantities, and the Lagrange interpolation of order $k$ in all other cases.
\begin{remark}\upshape
	The velocity equation \eqref{eq:2phase-h3} indicates that the proposed scheme is a pure Lagrangian method (i.e., all mesh vertices follow the flow from the PDE). In practice, it is well known that this velocity might lead to poor-quality meshes or mesh degeneration within a short time.
	A common strategy to mitigate this mesh distortion issue is to apply tangential smoothing techniques to the interface evolution velocity, such as the Barrett--Garcke--N\"urnberg (BGN) type methods \cite{BGN2007JCP, BGN2008JCP, DL24}, and then harmonically extend the smoothed interface velocity into the bulk regions. This approach can be viewed as an instance of a more general class of arbitrary Lagrangian--Eulerian (ALE) methods.
	
	An ALE semi-discrete finite element method, based on harmonic extension without interfacial tangential smoothing, has been analyzed in \cite{LQ25} in the context of tension-free two-phase flows. We believe that the numerical stability of the auxiliary ALE part in our fully-discrete setting is similar to the discussion in \cite{LQ25}. However, extending the arguments of \cite{LQ25} to the projection-error framework used in this paper might be challenging and will be addressed in future work.
	
	Unfortunately, the analysis presented in this paper does not apply to the standard BGN-type interfacial tangential smoothing. We refer the readers to Remark \ref{rmk:BGN} for further discussion.
\end{remark}


\subsection{Geometric relations near the interface}\label{section:geometry}


Let $\Gamma^m := \Gamma(t_m)$, with $t_m := m\tau$, and denote by $n^m$ the outward unit normal vector on $\Gamma^m$. We then define $n^m_*$ as the extension of $n^m$ to a neighborhood of $\Gamma^m$ via $n^m_*:=n^m\circ a^m$, where $a^m$ is the unique smooth distance retraction from a $\delta$-neighborhood $D_\delta(\Gamma^m):=\{x\in\R^d: {\rm dist}(x,\Gamma^m)\le \delta\}$ onto $\Gamma^m$, satisfying the following relation: 
\begin{align*}
	x - a^m(x) = \pm |x-a^m(x)| n^m(a^m(x)) .
\end{align*}
It is known that such a constant $\delta>0$ exists and only depends on the curvature of $\Gamma^m$ (thus $\delta$ is independent of $m$, but possibly dependent on $T$); see \cite[Lemma~14.17]{GT2001} and \cite[Theorem 6.40]{Lee18}.
The (extended) normal projection operators are defined as $N^m:=n^m (n^m)^\top$ and $\Nsm:=n^m_* (n^m_*)^\top$, and the (extended) tangential projection operators are defined as $T^m:=I - N^m$ and $\Tsm:=I - \Nsm$.

We denote by $x_j^m,j=1,\cdots,J,$ with $J=|\mathcal N(\Ghm)|$, the nodes of the numerical interface $\Gamma_h^m$. Then, we define $\hat X_{h,*}^m$ as the unique interface finite element function whose nodal values are $a^m(x_j^m)\in\Gm,j=1,\cdots,J$, and denote by $\Ghsm$ the graph of $\hat X_{h,*}^m$.
This discrete interface $\Ghsm$ for consistency analysis has been studied extensively in \cite{BL24FOCM,BL24,BGV26}, and will play an important role in this paper as well, owing to its excellent interface approximation properties; cf. Appendix \ref{sec:super}. 

In the remainder of this section, we will identify all finite element functions as elements either in $S_h(\Ghsm)$ or $S_h(\Ghsm)^d$, via the canonical nodal vector identification described in Section \ref{sec:FEM}.
We define $X^{m+1}:\Gamma^0\rightarrow\Gamma^{m+1}$ and $Y^{m+1}:= X^{m+1}\circ (X^m)^{-1} :\Gamma^m\rightarrow\Gamma^{m+1}$ to be the exact global and local flow maps along $u(t)$, respectively.
Let the interface finite element function $X_{h,*}^{m+1}:\Ghsm\rightarrow \Gamma_{h,*}^{m+1}$ be the interpolation of the local flow, which is uniquely determined by the relation $$X_{h,*}^{m+1}(p) - \hat X_{h,*}^m(p) = Y^{m+1}(p) - {\rm id_\Gm}(p)\quad\forall p\in\mathcal N(\Ghsm)\subset\Gm .$$
Then, it follows that 
\begin{align}
	&X_{h,*}^{m+1} - \hat X_{h,*}^m =I_h \big((Y^{m+1} - {\rm id_\Gm}) \circ a^m|_\Ghsm \big) &&\mbox{on}\,\,\, \Ghsm , \label{eq:X-id1} \\
	&Y^{m+1} - {\rm id}_\Gm = \tau ( {{u^m }}  +  g^m )  &&\mbox{on}\,\,\,  \Gamma^m, \label{eq:X-id2}
\end{align}
where $u^m := u(t_m)$ and $g^m$ is a smooth correction from the Taylor expansion, satisfying the following $W^{1,\infty}$ estimate: 
\begin{align}\label{W1infty-g}
	\|g^m\|_{W^{1,\infty}(\Gamma^m)}\le C\tau . 
\end{align}


The local trajectory error and the projection error at time level $m$ are defined as $\exm := X_h^m -  X_{h,*}^m$ and $\ehxm := X_h^m -  \hat X_{h,*}^m$ respectively.
According to \cite[Eqs. (3.12)--(3.13)]{BL24FOCM}, we have the following nodal relation
\begin{align}\label{eq:geo_rel_1}
	\ehxm = I_h\big[ (\exm\cdot \nsm)\nsm \big] + r_h^m ,
\end{align}
with $r_h^m := \ehxm - I_h\big[ (\exm\cdot \nsm)\nsm \big]$ satisfying
\begin{align}\label{eq:geo_rel_2}
	|r_h^m| \lesssim |[1 - \nsm (\nsm)^\top ] \exm|^2
	\quad\mbox{at the nodes of $\Ghsm$}.	
\end{align}
$r_h^m$ can be interpreted as a quadratic remainder of the nodal-wise orthogonal projection due to the presence of curvature.

Therefore, we deduce from \eqref{eq:X-id1}, \eqref{eq:X-id2} and \eqref{eq:geo_rel_1} that
\begin{align}\label{eq:geo_rel_3}
	\begin{aligned}
		X_h^{m+1} - X_h^m 
		&= \exM - \ehxm + X_{h,*}^{m+1} - \hat X_{h,*}^m \\
		&= \exM - \ehxm + \tau I_h\big((u^m + g^m) \circ a^m|_\Ghsm \big) ,
	\end{aligned}
\end{align}
at the finite element nodes in $\mathcal N(\Ghsm)$.
This relation helps us convert the numerical displacement $X_h^{m+1} - X_h^m$ to the error displacement $\exM - \ehxm$.

The following geometric identities related to the projection error are well-known (cf. \cite[Eqs. (A.15)--(A.17)]{BL24FOCM}): 
If we define $\rho_h^m := I_h(\Nsm (\hat X_{h,*}^{m + 1} -\hat X_{h,*}^{m})) - I_h((Y^{m + 1} - {\rm id}_\Gm)\circ a^m|_\Ghsm)\in S_h(\Ghsm)^d$, then at all finite element nodes in $\mathcal N(\Ghsm)$, it holds that
\begin{align}
	\Nsm (\hat X_{h,*}^{m + 1} -\hat X_{h,*}^{m})
	&= (Y^{m + 1} - {\rm id}_\Gm)\circ a^m|_\Ghsm + \rho_h^m 
	\label{eq:geo_rel_4}
	\\
	\mbox{where}\,\,\, |\Nsm \rho_h^m| 
	&\le C_0 \tau^2 + C_0 |T_*^m (\hat X_{h,*}^{m + 1} -\hat X_{h,*}^{m})|^2
	, \label{eq:geo_rel_5}\\
	T_*^m (\hat X_{h,*}^{m + 1} -\hat X_{h,*}^{m})
	&= T_*^m (X_{h}^{m + 1} - X_{h}^{m})
	+ T_*^m (N_*^{m+1}\circ\hat X_{h,*}^{m+1}-N_*^{m}\circ\hat X_{h,*}^{m})  \ehxM
	. \label{eq:geo_rel_6}
\end{align}

\begin{remark}\upshape
	The geometric setting here is slightly different from \cite[Figure 4]{BL24FOCM} since in this paper the local flow $X_{h,*}^{m+1}:\Ghsm\rightarrow \Gamma_{h,*}^{m+1}$ is not necessarily almost normal. Nevertheless, the non-zero tangential component of this local flow will only introduce a small term of order $O(\tau^2)$ along the normal direction, making the upper bound in \eqref{eq:geo_rel_5} still valid for the normal projection $\Nsm\rho_h^m$ while the estimate for $\rho_h^m$ itself will degenerate to
	\begin{align}
		| \rho_h^m| 
		&\le C_0 \tau + C_0 |T_*^m (\hat X_{h,*}^{m + 1} -\hat X_{h,*}^{m})|^2 \notag .
	\end{align}
\end{remark} 

For the reader's convenience, a summary of frequently used symbols introduced in this section is provided in Appendix \ref{sec:notation}.

\subsection{Convergence theorem}

Let $\Phi^m:\Ohsm\rightarrow \Omega^m := \Omega(t_m)=\Omega$ be Lenoir's lift operator and let the superscript $^{-\ell}$ denote the inverse lift operator on the interface; see Section \ref{sec:Lenoir} for more details.
With the above definitions of the hatted $X$-variables,
we are now ready to state the main convergence theorem.
\begin{theorem}\label{thm:main}
	Given a final time \(T>0\), suppose that the exact solution \((u(t),p(t),\Gamma(t))\) is sufficiently smooth on \([0,T]\) in the sense of Remark~\ref{rmk:regularity}. Let \(k\geq 3\), and assume the time step condition $\tau \leq C_0 h^k$.
	Suppose further that the initial triangulations \(\Omega_h^0:=\Ohso\) and \(\Gamma_h^0:=\Ghso\) are shape regular in the sense that $\omega_*(\Ohso)+\omega_*(\Ghso)\leq C_0$, and that the initial approximations satisfy 
	$\| X_h^0 - I_h(X^0\circ \Phi^0) \|_{L^\infty(\Omega_h^0)}  \leq C_0 h^{k+1}$.
	Then, for every \(m=0,\ldots,\lfloor T/\tau\rfloor\), the following error estimates hold:
	\begin{align}
		\Big(\sum_{j=0}^m \tau \| u_h^j - I_h(u^j\circ \Phi^j) \|_{H^1(\Omega_{h}^j)}^2\Big)^{1/2}
		&\leq
		C
		(\tau + h^{k-1/2}) ,
		\notag\\
		\Big(\sum_{j=0}^m \tau \| p_h^j - I_h(p^j\circ \Phi^j) \|_{L^2(\Omega_{h, \pm}^j)}^2\Big)^{1/2}
		&\leq
		C
		h^{-1/2}
		(\tau + h^{k-1/2}) ,
		\notag\\
		\max_{j=0,\cdots,m}
		\| X_h^j - \hat X_{h,*}^j \|_{H^1(\Gamma_{h}^j)}
		&\leq
		C
		(\tau + h^{k-1/2}) ,
		\notag\\
		\max_{j=0,\cdots,m}
		\| \kappa_h^j - I_h(H^j n^j)^{-\ell} \|_{L^2(\Gamma_{h}^j)}
		&\leq
		C
		h^{-1}
		(\tau + h^{k-1/2})
		.
		\notag
	\end{align}
	Here \(H^j\) denotes the mean curvature of \(\Gamma^j:=\Gamma(t_j)\), and we define $u^j:=u(t_j)$ and $p^j:=p(t_j)$. In the case \(d=2\), we additionally assume that the interface interpolation nodes associated with \(I_h\) coincide with the Gauss--Lobatto nodes. The constants \(C_0\) and \(C\) are independent of \(\tau\), \(h\), and \(m\), while \(C\) may depend on \(T\) and the exact solution.
\end{theorem}
\begin{remark}\upshape
	The CFL-type constraint $\tau\leq C h^k$ is needed in Section \ref{sec:conv-err}. Some evidence in \cite[Remark 2.2 and 5.1]{BGV26} indicates that this constraint might be unnecessary and removable.
\end{remark}
\begin{remark}\upshape
	The degree condition $k\geq 3$ is required to impose the induction hypothesis \eqref{eq:ind-hypo}.
\end{remark}
\begin{remark}\upshape
	The proof of Theorem \ref{thm:main} can be directly extended to more general situations, including 
	a non-zero external force and the nonlinear stationary Navier--Stokes equations.
\end{remark}
\begin{remark}\upshape\label{rmk:regularity}
	We assume that the exact solution $(u(t), p(t), \Gamma(t))$ is sufficiently smooth on
	\([0,T]\) in the following sense: The flow map \(X(t)\), its inverse, the distance
	projection \(a^m\), the extended normal \(n_*^m\), the mean curvature \(H^m\),
	and the Lenoir lift maps \(\Phi^m\) have all bounded spatial derivatives required by the
	interpolation estimates, super-approximation estimates, and geometric perturbation
	arguments used below. In particular, this assumption should cover the norm boundedness
	needed in Lemma~2.2, Lemmas~C.1--C.5, and the Taylor remainder estimates in
	\eqref{W1infty-g} and Section~6.1. 
\end{remark}

\section{Consistency analysis}\label{sec:cons_anal}

\subsection{Discrete trace extension operator $E_h$}
\label{sec:Eh}

Assuming the initial interface $\Gamma^0$ is sufficiently smooth, we know from \cite[Chapter 2, Theorem 5.8]{Necas11book}
that there exists a continuous trace extension operator $E: W^{s-1/q,q}(\Gamma^0)\rightarrow W^{s,q}(\Omega\backslash \Gamma^0)$ for any $q\in(1,\infty)$ and $s\geq 1$. We may fix a smooth cut-off function $\eta\in C_c^\infty(\Omega)$ such that $\eta=1$ in a small neighborhood of $\Omega^0_-$.

We denote by $\Phi^0: \Ohso\rightarrow\Omega$ the globally continuous and piecewise $C^{k+1}$ lift map constructed by Lenoir \cite[Section 5]{Lenoir86}; see Section \ref{sec:Lenoir} for the detailed construction and approximation properties.

On the iso-parametric sub-triangulations $\Ohsom$ and $\Ohsop$, we can separately define the corresponding Scott--Zhang interpolations (or Cl\'ement-type interpolations; see Remark \ref{rmk:SZ}) $I_h^{\rm SZ}(\Ohsom)$ and $I_h^{\rm SZ}(\Ohsop)$. We furthermore use $I_h^{\rm SZ}(\Ohso)$ or simply $I_h^{\rm SZ}$ to denote their trivial patch-wise gluing.
Note that $I_h^{\rm SZ}$ satisfies the standard iso-parametric interpolation error estimates in Lemma \ref{lemma:Ih}.

For any finite element function $f_h$ defined on the discrete interface (e.g. $\Ghsm$ and $\Ghm$), 
using the conventional identification of nodal values introduced in Section \ref{sec:FEM}, we can first identify its domain as the initial discrete interface, i.e., $f_h\in S_h^k(\Ghso)$, and then lift it onto $\Gamma^0$ via $f_h\circ (a^0|_\Ghso)^{-1}$.
Then, the discrete trace extension operator is defined in the following way:
\begin{align}\label{eq:Eh}
	E_h f_h := I_h^{\rm SZ} \left[\big(\eta E (f_h\circ (a^0|_\Ghso)^{-1})\big)\circ \Phi^0 \right] \in S_h^k(\Ohso)
\end{align}
with the mapping property
\begin{align}\label{eq:Eh-map}
	&\| E_h f_h \|_{H_h^s(\Ohso)}
	\leq
	C
	\Big\| E \big(f_h\circ (a^0|_\Ghso)^{-1}\big)\circ\Phi^0 \Big\|_{H_h^s(\Ohso)}
	\notag\\
	&\leq
	C
	\Big\| E \big(f_h\circ (a^0|_\Ghso)^{-1}\big) \Big\|_{H^s(\Omega)}
	\leq
	C
	\| f_h \|_{H^{s-1/2}(\Ghso)},
\end{align}
where $s\in[1,3/2]$ and we have used the well-known (quasi-local) $W^{s, q}$-stability of the Scott--Zhang interpolation $I_h^{\rm SZ}$ for any $sq\geq 1$ with $s\geq 1$ (cf. \cite[Corollary 4.1]{SZ90}) and the mapping properties of $E$ and $\Phi^0$ (cf. \cite[Section 5.2]{Lenoir86}) successively.  
\begin{remark}\upshape\label{rmk:SZ}
	The Scott--Zhang interpolation is well-defined on triangulations $\Ohsfm$ and $\Ohsfp$ whose facets are flat. 
	The iso-parametric version of the Scott--Zhang interpolation can be defined via pull-back: For instance, given $f_h\in S_h(\Ohsom)$, we set
	\begin{align}
		I_h^{\rm SZ}(\Ohsom) f_h := I_h^{\rm SZ}(\Ohsfm) (f_h\circ \hat X_{h,*}^0), \notag
	\end{align}
	where $\hat X_{h,*}^0: \Ohsf\rightarrow\Ohso$ denotes the parametrization map of $\Ohso$. When the domain is clear from the context, we will omit the domain parameter of $I_h^{\rm SZ}$.
	As a consequence of \cite[Theorem 1]{Lenoir86}, the approximation and stability properties of the Scott--Zhang interpolation remain valid for iso-parametric triangulations, provided that the shape regularity constants are bounded.
	
	We also note that the fractional-order stability of $I_h^{\rm SZ}$ follows from a fractional version of the Bramble--Hilbert lemma; see \cite[Theorem 6.1]{DS80}.
	
	Alternatively, one can use the Cl\'ement-type interpolation in the definition of $E_h$, which shares similar stability and approximation properties as $I_h^{\rm SZ}$; cf. \cite[Appendix B]{Li19}.
\end{remark}
\begin{remark}\upshape
	Since the Scott--Zhang interpolation $I_h^{\rm SZ}$ preserves the discrete boundary data (cf. \cite[Eqs. (2.17), (2.18)]{SZ90}), $E_h f_h$, defined in \eqref{eq:Eh}, is a legitimate extension of $f_h$ from the interface into the bulk domain.
\end{remark}

\subsection{Construction of the consistency bulk mesh $\Ohsm$}\label{sec:bulk-mesh}

To derive the error equation, we need to properly define a set of consistency equations which are reasonably compatible with the numerical scheme \eqref{subeq:2phase-h}.
To this end, the first step is to construct a consistency domain $\Ohsm\approx \Omega^m$ which should be equivalent to the initial triangulation $\Ohso$.
The most widely used construction of $\Ohsm$
in the previous literature of iso-parametric finite element methods and the arbitrary Lagrangian--Eulerian methods is the Lagrange interpolation of the exact flow map. This approach keeps track of the information along particle trajectories.
However, the corresponding bulk mesh discrepancy error is of the same order as $\nabla e_x^m$, i.e. $\D_\Ohm - \D_\Ohsm \approx \nabla e_x^m \, \D_\Ohsm$, which, unfortunately, cannot be controlled by the $H^1$ ellipticity of the Stokes velocity (cf. the first term on the left-hand side of \eqref{eq:parab}). Therefore, the associated errors in the trajectory sense are not stable in general.

For a finer construction of $\Ohsm$, we start from the numerical mesh $\Ohm$ which is already known at $t=t_m$.
Eqs. \eqref{eq:2phase-h3}--\eqref{eq:2phase-h4} can be viewed as a Dziuk-type method for the interface evolution transported by the Stokes velocity $u_h^m$. 
Since it is well-known in \cite{BL24,BL24FOCM,BGV26} that the interface distance projection error $\ehxm$ enjoys better stability than trajectory error $\exm$ (see Section \ref{section:geometry} above \eqref{eq:geo_rel_1} for the definitions of $\ehxm$ and $\exm$), a natural idea is to require $\Ohsm$ fitted to $\Ghsm$.

It remains to determine the bulk part of $\Ohsm$.
This raises the question:
\begin{center}
	{\it To which consistency bulk triangulation should $\Ohm$ be compared?}
\end{center}
Since $\Ghm$ differs from $\Ghsm$ by $\ehxm$, this question is equivalent to: 
\begin{center}
	{\it How to extend the interface deformation $\ehxm$ into the bulk domain in a stable way?}
\end{center}
The discrete extension operator defined in Section \ref{sec:Eh} seems to be a promising candidate according to its excellent stability; see \eqref{eq:Eh-map}.
To this end, we define the consistency bulk mesh $\Ohsm$ with its parametrization map given by 
\begin{align}\label{eq:Ohsm-def}
	\hat X_{h,*}^m  := X_{h}^m - E_h \ehxm ,
\end{align}
where $X_{h}^m:\Ohsf\rightarrow\Ohm$ is the parametrization map of $\Ohm$.
In other words, $\Ohsm$ can be viewed as the perturbation of the numerical bulk mesh $\Ohm$ by the extension of the interface projection error $E_h \ehxm$.
Note that $\hat X_{h,*}^m$ was previously defined in Section \ref{section:geometry} as an interface finite element function. Since \eqref{eq:Ohsm-def} defines an extension of this function, for notational simplicity, we retain the same symbol without risk of ambiguity.

By construction, $\hat X_{h,*}^m$ depends directly on the numerical solution $X_h^m$. Therefore, it is not a pure consistency object which is usually defined by the interpolation of the exact smooth solution.
The rigorous stability justification of $\hat X_{h,*}^m$ is a key component of this convergence analysis framework and is addressed in Section \ref{sec:shape-reg}.


For $\Ohsm$ and $\Ghsm$, we use the abbreviation
\begin{align}
	\begin{aligned}
		&\mu_m
		=
		\omega(\Ohsm),
		&&\mu_{*,m}=\omega_*(\Ohsm)
		,\notag\\
		&\nu_m
		=
		\omega(\Ghsm),
		&&\nu_{*,m}=\omega_*(\Ghsm)
		,
	\end{aligned}
\end{align}
where $\omega$ is defined in \eqref{P}.
It is easy to see the size equivalence relations
\begin{equation}\label{eq:mu-nu-equiv}
	\begin{split}
		&\mu_{*,m}\sim \| \hat X_{h,*}^m \|_{H_h^k(\Ohsf)} \sim \| \hat X_{h,*}^m \|_{H_h^k(\Ohso)} ,
		\\
		&\nu_{*,m}\sim \| \hat X_{h,*}^m \|_{H_h^k(\Ghsf)} \sim \| \hat X_{h,*}^m \|_{H_h^k(\Ghso)} .
	\end{split}
\end{equation}
In addition, we introduce the following notation:
\begin{align}
	&\bar\mu_m = \max_{j=1,\cdots,m} \mu_j,\quad
	\bar\nu_m = \max_{j=1,\cdots,m} \nu_j ,
	\notag\\
	&\bar\mu_{*,m} = \max_{j=1,\cdots,m} \mu_{*,j},\quad
	\bar\nu_{*,m} = \max_{j=1,\cdots,m} \nu_{*,j} .
	\notag
\end{align}


\subsection{Lenoir's lift operator $\Phi^m$ revisited}\label{sec:Lenoir}
The main numerical analysis will be carried out on the consistency bulk mesh $\Ohsm$ (see Section \ref{sec:bulk-mesh} for the detailed construction).
Note that the pressure variable $p(t)$ could have jumps across the interface $\Gamma(t)$. To avoid the degeneration of the approximation order due to the discontinuity, we are going to use Lenoir's lift operator $\Phi^m:\Ohsm\rightarrow\Omega$ (cf. \cite[Section 5]{Lenoir86}) which is globally continuous and piecewise smooth. 

We briefly recapitulate the construction of $\Phi^m$ (cf. \cite[Section 5.1]{Lenoir86}):
For each element $K \subset \Ohsm \subset \R^d$, let $0 \leq p_K \leq d-1$ denote the maximum dimension of the facets of $K$ that lie on the interface $\Ghsm$. The facet attaining maximum dimension is unique and is denoted by $e_K$, with $e_K\subset \Ghsm$. For example, in $\R^3$, if $K$ has exactly one edge lying on $\Ghsm$, then $p_K = \dim(e_K) = 1$.
If $K$ does not intersect $\Ghsm$, we set $p_K = 0$.
\begin{itemize}
	\item 
	For an element $K \subset \Ohsm$ such that $p_K = 0$, we define $\Phi^m|_K$ to be the identity map ${\rm id}_K$.
	
	\item 
	For any other element $K \subset \Ohsm$ with $p_K = \dim(e_K) \geq 1$, we define $\Phi^m|_K$ as the mapping which is uniquely determined by the following relation (cf. \cite[Eq. (32)]{Lenoir86}):
	\begin{align}\label{eq:Phim}
		(\Phi^m|_K - {\rm id}_K)\circ F_K := 
		\Big(1 - \sum_{i=p_K+2}^{d+1} \hat\lambda_i \Big)^{k+2}
		(a^m|_\Ghsm - {\rm id}_{\Ghsm}) \circ F_K |_{\hat e_K} \circ \hat Z^{p_K} ,
	\end{align}
	where $k$ denotes the degree of the finite element, $(\hat\lambda_1, \cdots, \hat\lambda_{d+1})$ are the barycentric coordinates on $\hat K$, $F_K: \hat K \rightarrow K$ is the polynomial parametrization map of $K$, 
	and the map $\hat Z^{p_K}: \hat K\rightarrow \hat e_{K}$ is a smooth barycentric coordinate retraction that maps the interior of the reference element $\hat K$ onto the reference interface facet $\hat e_{K}:= (F_K)^{-1}|_{e_K} (e_K)$, whose dimension is $p_K$; see \cite[Fig. 5]{Lenoir86} for an illustration of the retraction map $\hat Z^{\cdot}$ in $d=2,3$.
\end{itemize}
More specifically, in a $d$-dimensional simplex with barycentric coordinates $(\lambda_1, \cdots, \lambda_{d+1})$, the retraction map is defined for $p = 1, \dots, d-1$ as (cf. \cite[Eq. (21)]{Lenoir86})
\begin{align}
	\hat Z^p: (\hat\lambda_1, \cdots, \hat\lambda_{d+1})
	\mapsto
	\Big(\frac{\hat\lambda_1}{1 - \sum_{i=p+2}^{d+1} \hat\lambda_i}, \cdots, \frac{\hat\lambda_{p+1}}{1 - \sum_{i=p+2}^{d+1} \hat\lambda_i}, \underbrace{0, \cdots, 0}_{d-p \text{ zeros}}\Big). \notag
\end{align}
It is straightforward to verify that $\hat Z^p$ is smooth in the interior of the reference element $\hat K$, and continuous up to $\partial\hat K \backslash \{ (0,\cdots, 0, \hat\lambda_{p+2}, \cdots, \hat\lambda_{d+1}): \sum_{i=p+2}^{d+1} \hat\lambda_i = 1 \}$.
Nevertheless, it turns out that the non-smoothness of $\hat Z^p$ does not affect the approximation property of $\Phi^m$, as the factor $(1 - \sum_{i=p+2}^{d+1} \hat\lambda_i )^{k+2}$ in \eqref{eq:Phim} compensates for the singularities.
More precisely, one can show $\Phi^m\in C^\infty({\rm int}(K)) \cap C^{k+1}(\bar K)$ (cf. \cite[Lemma 5]{Lenoir86}).
Furthermore, by construction, $\Phi^m$ is globally continuous, i.e., continuous across adjacent elements (cf. \cite[Lemma 3]{Lenoir86}).

\begin{lemma}\label{lemma:Phi-approx}
	The following reverse trace inequality holds
	\begin{align}
		\| \Phi^m-{\rm id}_\Ohsm \|_{H^1(\Ohsm)}
		&\leq
		C_{\mum, \num}
		h^{1/2}
		\| \Phi^m|_\Ghsm-{\rm id}_\Ghsm \|_{H^{1}(\Ghsm)}
		\notag\\
		&\quad
		+
		C_{\mum, \num}
		h^{3/2}
		\| \Phi^m|_\Ghsm-{\rm id}_\Ghsm \|_{H_h^{2}(\Ghsm)}
		. \notag
	\end{align}
	Using the interpolation error estimate (Lemma \ref{lemma:Ih}),
	\begin{align}
		\| \Phi^m-{\rm id}_\Ohsm \|_{H^1(\Ohsm)}
		\leq
		C_{\mum, \num} (1 + \nusm)
		h^{k+1/2}
		\| \hat X_{h,*}^m \|_{H_h^{k}(\Ghsf)}
		. \notag
	\end{align}
\end{lemma}
\begin{proof}
	We only prove for $d=3$ since the proof of $d=2$ is easier.
	Let the sub-triangulation
	\begin{align}
		\mathcal T_1(\Ohsm)
		&:=
		\{
		K\in \mathcal T(\Ohsm): \mbox{$K$ has one face on $\Ghsm$}
		\} ,
		\notag\\
		\mathcal T_2(\Ohsm)
		&:=
		\{
		K\in \mathcal T(\Ohsm): \mbox{$K$ has exactly one edge on $\Ghsm$}
		\} ,
		\notag\\
		\mathcal T_3(\Ohsm)
		&:=
		\mathcal T(\Ohsm)\backslash(\mathcal T_1(\Ohsm)\cup\mathcal T_2(\Ohsm)) . \notag
	\end{align}
	If $K\in \mathcal T_1(\Ohsm)$, then by the construction of $\hat Z^{2}: \hat K\rightarrow  \hat e_{K}$ with ${\rm dim}(\hat e_{K}) = 2$, the definition \eqref{eq:Phim} yields the following pointwise size relations at any $\hat x \in \hat K$:
	\begin{align}
		|(\Phi^m|_K - {\rm id}_K)\circ F_K (\hat x) |
		&\leq
		\big| (a^m|_{e_K} - {\rm id}_{e_K}) \circ F_K |_{\hat e_K} \circ \hat Z^{2} (\hat x) \big| ,
		\notag\\
		|\nabla_K(\Phi^m|_K - {\rm id}_K)\circ F_K (\hat x) |
		&\leq
		C_{\mu_m, \nu_m}
		\big| \nabla_{e_K}(a^m|_{e_K} - {\rm id}_{e_K}) \circ F_K |_{\hat e_K} \circ \hat Z^{2} (\hat x) \big| , \notag
	\end{align}
	where we have used the cancellation of singularities in \cite[Lemma 5]{Lenoir86} and the following boundedness
	\begin{align}
		&\| (\nabla_{\hat K} F_K)^{-1} \|_{W^{1,\infty}(\hat K)} \| \nabla_{\hat e_K} (F_K |_{\hat e_K}) \|_{W^{1,\infty}(\hat e_K)}
		\notag\\
		&=
		\| \nabla_K(F_K)^{-1} \|_{W^{1,\infty}(K)} \| \nabla_{\hat e_K} (F_K |_{\hat e_K}) \|_{W^{1,\infty}(\hat e_K)}
		\notag\\
		&
		\leq 
		C_{\mu_m} h^{-1} \, C_{\nu_m} h
		=
		C_{\mu_m, \nu_m} . \notag
	\end{align}
	From the estimates above and the construction of $\hat Z^2$, we obtain the following size relations of (semi-)norms
	\begin{align}
		\| (\Phi^m|_K - {\rm id}_K)\circ F_K \|_{L^2(\hat K)}
		&\leq
		C \| (a^m|_{e_K} - {\rm id}_{e_K}) \circ F_K|_{\hat e_K} \|_{L^2(\hat e_{K})} 
		,
		\notag\\
		\| \nabla_K(\Phi^m|_K - {\rm id}_K)\circ F_K \|_{L^2(\hat K)}
		&\leq
		C_{\mu_m, \nu_m} \| 
		\nabla_{e_K}(a^m|_{e_K} - {\rm id}_{e_K}) \circ F_K |_{\hat e_K} \|_{L^2(\hat e_{K})}
		. \notag
	\end{align}
	Transforming back from $\hat K$ to $K$ and summing up over all $K\in\mathcal T_1(\Ohsm)$ yield
	\begin{align}
		\| \Phi^m - {\rm id}_\Ohsm \|_{L^2(\mathcal T_1(\Ohsm))}^2
		&\leq
		C_{\mu_m}
		\sum_{K\in \mathcal T_1(\Ohsm)}
		h^3
		\| (\Phi^m|_K - {\rm id}_K) \circ F_K \|_{L^2(\hat K)}^2
		\notag\\
		&\leq
		C_{\mu_m}
		\sum_{K\in \mathcal T_1(\Ohsm)}
		h^3
		\| (a^m|_{e_K} - {\rm id}_{e_K}) \circ F_K|_{\hat e_K} \|_{L^2(\hat e_{K})}^2
		\notag\\
		&\leq
		C_{\mum, \num}
		\sum_{K\in \mathcal T_1(\Ohsm)}
		h
		\| a^m|_{e_K} - {\rm id}_{e_K} \|_{L^2( e_{K})}^2
		\notag\\
		&=
		C_{\mum, \num}
		h
		\| (1-I_h)a^m|_\Ghsm \|_{L^2(\Ghsm)}^2 .
		\notag
	\end{align}
	The same argument applies to the $H^1$ semi-norm. Thus, we obtain
	\begin{align}
		\| \Phi^m - {\rm id}_\Ohsm \|_{L^2(\mathcal T_1(\Ohsm))}
		&\leq
		C_{\mum, \num}
		h^{1/2}
		\| (1-I_h)a^m|_\Ghsm \|_{L^2(\Ghsm)} 
		,
		\notag\\
		\| \nabla_\Ohsm (\Phi^m-{\rm id}_\Ohsm) \|_{L^2(\mathcal T_1(\Ohsm))}
		&\leq
		C_{\mum, \num}
		h^{1/2}
		\| \nabla_\Ghsm (1-I_h)a^m|_\Ghsm \|_{L^2(\Ghsm)} . \notag
	\end{align}
	If $K\in \mathcal T_2(\Ohsm)$, the only difference is the scaling process. In this case,  since $\hat e_K$ is only one dimensional, we get additional powers of $h$ from scaling:
	\begin{align}
		&\| \Phi^m - {\rm id}_\Ohsm \|_{L^2(\mathcal T_2(\Ohsm))}
		\leq
		C_{\mum, \num}
		h
		\| (1-I_h)a^m|_\Ghsm \|_{L^2(\mathcal E(\Ghsm))} 
		\notag\\
		&\leq
		C_{\mum, \num}
		h^{1/2}
		\| (1-I_h)a^m|_\Ghsm \|_{L^2(\Ghsm)} 
		+
		C_{\mum, \num}
		h^{3/2}
		\| \nabla_\Ghsm (1-I_h)a^m|_\Ghsm \|_{L^2(\Ghsm)} 
		\notag ,
	\end{align}
	and
	\begin{align}
		&\| \nabla_\Ohsm (\Phi^m-{\rm id}_\Ohsm) \|_{L^2(\mathcal T_2(\Ohsm))}
		\leq
		C_{\mum, \num}
		h
		\| \nabla_\Ghsm (1-I_h)a^m|_\Ghsm \|_{L^2(\mathcal E(\Ghsm))}
		\notag\\
		&\leq
		C_{\mum, \num}
		h^{1/2}
		\| \nabla_\Ghsm (1-I_h)a^m|_\Ghsm \|_{L^2(\Ghsm)}
		+
		C_{\mum, \num}
		h^{3/2}
		\| \nabla_\Ghsm^2 (1-I_h)a^m|_\Ghsm \|_{L^2(\Ghsm)} \notag,
	\end{align}
	where $\mathcal E(\Ghsm)$ is the edge set of $\Ghsm$ and we have used an iso-parametric version of the local trace inequality (cf. \cite[Eq. (10.3.8)]{Brenner08}). 
	
	The proof is complete in view of the fact that $\Phi^m - {\rm id}_\Ohsm = 0$ on any $K\in\mathcal T_3(\Ohsm)$. 
\end{proof}
Similarly, using the inverse inequality, we can show
\begin{align}\label{eq:Phi-approx2}
	\| \Phi^m-{\rm id}_\Ohsm \|_{W^{1,\infty}(\Ohsm)}
	&\leq
	C_{\mum, \num}
	\| (1 - I_h) a^m|_\Ghsm \|_{W^{1,\infty}(\Ghsm)}
	\notag\\
	&\leq
	C_{\mum, \num}
	(1 + \nusm)
	h^{k+1/2-d/2}
	\| \hat X_{h,*}^m \|_{H_h^{k}(\Ghsf)}
	.
\end{align}
A straightforward adaptation of the proof of Lemma \ref{lemma:Phi-approx} (also see \cite[Lemma 5]{Lenoir86}) leads to the following high-order approximation estimates on the initial mesh $\Ohso$:
\begin{align}
	\| (\Phi^m-{\rm id}_\Ohsm)\circ \hat X_{h,*}^m \|_{H_h^s(\Ohso)}
	&\leq
	C_{\mum, \num} (1 + \nusm)
	h^{k+3/2-s}
	\| \hat X_{h,*}^m \|_{H_h^{k}(\Ghsf)} ,
	\label{eq:Phi-approx4} \\
	\| (\Phi^m-{\rm id}_\Ohsm)\circ \hat X_{h,*}^m \|_{W_h^{s,\infty}(\Ohso)}
	&\leq
	C_{\mum, \num} (1 + \nusm)
	h^{k+3/2-d/2-s}
	\| \hat X_{h,*}^m \|_{H_h^{k}(\Ghsf)}
	, \label{eq:Phi-approx5} 
\end{align}
for all $s=0,\cdots, k+1$.

By construction, the restriction of Lenoir's lift map $\Phi^m|_\Ghsm=a^m|_\Ghsm:\Ghsm\rightarrow\Gm$ coincides with the inverse lift operator $^{-\ell}$ defined in the standard parametric finite element theory (cf. \cite[Section 3.4]{KLL19}). To be more precise, for any function $f$ defined on $\Gm$, we have
\begin{align}
	f\circ \Phi^m|_\Ghsm = f^{-\ell}
	:= f \circ a^m|_\Ghsm .
	\notag
\end{align}
The lift operator $\ell$ for parametric finite element can be defined similarly: For any function $f$ on $\Ghsm$,
\begin{align}
	f^{\ell}
	:= f \circ \big(a^m|_\Ghsm\big)^{-1},
	\notag
\end{align}
where $\big(a^m|_\Ghsm\big)^{-1}$ is well-defined if $\Ghsm$ is sufficiently close to $\Gm$ which is true given the induction hypothesis later in \eqref{eq:ind-hypo}.


\subsection{Geometric perturbation estimates}

By a change of variables,  the following elementary lemmas quantifying the interface lift perturbation errors for bilinear forms are standard (cf. \cite[Lemma 4.2]{BL24FOCM} and \cite[Lemma 5.6]{Kov18}). 
\begin{lemma}\label{lemma:geo-perturb}
	Given $f_1,f_2\in H^{1}(\Ghsm)$ and their lifts $f_1^\ell,f_2^\ell\in H^{1}(\Gm)$, the following interface geometric perturbation estimates hold for all $(p, q)$ satisfying $1/p+1/q=1/2$:
	\begin{align*}
		&\Big|\int_{\Ghsm} f_1f_2 - \int_{\Gm} f_1^\ell f_2^\ell \Big| 
		\leq 
		C_\num (1 + \nusm)
		h^{k+1} \|f_1\|_{ L^p(\Ghsm)}\|f_2\|_{L^q(\Ghsm)}
		,
		\\
		&
		\Big|\int_{\Ghsm}\nabla_{\Ghsm} f_1\cdot\nabla_{\Ghsm} f_2 - \int_{\Gm}\nabla_{\Gm} f_1^\ell \cdot \nabla_{\Gm} f_2^\ell \Big| \notag\\
		&\qquad\qquad\qquad
		\leq 
		C_\num (1 + \nusm) h^{k+1}\|\nabla_{\Ghsm}f_1\|_{ L^p(\Ghsm)}\|\nabla_{\Ghsm}f_2\|_{L^q(\Ghsm)} ,
	\end{align*}
	and
	\begin{align*}
		&\Big|\int_{\Ghsm}\nabla_{\Ghsm} f_1 \, f_2 - \int_{\Gm}\nabla_{\Gm} f_1^\ell \, f_2^\ell \Big| \notag\\
		&
		\qquad\leq 
		C_\num (1 + \nusm) h^{k}\|\nabla_{\Ghsm}f_1\|_{ L^p(\Ghsm)}\| f_2 \|_{L^q(\Ghsm)} .
	\end{align*}
\end{lemma} 
The following perturbation estimates of bulk domains are a direct consequence of Lemma \ref{lemma:Phi-approx}. The proof is standard and is therefore omitted.
\begin{lemma}\label{lemma:geo-perturb-b}
	Given $g_1,g_2\in H^{1}(\Om)$ and their lifts $g_1\circ\Phi^m,g_2\circ\Phi^m\in H^{1}(\Ohsm)$, the following bulk geometric perturbation estimates hold for all $(p, q)$ satisfying $1/p+1/q=1/2$:
	\begin{align*}
		&\Big|\int_{\Ohsm} g_1\circ\Phi^m g_2\circ\Phi^m - \int_{\Om} g_1 g_2 \Big| 
		\notag\\
		&\qquad\leq 
		C_{\mum, \num} (1 + \nusm)
		h^{k+1/2} \| g_1 \circ\Phi^m \|_{ L^p(\Ohsm)}\|g_2 \circ\Phi^m \|_{L^q(\Ohsm)}
		,
		\\
		&\Big|\int_{\Ohsm} \D_\Ohsm (g_1\circ\Phi^m) \cdot D_\Ohsm (g_2\circ\Phi^m) - \int_{\Om} \D_{\Om} g_1 \cdot \D_{\Om} g_2 \Big| 
		\notag\\
		&\qquad\leq 
		C_{\mum, \num} (1 + \nusm)
		h^{k+1/2} \| \nabla_\Ohsm (g_1\circ\Phi^m)\|_{ L^{p}(\Ohsm)}\| \nabla_\Ohsm (g_2\circ\Phi^m)\|_{L^q(\Ohsm)} ,
	\end{align*}
	and
	\begin{align*}
		&\Big|\int_{\Ohsm} \nabla_\Ohsm (g_1\circ\Phi^m)\, (g_2\circ\Phi^m) - \int_{\Om} \nabla_{\Om} g_1 \, g_2 \Big| 
		\notag\\
		&\qquad\leq 
		C_{\mum, \num} (1 + \nusm)
		h^{k+1/2} \| \nabla_\Ohsm (g_1\circ\Phi^m)\|_{ L^{p}(\Ohsm)}\|g_2\circ\Phi^m\|_{L^q(\Ohsm)} .
	\end{align*}
\end{lemma} 


\subsection{Consistency error estimates}\label{sec:cons_err1}
Based on $\Ohsm$ constructed in Section \ref{sec:bulk-mesh}, the consistency equations associated to the scheme \eqref{subeq:2phase-h} are defined as follows:
\begin{subequations}\label{subeq:2phase-hs}
	\begin{align}
		\int_{\Ohsm} 2\nu \D_\Ohsm I_h(u^m\circ\Phi^m) \cdot \D_\Ohsm \phi_h
		&=
		\int_\Ohsm \bar I_h (p^m\circ\Phi^m) \nabla_\Ohsm \cdot \phi_h
		\notag\\
		&\quad
		- 
		\int_\Ghsm I_h (H^m n^m)^{-\ell} \cdot \phi_h
		+
		d_u^m(\phi_h)
		\label{eq:2phase-hs1} , \\
		\int_\Ohsm \nabla_\Ohsm \cdot I_h(u^m\circ\Phi^m) \psi_h &= d_p^m(\psi_h)  , \label{eq:2phase-hs2}
		\\
		\int_\Ghsm \frac{X_{h,*}^{m+1} - \hat X_{h,*}^m}{\tau} \cdot \chi_h &= \int_\Ghsm I_h (u^m|_\Gm)^{-\ell} \cdot \chi_h 
		+
		d_x^m(\chi_h) \label{eq:2phase-hs3},
		\\
		\int_\Ghsm \nabla_\Ghsm X_{h,*}^{m+1} \cdot \nabla_\Ghsm \eta_h 
		&=
		\int_\Ghsm I_h (H^m n^m)^{-\ell} \cdot \eta_h
		+
		d_\kappa^m(\eta_h)
		, \label{eq:2phase-hs4}
	\end{align}
	where the averaged interpolation $\bar I_h (p^m\circ\Phi^m) := I_h (p^m\circ\Phi^m) - \frac{1}{|\Ohsm|} \int_\Ohsm I_h (p^m\circ\Phi^m)\in Q_h^{k-1}(\Ohsm)$.
\end{subequations}

Subtracting \eqref{subeq:2phase-hs} from \eqref{subeq:2phase}, we get
\begin{align}\label{eq:du}
	d_u^m(\phi_h)
	&=
	\int_{\Ohsm} 2\nu \D_\Ohsm I_h(u^m\circ\Phi^m) \cdot \D_\Ohsm \phi_h
	-
	\int_{\Om} 2\nu \D_\Om u^m \cdot \D_\Om \phi_h^\ell
	\notag\\
	&\quad
	-
	\int_\Ohsm \bar I_h (p^m\circ\Phi^m) \nabla_\Ohsm \cdot \phi_h
	+
	\int_\Om p^m \nabla_\Om\cdot \phi_h^\ell
	\notag\\
	&\quad
	+
	\int_\Ghsm I_h (H^m n^m)^{-\ell} \cdot \phi_h
	-
	\int_\Gm H^m n^m \cdot \phi_h^\ell
	\notag\\
	&=:
	d_u^{m,1}(\phi_h)
	+ d_u^{m,2}(\phi_h)
	+d_u^{m,3}(\phi_h),
	\\
	d_p^m(\psi_h)
	&=
	\int_\Ohsm \nabla_\Ohsm \cdot I_h(u^m\circ\Phi^m) \psi_h
	-
	\int_\Om \nabla_\Om \cdot u^m \psi_h^\ell , \label{eq:dphi}
	\\
	d_x^m(\chi_h)
	&=
	\int_\Ghsm \frac{X_{h,*}^{m+1} - \hat X_{h,*}^m}{\tau} \cdot \chi_h
	-
	\int_\Ghsm I_h (u^m|_\Gm)^{-\ell} \cdot \chi_h , \label{eq:dphi}
\end{align}
and
\begin{align}\label{eq:deta}
	d_\kappa^m(\eta_h)
	&=
	-\int_\Ghsm I_h (H^m n^m)^{-\ell} \cdot \eta_h
	+
	\int_\Gm H^m n^m \cdot \eta_h^\ell
	\notag\\
	&\quad
	+
	\int_\Ghsm \nabla_\Ghsm X_{h,*}^{m+1} \cdot \nabla_\Ghsm \eta_h 
	-
	\int_\Gm \nabla_\Gm {\rm id} \cdot \nabla_\Gm \eta_h^\ell 
	\notag\\
	&=:
	d_\kappa^{m,1}(\eta_h ) + d_\kappa^{m,2}(\eta_h ) .
\end{align}
The estimates for the $d^m$-terms are given in the lemma below.
\begin{lemma}\label{lemma:dm}
	The consistency errors defined in \eqref{eq:du}--\eqref{eq:deta} satisfy
	\begin{align}
		| d_u^m (\phi_h)|
		&\leq
		\big(C_{\mum}
		(1+\musm)h^k + C_{\mum, \num}
		(1+\nusm) h^{k+1/2} \big) \| \phi_h \|_{H^1(\Ohsm)} ,
		\notag\\
		| d_p^m (\psi_h)|
		&\leq
		\big(C_{\mum}
		(1+\musm)h^{k} + C_{\mum, \num}
		(1+\nusm) h^{k+1/2} \big) \| \psi_h \|_{L^2(\Ohsm)} ,
		\notag\\
		| d_x^m (\chi_h)|
		&\leq
		C_\num
		\tau \| \chi_h \|_{H^{-1}(\Ghsm)} ,
		\notag\\
		| d_\kappa^m (\eta_h)|
		&\leq
		C_\num
		(1+\nusm)h^{k+\gamma(d)} \| \nabla_\Ghsm \eta_h \|_{L^2(\Ghsm)} 
		+
		C_\num
		\tau \| \eta_h \|_{L^2(\Ghsm)}
		, \notag
	\end{align}
	where $\gamma(2)=1$ and $\gamma(3)=0$.
\end{lemma}
\begin{proof}
	From the domain perturbation estimates (Lemma \ref{lemma:geo-perturb} and \ref{lemma:geo-perturb-b}), we have
	\begin{align}
		|d_u^{m,1}(\phi_h)|
		&\leq
		\bigg|
		\int_{\Ohsm} \D_\Ohsm (1 - I_h)(u^m\circ\Phi^m) \cdot \D_\Ohsm \phi_h
		\bigg|
		\notag\\
		&\quad
		+
		\bigg|
		\int_{\Ohsm} \D_\Ohsm (u^m\circ\Phi^m) \cdot \D_\Ohsm \phi_h
		-
		\int_{\Om} \D_\Om u^m \cdot \D_\Om \phi_h^\ell
		\bigg|
		\notag\\
		&\leq
		\big(C_{\mum}
		(1+\musm)h^k + C_{\mum, \num}
		(1+\nusm) h^{k+1/2} \big) \| \phi_h \|_{H^1(\Ohsm)}
		\notag ,
	\end{align}
	where the first term on the right-hand side comes from the following bound using the norm equivalence, interpolation error estimate (Lemma \ref{lemma:Ih}), and the boundedness of $\Phi^m$ (cf. \eqref{eq:Phi-approx4}--\eqref{eq:Phi-approx5}):
	\begin{align}
		&\|\mathbb{D}_{\hat{\Omega}_{h,\ast}^m} (1 - I_h) (u^m \circ \Phi^m)\|_{L^2(\hat{\Omega}_{h,\ast}^m)} 
		\notag\\
		&\leq
		C_{\mu_m}
		\| \nabla_\Ohso (1 - I_h) (u^m \circ \Phi^m \circ \hat X_{h,*}^m)\|_{L^2(\hat{\Omega}_{h,\ast}^0)} 
		\notag\\
		&\leq C_{\mu_m} h^k \| u^m \circ \Phi^m \circ \hat X_{h,*}^m \|_{H^{k+1}_h(\hat{\Omega}_{h,\ast}^0)}
		\notag\\
		&\leq C_{\mu_m} h^k 
		\Big(\| u^m \|_{H^{k+1}(\Omega_\pm)}
		\| \Phi^m\circ\hat X_{h,*}^m \|_{W^{1,\infty}(\hat{\Omega}_{h,\ast}^0)}^{k+1}
		+
		\cdots
		\notag\\
		&\quad
		+
		\| u^m  \|_{W^{2,\infty}(\Omega_\pm)}
		\| \Phi^m\circ\hat X_{h,*}^m \|_{W^{1,\infty}(\hat{\Omega}_{h,\ast}^0)}
		\| \Phi^m\circ\hat X_{h,*}^m \|_{H^{k}_h(\hat{\Omega}_{h,\ast}^0)}
		\notag\\
		&\quad
		+
		\| u^m  \|_{W^{1,\infty}(\Omega_\pm)}
		\| \Phi^m\circ\hat X_{h,*}^m \|_{H^{k+1}_h(\hat{\Omega}_{h,\ast}^0)}
		\Big)
		\notag\\
		&\leq C_{\mu_m}h^k 
		\Big(1
		+
		\| (\Phi^m - {\rm id}_{\Ohsm})\circ\hat X_{h,*}^m \|_{H^{k+1}_h(\hat{\Omega}_{h,\ast}^0)}
		+
		\| \hat X_{h,*}^m \|_{H^{k+1}_h(\hat{\Omega}_{h,\ast}^0)}
		\Big)
		\notag\\
		&\leq C_{\mu_m}(1+\mu_{*,m})h^k , \notag
	\end{align}
	for sufficiently small $h \leq h_{\mu_{*,m}, \nu_{*,m}, \mu_m, \nu_m}$.
	Similarly,
	\begin{align}
		|d_u^{m,2}(\phi_h)|
		&\leq
		\big(C_{\mum}
		(1+\musm)h^{k} + C_{\mum, \num}
		(1+\nusm) h^{k+1/2} \big) \| \phi_h \|_{H^1(\Ohsm)} ,
		\notag\\
		|d_u^{m,3}(\phi_h)|
		&\leq
		C_\num
		(1+\nusm) h^{k+1} \| \phi_h \|_{L^2(\Ghsm)}
		. \notag
	\end{align}
	Analogous to $d_u^{m,2}$,
	\begin{align}
		|d_p^{m}(\psi_h)|
		&\leq
		\big(C_{\mum}
		(1+\musm)h^{k} + C_{\mum, \num}
		(1+\nusm) h^{k+1/2} \big) \| \psi_h \|_{L^2(\Ohsm)}
		. \notag
	\end{align}
	The estimate for $d_x^m$ follows from the geometric relation \eqref{eq:X-id1}--\eqref{W1infty-g}:
	\begin{align}
		|d_x^m(\chi_h)|
		\leq
		\| I_h(g^{-\ell}) \|_{H^1(\Ghsm)}
		\| \chi_h \|_{H^{-1}(\Ghsm)}
		\leq
		C_\num \tau
		\| \chi_h \|_{H^{-1}(\Ghsm)} . \notag
	\end{align}
	For $d_\kappa^{m,1}$,
	\begin{align}
		| d_\kappa^{m,1}(\eta_h) |
		=
		| d_u^{m,3}(\eta_h) |
		\leq
		C_\num
		(1+\nusm) h^{k+1} \| \eta_h \|_{L^2(\Ghsm)} . \notag
	\end{align} 
	The term $d_\kappa^{m,2}$ is standard and is estimated as in \cite[Eq. (4.3)]{BL24FOCM}, where the dominant contributions correspond to $d_{21}^m$ and $d_{23}^m$ in that reference. Collecting these estimates, we get:
	\begin{align}\label{eq:dkm2}
		| d_\kappa^{m,2}(\eta_h) |
		\leq
		C_\num
		(1+\nusm) h^{k} \| \eta_h \|_{H^1(\Ghsm)} 
		+
		C_\num
		\tau \| \eta_h \|_{L^2(\Ghsm)}
		.
	\end{align}
	If $d=2$, using the super-approximation result (Lemma \ref{Lemma-GLW}), the factor $h^k$ on the right-hand side of \eqref{eq:dkm2} can be improved to $h^{k+1}$.
	
	The proof is complete.
\end{proof}
For any linear functional $F(\cdot)$ defined on a finite element space $S_h(D)$, where $D$ is a bounded domain equipped with some triangulation $\mathcal T(D)$, we introduce the following associated norms:
\begin{align}
	\| F(\cdot) \|_{W^{s,q}(D)} 
	&=
	\sup_{0\neq\phi_h\in S_h(D)} \frac{| F(\phi_h) |}{\| \phi_h \|_{(W^{s,q}(D))^*}} ,
	\notag\\
	\| F(\cdot) \|_{W_0^{s,q}(D)} 
	&=
	\sup_{0\neq\phi_h\in S_h(D)} \frac{| F(\phi_h) |}{\| \phi_h \|_{(W_0^{s,q}(D))^*}} , \notag
\end{align}
for any $s\in\R$ and $q\geq 1$.
By employing these definitions, one can readily obtain norm estimates for the linear functionals associated with consistency errors. For instance, Lemma \ref{lemma:dm} directly implies
\begin{align}\label{eq:dpm-H1/2}
	\| d_p^m(\cdot) \|_{H^{1/2}(\Ohsm)}
	&\leq
	\big(C_{\mum}
	(1+\musm)h^{k} + C_{\mum, \num}
	(1+\nusm) h^{k+1/2} \big)
	\sup_{\psi_h\neq0} \frac{\| \psi_h \|_{L^2(\Ohsm)}}{\| \psi_h \|_{H_0^{-1/2}(\Ohsm)}}
	\notag\\
	&\leq
	C_{\mum}
	(1+\musm)h^{k-1/2} + C_{\mum, \num}
	(1+\nusm) h^{k} ,
\end{align}
where we have used the following inverse inequality for negative norms: From the stability of the $L^2$ projection operator $P_h:L^2(\Ohsm)\rightarrow S_h(\Ohsm)$,
\begin{align}\label{eq:inv-H-1/2}
	&\| \psi_h \|_{L^2(\Ohsm)}
	=
	\sup_{0\neq\phi \in L^2}
	\frac{\big| \int_\Ohsm \psi_h \phi \big|}{\| \phi \|_{L^2(\Ohsm)}}
	\leq
	C_{\mu_m}
	\sup_{0\neq\phi \in L^2}
	\frac{\big| \int_\Ohsm \psi_h P_h\phi \big|}{\| P_h \phi \|_{L^2(\Ohsm)}}
	\notag\\
	&
	\leq
	C_{\mu_m}
	\sup_{0\neq\phi_h\in S_h}
	\frac{\big| \int_\Ohsm \psi_h \phi_h \big|}{\| \phi_h \|_{L^2(\Ohsm)}}
	\leq
	C_{\mu_m}
	\sup_{0\neq\phi_h\in S_h}
	\frac{\| \psi_h \|_{H_0^{-1/2}(\Ohsm)} \| \phi_h \|_{L^{2}(\Ohsm)}^{1/2} \| \phi_h \|_{H^{1}(\Ohsm)}^{1/2}}{\| \phi_h \|_{L^2(\Ohsm)}}
	\notag\\
	&
	\leq
	C_{\mu_m} h^{-1/2} \| \psi_h \|_{H_0^{-1/2}(\Ohsm)} .
\end{align}

\section{Stability analysis for numerical errors}
\label{sec:stab-anal}

\subsection{Errors and induction hypothesis}\label{sec:err-ind-hypo}


The relevant errors are defined as follows:
\begin{equation}\label{eq:err-def}
	\begin{aligned}
		e_u^m
		&= \uhm - I_h(u^m\circ \Phi^m),
		&
		e_p^m
		&= \phm - \bar I_h(p^m\circ \Phi^m),
		\\
		\hat e_x^m
		&= X_h^m - \hat X_{h,*}^m,
		&
		e_x^m
		&= X_h^m - X_{h,*}^m,
		\\
		e_{x,\#}^m
		&= X_h^m - I_h(X^m\circ \Phi^0),
		&
		e_\kappa^m
		&= \khm - I_h(H^mn^m)^{-\ell},
	\end{aligned}
\end{equation}
whose quantitative behavior will be the main focus of this section.


At time $t_m$, we impose the induction hypothesis
\begin{align}\label{eq:ind-hypo}
	&
	h^{-1} \| E_h \ehxm \|_{W^{1,\infty}(\Ohsf)}
	+
	h^{-3/2}\| E_h \ehxm \|_{H^{1}(\Ohsf)}
	+
	h^{-1} \| \ehxm \|_{W^{1,\infty}(\Ghsf)}
	\notag\\
	&\quad
	+
	h^{-3/2}\| \ehxm \|_{H^{1}(\Ghsf)}
	+
	h^{-3/2}\| e_{x,\#}^m \|_{H^{1}(\Ohsf)}
	\leq 1 .
\end{align}
Note that, since $\uhm$ and $\phm$ are not yet known at time $t_m$, we cannot perform induction on them, nor on $\eum$ and $\epm$.
According to \cite[Lemma 4.3]{KLL17} and \cite[Lemma 7.2]{KLL19}, Eq. \eqref{eq:ind-hypo} implies the equivalence of $L^q$ and $W^{1,q}$ norms (up to a constant $C$ for sufficiently small $h\leq h_{\mum,\num}$), for all $1\le q\le \infty$, of finite element functions with a common nodal vector on the linearly interpolated families
\begin{align}
	\hat\Omega_{h,\theta}^m&=(1 - \theta)\Ohsm + \theta\Ohm ,\quad  \theta\in [0, 1] ,
	\notag\\
	\hat\Gamma_{h,\theta}^m&=(1 - \theta)\Ghsm + \theta\Ghm ,\quad  \theta\in [0, 1] . \notag
\end{align}
From the real interpolation theory (cf. \cite[Proposition 2.4.3]{SS10book}, \cite[Theorem 1.6]{Lunardi2018book} and \cite[Chapter 1]{LM2012book}), such norm equivalence holds for the $H^{1/2}$ norm. Note that the space $H_0^{1/2}$ is equipped with the same norm. Moreover, by duality, the $H^{-1}$, $H_0^{-1}$, $H^{-1/2}$ and $H_0^{-1/2}$ norms are also equivalent on the linearly interpolated families.
We aim to prove the error estimates in Theorem \ref{thm:main} at the next time level $t_{m+1}$ with the aid of \eqref{eq:ind-hypo}.
To close the proof, we shall recover the induction hypothesis \eqref{eq:ind-hypo} at $t_{m+1}$.
%

In the rest of this paper, we use $C_{\mum,\num}$ as a generic positive constant which may be different at different occurrences, possibly dependent on $T$, $\mu_m$ and $\nu_m$ ($m$ is the current time level and can be easily read off from the context), but is independent of $\tau$, $h$, $\musm$, and $\nusm$. 
We use the notation $A \lesssim B$ to denote the relation ``$A\le C_{\mum,\num} B$ for some constant $C_{\mum,\num}$''. If $A\lesssim B$ and $B\lesssim A$ at the same time, then we use the notation $A\sim B$.
Besides, we denote by $C$ or $C_0$ another generic positive constant which is independent of $\tau$, $h$, $\mu_m$, $\musm$, $\nu_m$ and $\nusm$.

\subsection{Error equations}\label{sec:err-eq}
To derive the error equation, we subtract the consistency equations \eqref{subeq:2phase-hs} from the numerical scheme \eqref{subeq:2phase-h}. 
For the variable $u$, the corresponding error equation is
\begin{align}\label{eq:eu}
	-d_u^m(\phi_h)
	&=
	\int_{\Ohm} 2\nu \D_\Ohm \uhm \cdot \D_\Ohm \phi_h
	-
	\int_{\Ohsm} 2\nu \D_\Ohsm {I_h}(u^m\circ\Phi^m) \cdot \D_\Ohsm \phi_h
	\notag\\
	&\quad
	- 
	\int_\Ohm \phm \nabla_\Ohm \cdot \phi_h
	+
	\int_\Ohsm \bar I_h (p^m\circ\Phi^m) \nabla_\Ohsm \cdot \phi_h
	\notag\\
	&\quad
	+
	\int_\Ghm \khm \cdot \phi_h
	-
	\int_\Ghsm I_h (H^m n^m)^{-\ell} \cdot \phi_h
	\notag\\
	&=
	\int_{\Ohsm} 2\nu \D_\Ohsm \eum \cdot \D_\Ohsm \phi_h
	- 
	\int_\Ohsm \epm \nabla_\Ohsm \cdot \phi_h
	+
	\int_\Ghsm \ekm \cdot \phi_h
	+
	J_u^{m}(\phi_h),
\end{align}
with
\begin{align}
	J_u^{m}(\phi_h)
	&=
	\int_{\Ohm} 2\nu \D_\Ohm \uhm \cdot \D_\Ohm \phi_h
	-
	\int_{\Ohsm} 2\nu \D_\Ohsm \uhm \cdot \D_\Ohsm \phi_h
	\notag\\
	&\quad
	- 
	\int_\Ohm \phm \nabla_\Ohm \cdot \phi_h
	+
	\int_\Ohsm \phm \nabla_\Ohsm \cdot \phi_h
	\notag\\
	&\quad
	+
	\int_\Ghm \khm \cdot \phi_h
	-
	\int_\Ghsm \khm \cdot \phi_h
	\notag\\
	&=:
	J_u^{m,1}(\phi_h)
	+
	J_u^{m,2}(\phi_h)
	+
	J_u^{m,3}(\phi_h) .
	\notag
\end{align}
Similarly, we have
\begin{align}\label{eq:ex}
	-d_x^m(\chi_h)
	&=
	\int_\Ghm \frac{\XhM - \Xhm}{\tau} \cdot \chi_h
	-
	\int_\Ghsm \frac{X_{h,*}^{m+1} - \hat X_{h,*}^m}{\tau} \cdot \chi_h
	\notag\\
	&\quad
	-
	\int_\Ghm \uhm \cdot \chi_h 
	+
	\int_\Ghsm I_h (u^m|_\Gm)^{-\ell} \cdot \chi_h 
	\notag\\
	&=
	\int_\Ghsm \frac{\exM -\ehxm}{\tau} \cdot \chi_h
	-
	\int_\Ghsm \eum \cdot \chi_h 
	+
	J_x^m(\chi_h) ,
\end{align}
with
\begin{align}
	J_x^m(\chi_h)
	&=
	\int_\Ghm \frac{\XhM - \Xhm}{\tau} \cdot \chi_h
	-
	\int_\Ghsm \frac{\XhM - \Xhm}{\tau} \cdot \chi_h
	\notag\\
	&\quad
	-
	\int_\Ghm \uhm \cdot \chi_h 
	+
	\int_\Ghsm \uhm \cdot \chi_h 
	\notag\\
	&=:
	J_x^{m,1}(\chi_h)
	+
	J_x^{m,2}(\chi_h)
	\notag .
\end{align}
Finally, the error equation for $\kappa$ reads
\begin{align}\label{eq:ek}
	-d_\kappa^m(\eta_h)
	&=
	-\int_\Ghm \khm \cdot \eta_h
	+
	\int_\Ghsm I_h (H^m n^m)^{-\ell} \cdot \eta_h
	\notag\\
	&\quad
	+
	\int_\Ghm \nabla_\Ghm \XhM \cdot \nabla_\Ghm \eta_h
	-
	\int_\Ghsm \nabla_\Ghsm X_{h,*}^{m+1} \cdot \nabla_\Ghsm \eta_h 
	\notag\\
	&=:
	-
	\int_\Ghsm \ekm \cdot \eta_h
	+
	\int_\Ghsm \nabla_\Ghsm \exM \cdot \nabla_\Ghsm \eta_h 
	+
	J_\kappa^{m}(\eta_h) ,
\end{align}
with
\begin{align}
	J_\kappa^{m}(\eta_h)
	&=
	-\int_\Ghm \khm \cdot \eta_h
	+
	\int_\Ghsm \khm \cdot \eta_h
	\notag\\
	&\quad
	+
	\int_\Ghm \nabla_\Ghm \XhM \cdot \nabla_\Ghm \eta_h
	-
	\int_\Ghsm \nabla_\Ghsm X_{h}^{m+1} \cdot \nabla_\Ghsm \eta_h 
	\notag\\
	&=:
	J_\kappa^{m,1}(\eta_h)
	+
	J_\kappa^{m,2}(\eta_h) .
	\notag
\end{align}

\subsection{Bilinear error estimates}\label{sec:bilinear-err-est}

%
%
We first record a new bilinear error estimate for the mass bilinear form whose upper bound does not involve any gradients due to the intrinsic orthogonality of the projection error $\ehxm$. This result significantly improves the known standard results (cf. \cite{Kov18,KLL19}).
The absence of gradient in the upper bound is of vital importance to conclude the stability of numerical errors (cf. \eqref{eq:Jxm}).
\begin{lemma}\label{lemma:e-blinear}
	Assuming the induction hypothesis \eqref{eq:ind-hypo}, for all finite element functions $f_h,g_h\in S_h(\Ghsm)$, it holds that
	\begin{align*}
		\Big|\int_{\Ghm} f_h g_h - \int_{\Ghsm} f_h g_h \Big| 
		&\lesssim
		\|  \ehxm \|_{L^p(\Ghsm)} \|f_h\|_{ L^q(\Ghsm)}
		\|g_h\|_{L^r(\Ghsm)}
		,
	\end{align*}
	for all $1/p+1/q+1/r=1$.
\end{lemma}
\begin{proof}
	Define the linearly parametrized intermediate surface $\hat \Gamma_{h,*}^{m,\theta}:= \theta \Ghm + (1-\theta) \Ghsm$. Then using Lemma \ref{lemma:ud}, we derive
	\begin{align}\label{eq:Q-decomp}
		&\int_{\Ghm} f_h g_h - \int_{\Ghsm} f_h g_h
		=
		\int_0^1\frac{\d}{\d\theta}\bigg(\int_{\Gamma_{h,*}^{m,\theta}} f_h g_h \bigg)\d\theta
		\notag\\
		&=
		\int_0^1 \int_{\Gamma_{h,*}^{m,\theta}} f_h g_h (\nabla_{\Gamma_{h,*}^{m,\theta}}\cdot\ehxm) \d\theta
		\notag\\
		&=
		\int_{\Gm} f_h^\ell g_h^\ell (\nabla_\Gm\cdot(\ehxm)^\ell)
		\notag\\
		&\quad
		+
		\int_{\Ghsm} f_h g_h (\nabla_\Ghsm\cdot\ehxm)
		-
		\int_{\Gm} f_h^\ell g_h^\ell (\nabla_\Gm\cdot(\ehxm)^\ell)
		\notag\\
		&\quad
		+
		\int_0^1 \bigg(\int_{\Gamma_{h,*}^{m,\theta}} f_h g_h (\nabla_{\Gamma_{h,*}^{m,\theta}}\cdot\ehxm)
		-
		\int_{\Ghsm} f_h g_h (\nabla_\Ghsm\cdot\ehxm)
		\bigg) \d\theta
		\notag\\
		&=:
		Q_1^m + Q_2^m + Q_3^m.
	\end{align}
	For $Q_1^m$,
	\begin{align}\label{eq:Q1}
		|Q_1^m|
		&=
		\bigg|\int_{\Gm} f_h^\ell g_h^\ell \nabla_\Gm\cdot(\Nsm(\ehxm)^\ell)
		+
		\int_{\Gm} f_h^\ell g_h^\ell \nabla_\Gm\cdot(\Tsm(\ehxm)^\ell)
		\bigg|
		\notag\\
		&=
		\bigg|
		\int_{\Gm} f_h^\ell g_h^\ell (\nabla_\Gm\cdot\Nsm)\cdot(\ehxm)^\ell
		+
		\int_{\Gm} f_h^\ell g_h^\ell \nabla_\Gm\cdot(\Tsm(\ehxm)^\ell)
		\bigg|
		\notag\\
		&\lesssim
		(\|  \ehxm \|_{L^p(\Ghsm)} + h \|  \nabla_\Ghsm \ehxm \|_{L^p(\Ghsm)}) \|f_h\|_{ L^q(\Ghsm)}
		\|g_h\|_{L^r(\Ghsm)}
		,
	\end{align}
	where we have used the orthogonality $\Nsm \cdot \nabla_\Gm = 0$ and the super-approximation (cf. Lemma \ref{lemma:T<=N2}).
	
	Note that, for any given integrability triplet $(p,q,r)$ satisfying $1/p + 1/q + 1/r = 1$, there exists another triplet $(\tilde p,\tilde q,\tilde r)$ satisfying $\tilde p \geq p, \tilde q \geq q, \tilde r \geq r$, and $1/ \tilde p + 1/ \tilde q + 1/ \tilde r = 1/2$. 
	Applying the consistency error estimate (Lemma \ref{lemma:geo-perturb}) and using the inverse inequality to convert $(\tilde p,\tilde q,\tilde r)$ to $(p,q,r)$, we get
	\begin{align}\label{eq:Q2}
		|Q_2^m|
		&\lesssim
		\nusm h^{k}
		\| \nabla_\Ghsm \ehxm \|_{L^{\tilde p}(\Ghsm)} \|f_h\|_{ L^{\tilde q}(\Ghsm)}
		\|g_h\|_{L^{\tilde r}(\Ghsm)}
		\notag\\
		&\lesssim
		\nusm h^{k+1/2-d/2}
		\| \nabla_\Ghsm \ehxm \|_{L^p(\Ghsm)} \|f_h\|_{ L^q(\Ghsm)}
		\|g_h\|_{L^r(\Ghsm)}
		.
	\end{align}
	Using the fundamental theorem of calculus again (cf. \eqref{eq:Q-decomp}),
	\begin{align}\label{eq:Q3}
		|Q_3^m|
		&\lesssim
		\| \nabla_\Ghsm \ehxm \|_{L^\infty(\Ghsm)}
		\| \nabla_\Ghsm \ehxm \|_{L^p(\Ghsm)} \|f_h\|_{ L^q(\Ghsm)}
		\|g_h\|_{L^r(\Ghsm)} 
		.
	\end{align}
	The proof is complete by applying the induction hypothesis \eqref{eq:ind-hypo} and the inverse inequality to \eqref{eq:Q1}--\eqref{eq:Q3} for sufficiently small $h\leq h_\nusm$.
\end{proof}
Following a similar proof, we can show
\begin{align}\label{lemma:e-blinear-coro}
	\Big|\int_{\Ghm} f_h - \int_{\Ghsm} f_h \Big| 
	&\lesssim
	\|  \ehxm \|_{H^1(\Ghsm)} \|f_h\|_{ H^{-1}(\Ghsm)}
	\quad\forall f_h\in S_h(\Ghsm)
	.
\end{align}

Thanks to the construction of $\Ohsm$ in Section \ref{sec:bulk-mesh} (see \eqref{eq:Ohsm-def}), the errors of bilinear forms defined in the bulk domains are also stable.
According to the pointwise identity in Lemma \ref{lemma:udb}, we have the following lemma.
\begin{lemma}\label{lemma:e-blinear-b}
	Assuming the induction hypothesis \eqref{eq:ind-hypo},
	for all finite element functions $f_h,g_h\in S_h(\Ohsm)$ and $1/p+1/q+1/r=1$, it holds that
	\begin{align*}
		\Big|\int_{\Ohm} f_h g_h - \int_{\Ohsm} f_h g_h \Big| 
		&\lesssim
		\| \nabla_\Ohsm E_h \ehxm \|_{L^p(\Ohsm)} \|f_h\|_{ L^q(\Ohsm)}
		\|g_h\|_{L^r(\Ohsm)} ,
		\\
		\Big|\int_{\Ohm} \nabla_\Ohm f_h g_h - \int_{\Ohsm} \nabla_\Ohsm f_h g_h \Big| 
		&\lesssim
		\| \nabla_\Ohsm E_h \ehxm \|_{L^p(\Ohsm)} \| \nabla_\Ohsm f_h\|_{ L^q(\Ohsm)}
		\|g_h\|_{L^r(\Ohsm)}
		,
	\end{align*}
	and
	\begin{align*}
		&\Big|\int_{\Ohm}\D_{\Ohm} f_h\cdot\D_{\Ohm} g_h - \int_{\Ohsm}\D_{\Ohsm} f_h \cdot \D_{\Ohsm} g_h \Big| 
		\notag\\
		&
		\qquad
		\lesssim
		\| \nabla_\Ohsm E_h \ehxm \|_{L^p(\Ohsm)} \| \nabla_\Ohsm f_h\|_{ L^q(\Ohsm)}
		\| \nabla_\Ohsm g_h\|_{L^r(\Ohsm)}
		.
	\end{align*}
\end{lemma} 

\subsection{Estimates for $J^m$s}\label{sec:Jm}

This subsection is devoted to estimating the $J^m$-terms using the bilinear error estimates developed in Section \ref{sec:bilinear-err-est}.
\begin{lemma}\label{lemma:Jm}
	Applying the above two lemmas to
	the $J^m$s defined in \eqref{eq:eu}--\eqref{eq:ek}, we immediately get
	\begin{align}
		| J_u^m (\phi_h)|
		&\lesssim
		\| \nabla_\Ohsm E_h \ehxm \|_{L^2(\Ohsm)} \| \nabla_\Ohsm \phi_h \|_{L^2(\Ohsm)} 
		+
		\| \ehxm \|_{L^2(\Ghsm)} \| \phi_h \|_{L^2(\Ghsm)} 
		\notag\\
		&\quad
		+
		\| \nabla_\Ohsm E_h \ehxm \|_{L^\infty(\Ohsm)}
		(\| \nabla_\Ohsm \eum \|_{L^2(\Ohsm)} + \| \epm \|_{L^2(\Ohsm)}) 
		\notag\\
		&\qquad\times
		\| \nabla_\Ohsm \phi_h \|_{L^2(\Ohsm)} 
		\notag\\
		&\quad
		+
		\| \ehxm \|_{L^\infty(\Ghsm)}
		\| \ekm \|_{L^2(\Ghsm)} \| \phi_h \|_{L^2(\Ghsm)} 
		,
		\label{eq:Jum}\\
		| J_x^m (\chi_h)|
		&\lesssim
		\min\Big\{\|  \ehxm \|_{L^2(\Ghsm)} \| \chi_h \|_{L^2(\Ghsm)},
		\|  \ehxm \|_{H^1(\Ghsm)} \| \chi_h \|_{H^{-1}(\Ghsm)} 
		\Big\}
		\notag\\
		&\quad
		+
		\|  \ehxm \|_{L^\infty(\Ghsm)}
		\Big(\| \eum \|_{L^2(\Ghsm)} + \Big\| \frac{\exM-\ehxm}{\tau} \Big\|_{L^2(\Ghsm)}\Big) \| \chi_h \|_{L^2(\Ghsm)}
		\label{eq:Jxm} , \\
		| J_\kappa^m (\eta_h)|
		&\lesssim
		\big(1 + \tau + \| \nabla_\Ghsm (\exM - \ehxm) \|_{L^\infty(\Ghsm)}\big)
		\notag\\
		&\quad
		\times
		\min\Big\{\| \ehxm \|_{L^2(\Ghsm)} \| \nabla_\Ghsm \eta_h \|_{L^2(\Ghsm)}, 
		\| \nabla_\Ghsm \ehxm \|_{L^2(\Ghsm)} \| \eta_h \|_{L^2(\Ghsm)}
		\Big\}
		\notag\\
		&\quad
		+
		\| \ehxm \|_{L^2(\Ghsm)} \| \eta_h \|_{L^2(\Ghsm)} 
		+
		\| \ehxm \|_{L^\infty(\Ghsm)}
		\| \ekm \|_{L^2(\Ghsm)} \| \eta_h \|_{L^2(\Ghsm)} 
		\label{eq:Jkm}
		.
	\end{align}
\end{lemma}
\begin{proof}
	Using Lemma \ref{lemma:e-blinear-b}, we first derive as follows:
	\begin{align}\label{eq:Jum1}
		&|J_u^{m,1}(\phi_h)|
		=
		\bigg|\int_{\Ohm} \D_\Ohm \uhm \cdot \D_\Ohm \phi_h
		-
		\int_{\Ohsm} \D_\Ohsm \uhm \cdot \D_\Ohsm \phi_h \bigg|
		\notag\\
		&\leq	
		\bigg|\int_{\Ohm} \D_\Ohm \eum \cdot \D_\Ohm \phi_h
		-
		\int_{\Ohsm} \D_\Ohsm \eum \cdot \D_\Ohsm \phi_h \bigg|
		\notag\\
		&\quad
		+
		\bigg|\int_{\Ohm} \D_\Ohm I_h(u^m\circ\Phi^m) \cdot \D_\Ohm \phi_h
		-
		\int_{\Ohsm} \D_\Ohsm I_h(u^m\circ\Phi^m) \cdot \D_\Ohsm \phi_h \bigg|
		\notag\\
		&\lesssim
		\| \nabla_\Ohsm E_h \ehxm \|_{L^{\infty}(\Ohsm)} \| \nabla_\Ohsm \eum \|_{L^{2}(\Ohsm)} \| \nabla_\Ohsm \phi_h \|_{L^{2}(\Ohsm)}
		\notag\\
		&\quad
		+
		\| \nabla_\Ohsm E_h \ehxm \|_{L^{2}(\Ohsm)} 
		\| \nabla_\Ohsm \phi_h \|_{L^{2}(\Ohsm)} ,
	\end{align}
	where we have used the boundedness of $\| \nabla_\Ohsm I_h(u^m\circ\Phi^m) \|_{L^{\infty}(\Ohsm)}$.
	
	
	Similarly, 
	\begin{align}\label{eq:Jum2}
		&| J_u^{m,2}(\phi_h) |
		=
		\bigg|
		\int_\Ohm \phm \nabla_\Ohm \cdot \phi_h
		-
		\int_\Ohsm \phm \nabla_\Ohsm \cdot \phi_h
		\bigg|
		\notag\\
		&\leq
		\bigg|
		\int_\Ohm \epm \nabla_\Ohm \cdot \phi_h
		-
		\int_\Ohsm \epm \nabla_\Ohsm \cdot \phi_h
		\bigg|
		\notag\\
		&\quad
		+
		\bigg|	
		\int_\Ohm \bar I_h(p^m\circ\Phi^m) \nabla_\Ohm \cdot \phi_h
		-
		\int_\Ohsm \bar I_h(p^m\circ\Phi^m) \nabla_\Ohsm \cdot \phi_h
		\bigg|
		\notag\\
		&\lesssim
		\| \nabla_\Ohsm E_h \ehxm \|_{L^{\infty}(\Ohsm)} \| \epm \|_{L^{2}(\Ohsm)} \| \nabla_\Ohsm \phi_h \|_{L^{2}(\Ohsm)}
		\notag\\
		&\quad
		+
		\| \nabla_\Ohsm E_h \ehxm \|_{L^{2}(\Ohsm)} \| \nabla_\Ohsm \phi_h \|_{L^{2}(\Ohsm)} .
	\end{align}
	From Lemma \ref{lemma:e-blinear},
	\begin{align}\label{eq:Jum3}
		&|J_u^{m,3}(\phi_h)|
		=
		\bigg|
		\int_\Ghm \khm \cdot \phi_h
		-
		\int_\Ghsm \khm \cdot \phi_h
		\bigg|
		\notag\\
		&\leq
		\bigg|
		\int_\Ghm \ekm \cdot \phi_h
		-
		\int_\Ghsm \ekm \cdot \phi_h
		\bigg|
		+
		\bigg|
		\int_\Ghm I_h(\Hm\nm) \cdot \phi_h
		-
		\int_\Ghsm I_h(\Hm\nm) \cdot \phi_h
		\bigg|
		\notag\\
		&\lesssim
		\|  \ehxm \|_{L^\infty(\Ghsm)} \| \ekm \|_{ L^2(\Ghsm)}\|  \phi_h \|_{L^2(\Ghsm)}
		+
		\|  \ehxm \|_{L^2(\Ghsm)} \|  \phi_h \|_{L^2(\Ghsm)} .
	\end{align}
	We conclude \eqref{eq:Jum} from \eqref{eq:Jum1}--\eqref{eq:Jum3}.
	
	
	Analogous to $J_u^{m,3}$, using \eqref{eq:geo_rel_3} and \eqref{lemma:e-blinear-coro}, we have
	\begin{align}\label{eq:Jxm1}
		&|J_x^{m,1}(\chi_h)|
		=
		\bigg|
		\int_\Ghm \frac{\XhM-\Xhm}{\tau} \cdot \chi_h
		-
		\int_\Ghsm \frac{\XhM-\Xhm}{\tau} \cdot \chi_h
		\bigg|
		\notag\\
		&\leq
		\bigg|
		\int_\Ghm \frac{\exM-\ehxm}{\tau} \cdot \chi_h
		-
		\int_\Ghsm \frac{\exM-\ehxm}{\tau} \cdot \chi_h
		\bigg|
		\notag\\
		&\quad
		+
		\bigg|
		\int_\Ghm I_h(u^m+g^m)^{-\ell} \cdot \chi_h
		-
		\int_\Ghsm I_h(u^m+g^m)^{-\ell} \cdot \chi_h
		\bigg|
		\notag\\
		&\lesssim
		\|  \ehxm \|_{L^\infty(\Ghsm)} \Big\| \frac{\exM-\ehxm}{\tau} \Big\|_{ L^2(\Ghsm)}\|  \chi_h \|_{L^2(\Ghsm)}
		\notag\\
		&\quad
		+
		\min\Big\{\|  \ehxm \|_{L^2(\Ghsm)} \| \chi_h \|_{L^2(\Ghsm)},
		\|  \ehxm \|_{H^1(\Ghsm)} \| \chi_h \|_{H^{-1}(\Ghsm)} 
		\Big\} 
		,
	\end{align}
	and
	\begin{align}\label{eq:Jxm2}
		&|J_x^{m,2}(\chi_h)|
		=
		\bigg|
		\int_\Ghm \uhm \cdot \chi_h
		-
		\int_\Ghsm \uhm \cdot \chi_h
		\bigg|
		\notag\\
		&\leq
		\bigg|
		\int_\Ghm \eum \cdot \chi_h
		-
		\int_\Ghsm \eum \cdot \chi_h
		\bigg|
		+
		\bigg|
		\int_\Ghm I_h (u^m)^{-\ell} \cdot \chi_h
		-
		\int_\Ghsm I_h (u^m)^{-\ell} \cdot \chi_h
		\bigg|
		\notag\\
		&\lesssim
		\|  \ehxm \|_{L^\infty(\Ghsm)} \| \eum \|_{ L^2(\Ghsm)}\|  \chi_h \|_{L^2(\Ghsm)}
		\notag\\
		&\quad
		+
		\min\Big\{\|  \ehxm \|_{L^2(\Ghsm)} \| \chi_h \|_{L^2(\Ghsm)},
		\|  \ehxm \|_{H^1(\Ghsm)} \| \chi_h \|_{H^{-1}(\Ghsm)} 
		\Big\} .
	\end{align}
	Estimates \eqref{eq:Jxm1}--\eqref{eq:Jxm2} lead to \eqref{eq:Jxm}.
	
	Finally, $J_\kappa^{m,1}$ has the same structure as $J_u^{m,3}$ and therefore following \eqref{eq:Jum3}
	\begin{align}\label{eq:Jkm1}
		|J_\kappa^{m,1}(\eta_h)|
		\lesssim
		\|  \ehxm \|_{L^\infty(\Ghsm)} \| \ekm \|_{ L^2(\Ghsm)}\|  \eta_h \|_{L^2(\Ghsm)}
		+
		\|  \ehxm \|_{L^2(\Ghsm)} \|  \eta_h \|_{L^2(\Ghsm)},
	\end{align}
	and $J_\kappa^{m,2}$ happens to be the domain discrepancy error of the interface stiffness bilinear form which is the central object in the numerical analysis for mean curvature flow and surface diffusion and has been well studied in \cite{BL24FOCM,BL22A,BL24}. From \cite[Eqs. (5.16), (5.17) and (5.22)]{BL24FOCM} together with the induction hypothesis \eqref{eq:ind-hypo} and the inverse inequality,
	\begin{align}\label{eq:Jkm2}
		&| J_\kappa^{m,2} (\eta_h)|
		\lesssim
		\big(1 + \tau + \| \nabla_\Ghsm (\exM - \ehxm) \|_{L^\infty(\Ghsm)}\big)
		\notag\\
		&\qquad
		\times
		\min\Big\{\| \ehxm \|_{L^2(\Ghsm)} \| \nabla_\Ghsm \eta_h \|_{L^2(\Ghsm)}, 
		\| \nabla_\Ghsm \ehxm \|_{L^2(\Ghsm)} \| \eta_h \|_{L^2(\Ghsm)}
		\Big\} .
	\end{align}
	Thus, \eqref{eq:Jkm1}--\eqref{eq:Jkm2} imply \eqref{eq:Jkm}.
	
	The proof of Lemma \ref{lemma:Jm} is complete.
\end{proof}
$J_u^m(\cdot)$, $J_x^m(\cdot)$ and $J_\kappa^m(\cdot)$ can be viewed as linear functionals on $S_h(\Ohsm)^d$, $S_h(\Ghsm)^d$ and $S_h(\Ghsm)^d$, respectively.
According to the main discrete parabolic estimate \eqref{eq:parab}, the relevant norms for the $J^m$s are
$$\| J_u^m(\cdot) \|_{H^{-1}(\Ohsm)},\quad \| J_x^m(\cdot) \|_{H^{1}(\Ghsm)},\quad \mbox{and}\quad \| J_\kappa^m(\cdot) \|_{H^{-1/2}(\Ghsm)} .$$
These norm estimates can be obtained in a manner similar to \eqref{eq:dpm-H1/2}. For instance,
\begin{align}
	&\| J_x^m(\cdot) \|_{H^{1}(\Ghsm)}
	\lesssim
	\|  \ehxm \|_{H^1(\Ghsm)}
	+
	\|  \ehxm \|_{L^\infty(\Ghsm)}
	\notag\\
	&\qquad\times
	\Big(\| \eum \|_{L^2(\Ghsm)} + \Big\| \frac{\exM-\ehxm}{\tau} \Big\|_{L^2(\Ghsm)}\Big) 
	\sup_{0\neq\chi_h\in S_h} \frac{\| \chi_h \|_{L^2(\Ghsm)}}{\| \chi_h \|_{H^{-1}(\Ghsm)}}
	\notag\\
	&\qquad\lesssim
	\|  \ehxm \|_{H^1(\Ghsm)}
	+
	h^{-1}
	\|  \ehxm \|_{L^\infty(\Ghsm)}
	\Big(\| \eum \|_{L^2(\Ghsm)} + \Big\| \frac{\exM-\ehxm}{\tau} \Big\|_{L^2(\Ghsm)}\Big) .
	\notag
\end{align}
Here we have used the following inverse inequality
\begin{align}
	\| \chi_h \|_{L^2(\Ghsm)}
	\lesssim
	h^{-1}
	\| \chi_h \|_{H^{-1}(\Ghsm)} , \notag
\end{align}
whose derivation is similar to \eqref{eq:inv-H-1/2}.
The right-hand side of the estimate for $\| J_x^m(\cdot) \|_{H^{1}(\Ghsm)}$ consists of a stable term $\|  \ehxm \|_{H^1(\Ghsm)}$ and a product of errors that represents a small perturbation. Hence, $J_x^m$ is harmless. Similar arguments apply to $J_u^m$ and $J_\kappa^m$ as well.

\subsection{Error estimates on the moving interface}

In view of the error equations for $x$ and $\kappa$ (i.e., \eqref{eq:ex} and \eqref{eq:ek}, respectively), the standard parabolic (for variable $x$) and elliptic (for variable $\kappa$) error estimates immediately give
\begin{align}
	\Big\| \frac{\exM-\ehxm}{\tau} \Big\|_{L^{2}(\Ghsm)}
	&\lesssim
	\| \eum \|_{L^{2}(\Ghsm)}
	+
	\| J_x^m(\cdot) \|_{L^{2}(\Ghsm)}
	+
	\| d_x^m(\cdot) \|_{L^{2}(\Ghsm)}
	\label{eq:ev-est} , \\
	\Big\| \frac{\exM-\ehxm}{\tau} \Big\|_{H^{1/2}(\Ghsm)}
	&\lesssim
	\| \eum \|_{H^{1/2}(\Ghsm)}
	+
	\| J_x^m(\cdot) \|_{H^{1/2}(\Ghsm)}
	+
	\| d_x^m(\cdot) \|_{H^{1/2}(\Ghsm)}
	\label{eq:ev-est1} , \\
	\Big\| \frac{\exM-\ehxm}{\tau} \Big\|_{H^{1}(\Ghsm)}
	&\lesssim
	\| \eum \|_{H^{1}(\Ghsm)}
	+
	\| J_x^m(\cdot) \|_{H^{1}(\Ghsm)}
	+
	\| d_x^m(\cdot) \|_{H^{1}(\Ghsm)}
	\label{eq:ev-est2} , \\
	\frac{\| \exM \|_{L^2(\Ghsm)}^2 - \| \ehxm \|_{L^2(\Ghsm)}^2}{\tau}
	&\lesssim
	\epsilon
	\| \eum \|_{L^{2}(\Ghsm)}^2
	+
	\epsilon^{-1}
	\| \exM \|_{L^{2}(\Ghsm)}^2
	\notag\\
	&\qquad\qquad
	+
	\| J_x^m(\cdot) \|_{L^{2}(\Ghsm)}^2
	+
	\| d_x^m(\cdot) \|_{L^{2}(\Ghsm)}^2
	\label{eq:ex-est-L2} , \\
	\| \ekm \|_{H^{-1}(\Ghsm)}
	&\lesssim
	\| \exM \|_{H^{1}(\Ghsm)}
	+
	\| J_\kappa^m(\cdot) \|_{H^{-1}(\Ghsm)}
	+
	\| d_\kappa^m(\cdot) \|_{H^{-1}(\Ghsm)}.
	\label{eq:ek-est}
\end{align}
Substituting the estimates for $J^m s$ (Lemma \ref{lemma:Jm}) into the above estimates, after simplification we obtain 
\begin{align}
	\Big\| \frac{\exM-\ehxm}{\tau} \Big\|_{L^{2}(\Ghsm)}
	&\lesssim
	\| \eum \|_{L^{2}(\Ghsm)}
	+
	\| \ehxm \|_{L^{2}(\Ghsm)}
	+
	\| d_x^m(\cdot) \|_{L^{2}(\Ghsm)}
	\label{eq:ev-est-s} , \\
	\Big\| \frac{\exM-\ehxm}{\tau} \Big\|_{H^{1/2}(\Ghsm)}
	&\lesssim
	\| \eum \|_{H^{1/2}(\Ghsm)}
	+
	\| \ehxm \|_{H^{1}(\Ghsm)}
	+
	\| d_x^m(\cdot) \|_{H^{1/2}(\Ghsm)}
	\label{eq:ev-est1-s} , \\
	\Big\| \frac{\exM-\ehxm}{\tau} \Big\|_{H^{1}(\Ghsm)}
	&\lesssim
	\| \eum \|_{H^{1}(\Ghsm)}
	+
	\| \ehxm \|_{H^{1}(\Ghsm)}
	+
	\| d_x^m(\cdot) \|_{H^{1}(\Ghsm)}
	\label{eq:ev-est2-s} , \\
	\frac{\| \exM \|_{L^2(\Ghsm)}^2 - \| \ehxm \|_{L^2(\Ghsm)}^2}{\tau}
	&\lesssim
	\epsilon
	\| \eum \|_{L^{2}(\Ghsm)}^2
	+
	\epsilon^{-1}
	\| \ehxm \|_{L^{2}(\Ghsm)}^2
	+
	\| d_x^m(\cdot) \|_{L^{2}(\Ghsm)}^2
	\label{eq:ex-est-L2-s} , \\
	\| \ekm \|_{H^{-1}(\Ghsm)}
	&\lesssim
	\| \ehxm \|_{H^{1}(\Ghsm)}
	+
	\| d_\kappa^m(\cdot) \|_{H^{-1}(\Ghsm)}
	\notag\\
	&\qquad
	+
	\tau(
	\| \eum \|_{H^{1}(\Ghsm)}
	+
	\| d_x^m(\cdot) \|_{H^1(\Ghsm)}
	)
	\label{eq:ek-est-s} .
\end{align}

\subsection{$L^2$ and $H^{-1/2}$ error estimates for pressure}\label{sec:epm}

We define the averaging operator $\bar\phi := \phi - |\Ohsm|^{-1} \int_\Ohsm \phi$, which is bounded from $W^{s, q}(\Ohsm)$ to $W^{s, q}(\Ohsm) \cap L_0^q(\Ohsm)$ for any $s\geq 0$ and $q\in [1,\infty]$. 
Simple consequences are $(\phi, \bar \psi)_\Ohsm = (\bar\phi, \bar \psi)_\Ohsm = (\bar\phi, \psi)_\Ohsm$ and, if $\epm$ and $\phm$ are identified as finite element functions on $\Ohsm$, then $\epm - \bar e_p^m = p_h^m - \bar p_h^m = |\Ohsm|^{-1} \int_\Ohsm \phm$. Note that we also have $\int_\Ohm \phm = 0$, since $\phm\in Q_h^{k-1}(\Ohm)$.

Since the discrete inf-sup (or the discrete Ladyzhenskaya–Babuška–Brezzi (LBB) condition) is a local concept (cf. \cite[Section 8.5.3]{BBF13book}), it is satisfied in the two-phase $\mathbb P^{k}$-$\mathbb P^{k-1}$ iso-parametric setting considered in this paper. 
The discrete inf-sup condition together with the error equation of $\eum$ (i.e., \eqref{eq:eu}) implies
\begin{align}\label{eq:ep-est}
	&\| \epm \|_{L^2(\Ohsm)}
	\leq
	\| \bar e_p^m \|_{L^2(\Ohsm)}
	+
	\| \epm - \bar e_p^m \|_{L^2(\Ohsm)}
	\notag\\
	&\lesssim
	\sup_{\phi_h\in V_h^k(\Ohsm)^d} \frac{(\nabla_\Ohsm \cdot \phi_h, \bar e_p^m)_\Ohsm}{\| \nabla_\Ohsm \phi_h \|_{L^2(\Ohsm)}}
	+
	\| \epm - \bar e_p^m \|_{L^2(\Ohsm)}
	\quad\mbox{(by inf-sup condition)}
	\notag\\
	&\leq
	\sup_{\phi_h\in V_h^k(\Ohsm)^d} \frac{(\nabla_\Ohsm \cdot \phi_h, e_p^m)_\Ohsm}{\| \nabla_\Ohsm \phi_h \|_{L^2(\Ohsm)}}
	+
	2 \| \epm - \bar e_p^m \|_{L^2(\Ohsm)}
	\quad\mbox{(by H\"older's inequality)}
	\notag\\
	&\lesssim
	\sup_{\phi_h\in V_h^k(\Ohsm)^d}
	\frac{1}{\| \nabla_\Ohsm \phi_h \|_{L^2(\Ohsm)}}
	\bigg(\int_{\Ohsm} 2\nu \D_\Ohsm \eum \cdot \D_\Ohsm \phi_h
	\notag\\
	&\qquad
	+
	\int_\Ghsm \ekm \cdot \phi_h
	+
	J_u^{m}(\phi_h)
	+
	d_u^{m}(\phi_h) \bigg)
	+
	\Big| \int_\Ohm \phm - \int_\Ohsm \phm \Big|
	\notag\\
	&\lesssim
	\| \D_\Ohsm \eum \|_{L^2(\Ohsm)}
	+
	h^{-1/2}\| \ekm \|_{H^{-1}(\Ghsm)}
	+
	\| J_u^m(\cdot) \|_{H^{-1}(\Ohsm)}
	+
	\| d_u^{m}(\cdot) \|_{H^{-1}(\Ohsm)} 
	\notag\\
	&\quad
	+
	(1 + \| \epm \|_{L^2(\Ohsm)})
	\| \nabla_\Ohsm E_h \ehxm \|_{L^2(\Ohsm)}
	\notag\\
	&\lesssim
	\| \D_\Ohsm \eum \|_{L^2(\Ohsm)}
	+
	{h^{-1/2}}\| \ekm \|_{H^{-1}(\Ghsm)}
	+
	\| d_u^{m}(\cdot) \|_{H^{-1}(\Ohsm)} 
	\notag\\
	&\quad
	+
	\| \nabla_\Ohsm E_h \ehxm \|_{L^2(\Ohsm)}
	+
	\| \ehxm \|_{L^{2}(\Ghsm)}
	\qquad\mbox{(using \eqref{eq:Jum})}
	\notag\\
	&\quad
	+
	\| \nabla_\Ohsm E_h \ehxm \|_{L^\infty(\Ohsm)}
	(\| \nabla_\Ohsm \eum \|_{L^2(\Ohsm)} + \| \epm \|_{L^2(\Ohsm)}) 
	\notag\\
	&\quad
	+
	\| \ehxm \|_{L^\infty(\Ghsm)}  \| \ekm \|_{L^2(\Ghsm)}
	\notag\\
	&\quad
	+
	(1 + \| \epm \|_{L^2(\Ohsm)})
	\| \nabla_\Ohsm E_h \ehxm \|_{L^2(\Ohsm)}
	\notag\\
	&\lesssim
	\| \D_\Ohsm \eum \|_{L^2(\Ohsm)}
	+
	{h^{-1/2}}\| \ekm \|_{H^{-1}(\Ghsm)}
	+
	\| \ehxm \|_{H^{1/2}(\Ghsm)}
	+
	\| d_u^{m}(\cdot) \|_{H^{-1}(\Ohsm)} 
	,
\end{align}
where in the last line we have used the induction hypothesis \eqref{eq:ind-hypo}, the reverse trace inequality \eqref{eq:Eh-map} and the absorption of $\epm$ into the left-hand side.

To avoid the blow-up factor $h^{-1/2}$ on the right-hand side of \eqref{eq:ep-est}, we examine the weaker $H^{-1/2}$ norm on $\epm$. First, we need a useful lemma (cf. \cite[Corollary 1.5]{FS94} and \cite[Theorem 3.4]{GHH2006}).
\begin{lemma}\label{lemma:H-1}
	Let $s\in[-1, 1]$, $q\in(1,\infty)$ and $\hat W^{s, q}(\Omega) := W^{s, q}(\Omega) \cap L_0^q(\Omega)$ equipped with the $\| \cdot \|_{W^{s, q}(\Omega)}$ norm. Then there exists a bounded linear operator $R: \hat W^{s, q}(\Omega) \rightarrow W_0^{s+1,q}(\Omega)^d$ such that $\nabla\cdot (Rg) = g$ for all $g\in \hat W^{s, q}(\Omega)$.
	
	Specifically, taking $(s, q) = (1/2, 2)$, we have
	\begin{align}
		\| R g \|_{H^{3/2}(\Omega)}
		\leq
		C
		\| g \|_{H^{1/2}(\Omega)},\quad\forall g\in H^{1/2}(\Omega) \cap L_0^2(\Omega) .
		\notag
	\end{align}
\end{lemma}
Since $\Ohsm$ differs from $\Omega$ by an optimal perturbation error, Lemma \ref{lemma:H-1} is still applicable to $\Ohsm$ (up to lifting back from $\Ohsm$ to $\Omega$). For simplicity of presentation, we omit these optimal-order perturbation errors. Then, similar to the $L^2$ pressure estimate, we can derive
\begin{align}\label{eq:ep-est-H-1}
	&\| \epm \|_{H^{-1/2}(\Ohsm)}
	\leq
	\| \epm \|_{H_0^{-1/2}(\Ohsm)}
	\leq
	\| \bar e_p^m \|_{H_0^{-1/2}(\Ohsm)}
	+
	\| \epm - \bar e_p^m \|_{H_0^{-1/2}(\Ohsm)}
	\notag\\
	&=
	\sup_{\phi\in H^{1/2}(\Ohsm)} \frac{(\phi, \bar e_p^m)_\Ohsm}{\| \phi \|_{H^{1/2}(\Ohsm)}}
	+
	\| \epm - \bar e_p^m \|_{H_0^{-1/2}(\Ohsm)}
	\notag\\
	&\lesssim
	\sup_{\phi\in H^{1/2}(\Ohsm)} \frac{(\bar\phi, e_p^m)_\Ohsm}{\| \bar\phi \|_{H^{1/2}(\Ohsm)}}
	+
	\| \epm - \bar e_p^m \|_{H_0^{-1/2}(\Ohsm)}
	\notag\\
	&\lesssim
	\sup_{\phi\in H^{1/2}\cap L_0^2(\Ohsm)} \frac{(\phi, \epm)_\Ohsm}{\| \phi \|_{H^{1/2}(\Ohsm)}}
	+
	\Big| \int_\Ohm \phm - \int_\Ohsm \phm \Big|
	\notag\\
	&
	\lesssim
	\sup_{\phi\in C_c^\infty(\Ohsm)^d} \frac{(\nabla_\Ohsm \cdot \phi, \epm)_\Ohsm}{\| \phi \|_{H^{3/2}(\Ohsm)}}
	+
	\Big| \int_\Ohm \phm - \int_\Ohsm \phm \Big|
	\qquad\mbox{(using Lemma \ref{lemma:H-1})}
	\notag\\
	&
	=
	\sup_{\phi\in C_c^\infty(\Ohsm)^d} \frac{
		(\nabla_\Ohsm \cdot I_h^{\rm SZ} \phi, \epm)_\Ohsm
		+
		(\nabla_\Ohsm \cdot (1-I_h^{\rm SZ}) \phi, \epm)_\Ohsm 
	}
	{\| \phi \|_{H^{3/2}(\Ohsm)}}
	\notag\\
	&\quad
	+
	\Big| \int_\Ohm \phm - \int_\Ohsm \phm \Big|
	\notag\\
	&=
	\sup_{\phi\in C_c^\infty(\Ohsm)^d}
	\frac{1}{\| \phi \|_{H^{3/2}(\Ohsm)}}
	\bigg(
	\int_{\Ohsm} 2\nu \D_\Ohsm \eum \cdot \D_\Ohsm I_h^{\rm SZ} \phi
	\notag\\
	&\qquad
	+
	\int_{\Ohsm} \nabla_\Ohsm \cdot (1- I_h^{\rm SZ}) \phi\, \epm
	+
	\int_\Ghsm \ekm \cdot \phi
	-
	\int_\Ghsm \ekm \cdot (1 - I_h^{\rm SZ}) \phi
	\notag\\
	&\qquad
	+
	J_u^{m}(I_h^{\rm SZ} \phi)
	+
	d_u^{m}(I_h^{\rm SZ} \phi) \bigg)
	+
	\Big| \int_\Ohm \phm - \int_\Ohsm \phm \Big|
	\notag\\
	&\lesssim
	\| \D_\Ohsm \eum \|_{H^{-1/2}(\Ohsm)}
	+
	h^{1/2}
	\| \D_\Ohsm \eum \|_{L^{2}(\Ohsm)}
	+
	h^{1/2} \| \epm \|_{L^2(\Ohsm)}
	\notag\\
	&\quad
	+
	\| \ekm \|_{H^{-1}(\Ghsm)}
	+
	h
	\| \ekm \|_{L^{2}(\Ghsm)}
	+
	\| J_u^m(\cdot) \|_{H^{-1}(\Ohsm)}
	+
	\| d_u^{m}(\cdot) \|_{H^{-1}(\Ohsm)} 
	\notag\\
	&\quad
	+
	(1 + \| \epm \|_{L^2(\Ohsm)})
	\| \nabla_\Ohsm E_h \ehxm \|_{L^2(\Ohsm)}
	\notag\\
	&\lesssim
	\| \D_\Ohsm \eum \|_{H^{-1/2}(\Ohsm)}
	+
	h^{1/2}
	\| \D_\Ohsm \eum \|_{L^{2}(\Ohsm)}
	+
	\| \ekm \|_{H^{-1}(\Ghsm)}
	\notag\\
	&\quad
	+
	\| \ehxm \|_{H^{1/2}(\Ghsm)}
	+
	\| d_u^{m}(\cdot) \|_{H^{-1}(\Ohsm)} 
	,
\end{align}
where, in the second-to-last inequality, we have used interpolation error estimates, H\"older's inequality, and Lemma \ref{lemma:e-blinear-b}, and in the last inequality, we have used the estimates for $J_u^m$ and $\epm$, i.e., \eqref{eq:Jum} and \eqref{eq:ep-est} respectively, and have applied the inverse inequality.

\subsection{Main parabolicity structure}\label{sec:parab}
We test the error equations \eqref{eq:eu}--\eqref{eq:ek} with $(\phi_h, \chi_h, \eta_h)=\big(\eum, \ekm, \frac{\exM-\ehxm}{\tau}\big)$ and add them up to obtain:
\begin{align}\label{eq:parab}
	&
	\int_{\Ohsm} 2\nu \D_\Ohsm \eum \cdot \D_\Ohsm \eum
	+
	\int_\Ghsm \nabla_\Ghsm \frac{\exM-\ehxm}{\tau}  \cdot \nabla_\Ghsm \exM 
	\notag\\
	&=
	-
	\Big(\int_\Ohm \nabla_\Ohm \cdot  \uhm \epm
	-
	\int_\Ohsm \nabla_\Ohsm \cdot  \uhm \epm
	\Big)
	\notag\\
	&\quad
	-J_u^m(\eum)
	-J_x^m(\ekm)
	-J_\kappa^m\Big(\frac{\exM-\ehxm}{\tau}\Big)
	\notag\\
	&\quad
	-d_u^{m}(\eum)
	-d_p^m(\epm)
	-d_x^m(\ekm)
	-d_\kappa^m\Big(\frac{\exM-\ehxm}{\tau}\Big) ,
\end{align}
where we have used the following identity resulting from \eqref{eq:2phase-hs2}:
\begin{align}
	\int_\Ohsm \epm \nabla_\Ohsm \cdot \eum
	&=
	-d_p^{m}(\epm) 
		-
		\Big(\int_\Ohm \nabla_\Ohm \cdot  \uhm \epm
		-
		\int_\Ohsm \nabla_\Ohsm \cdot  \uhm \epm
		\Big)
	\notag .
\end{align}
Analogous to Lemma \ref{lemma:e-blinear-b}, using the mapping property \eqref{eq:Eh-map}, we can show
\begin{align}\label{eq:uhm-epm}
	&\Big|
	\int_\Ohm \nabla_\Ohm \cdot  \uhm \epm
	-
	\int_\Ohsm \nabla_\Ohsm \cdot  \uhm \epm
	\Big|
	\notag\\
	&\quad\lesssim
	\|  \ehxm \|_{H^1(\Ghsm)} 
	\big(\|\epm\|_{H^{-1/2}(\Ohsm)}
	+
	h^{1/2}\|\epm\|_{L^{2}(\Ohsm)}
	\big)
	+
	h^{1/2}
	\| \eum \|_{H^1(\Ohsm)} 
	\|\epm\|_{L^{2}(\Ohsm)} .
\end{align}
The detailed proof of \eqref{eq:uhm-epm} can be found in Appendix \ref{sec:H3/2}.
Applying \eqref{eq:uhm-epm}, H\"older's inequality and Young's inequality $-ab\geq -\frac{1}{2}(a^2+b^2)$ to \eqref{eq:parab}, we get
\begin{align}
	&
	\| \D_\Ohsm \eum \|_{L^2(\Ohsm)}^2
	+
	\frac{\| \nabla_\GhsM \ehxM \|_{L^2(\GhsM)}^2 - \| \nabla_\Ghsm \ehxm \|_{L^2(\Ghsm)}^2}{\tau}
	\notag\\
	&\lesssim
	\frac{\| \nabla_\GhsM \ehxM \|_{L^2(\GhsM)}^2 - \| \nabla_\Ghsm \exM \|_{L^2(\Ghsm)}^2}{\tau}
	\notag\\
	&\quad
	+
	\epsilon\| \eum \|_{H^{1}(\Ohsm)}^2
	+
	h \| \epm \|_{L^{2}(\Ohsm)}^2
	+
	\epsilon\| \epm \|_{H^{-1/2}(\Ohsm)}^2
	\notag\\
	&\quad
	+
	\epsilon^{-1}
	\| \ehxm \|_{H^{1}(\Ghsm)}^2
	+
	\| \ekm \|_{H^{-1}(\Ghsm)}^2
	+
	\epsilon
	\Big\| \frac{\exM - \ehxm}{\tau} \Big\|_{H^{1/2}(\Ghsm)}^2
	\notag\\
	&\quad
	+
	\epsilon^{-1}\| J_u^m(\cdot) \|_{H^{-1}(\Ohsm)}^2
	+
	\| J_x^m(\cdot) \|_{H^{1}(\Ghsm)}^2
	+
	\epsilon^{-1}
	\| J_\kappa^m(\cdot) \|_{H^{-1/2}(\Ghsm)}^2
	\notag\\
	&\quad
	+
	\epsilon^{-1}\| d_u^{m}(\cdot) \|_{H^{-1}(\Ohsm)}^2
	+
	h^{-1}
	\| d_p^m(\cdot) \|_{L^{2}(\Ohsm)}^2
	+
	\| d_x^m(\cdot) \|_{H^{1}(\Ghsm)}^2
	+
	\epsilon^{-1}
	\| d_\kappa^m(\cdot) \|_{H^{-1/2}(\Ghsm)}^2
	. \notag
\end{align}
Substituting the estimates for the $J^m$-terms, $(\exM-\ehxm)/\tau$, $\ekm$, and $\epm$, i.e., Lemma \ref{lemma:Jm}, \eqref{eq:ev-est1-s}, \eqref{eq:ek-est-s} and \eqref{eq:ep-est-H-1},
\begin{align}\label{eq:eu-est}
	&
	\| \D_\Ohsm \eum \|_{L^2(\Ohsm)}^2
	+
	\frac{\| \nabla_\GhsM \ehxM \|_{L^2(\GhsM)}^2 - \| \nabla_\Ghsm \ehxm \|_{L^2(\Ghsm)}^2}{\tau}
	\notag\\
	&\lesssim
	\frac{\| \nabla_\GhsM \ehxM \|_{L^2(\GhsM)}^2 - \| \nabla_\Ghsm \exM \|_{L^2(\Ghsm)}^2}{\tau}
	+
	\epsilon\| \eum \|_{H^{1}(\Ohsm)}^2
	+
	\epsilon^{-1}
	\| \ehxm \|_{H^{1}(\Ghsm)}^2
	\notag\\
	&\quad
	+
	\epsilon^{-1}\| d_u^{m}(\cdot) \|_{H^{-1}(\Ohsm)}^2
	+
	h^{-1}
	\| d_p^m(\cdot) \|_{L^{2}(\Ohsm)}^2
	+
	\| d_x^m(\cdot) \|_{H^{1}(\Ghsm)}^2
	+
	\epsilon^{-1}
	\| d_\kappa^m(\cdot) \|_{H^{-1/2}(\Ghsm)}^2
	.
\end{align}
The first term on the right-hand side is a norm conversion error which shall be shown stable in Lemma \ref{lemma:e-convert} later.

\begin{remark}\label{rmk:BGN}
	\upshape
	For the BGN-type formulation, \eqref{eq:2phase-h3} should be replaced by
	\begin{align}
		\int_\Ghm \frac{\XhM-\Xhm}{\tau} \cdot \nhm\, \nhm\cdot \chi_h &=
		\int_\Ghm \uhm \cdot \chi_h
		\qquad\forall \chi_h\in S_h(\Ghm)^{d},
	\end{align}
	where $\nhm$ denotes the unit outer normal vector of $\Ghm$. Due to the presence of $\nhm$, the corresponding error equation \eqref{eq:ex} will contain an additional term of the form $\int_\Ghsm \nabla_\Ghsm\ehxm\, \chi_h$. Unfortunately, this extra error term cannot be controlled by the left-hand side of the main parabolic estimate \eqref{eq:parab} when choosing $\chi_h = \ekm$.
\end{remark}

\subsection{Time-marching estimates}
\label{sec:marching}
By absorbing $\eum$ on the right-hand side of \eqref{eq:eu-est} into the left-hand side and using the estimates for $d^m$s (Lemma \ref{lemma:dm}), we obtain
\begin{align}
	&\| \D_\Ohsm \eum \|_{L^2(\Ohsm)}^2
	+
	\frac{\| \nabla_\Ghsm \exM \|_{L^2(\Ghsm)}^2}{\tau}
	\lesssim
	\frac{\| \nabla_\Ghsm \ehxm \|_{L^2(\Ghsm)}^2}{\tau}
	+
	\epsilon^{-1}
	\| \ehxm \|_{H^1(\Ghsm)}^2
	\notag\\
	&\qquad\qquad\qquad\qquad\qquad\qquad\qquad
	+
	\epsilon^{-1} \big((1+\musm)h^{k-1/2} +(1+\nusm)h^{k+\alpha(d)} + \tau\big)^2
	\notag
	,
\end{align}
where $\alpha(2)=0,\alpha(3)=-1/2$.
Consequently,
\begin{align}
	\| \nabla_\Ghsm \exM \|_{L^2(\Ghsm)}
	&\lesssim
	\| \nabla_\Ghsm \ehxm \|_{L^2(\Ghsm)}
	+
	\tau^{1/2}\big(\tau+h^{k-1} + \| \ehxm \|_{L^{2}(\Ghsm)}\big)
	,
	\label{eq:ex-bbd}\\
	\| \nabla_\Ohsm \eum \|_{L^2(\Ohsm)}
	&\lesssim
	\tau^{-1/2}
	\| \nabla_\Ghsm \ehxm \|_{L^2(\Ghsm)}
	+
	\big(\tau+h^{k-1} + \| \ehxm \|_{L^{2}(\Ghsm)}\big)
	\label{eq:ex-bbd1}
	,
\end{align}
where we have used the boundedness
$\mu_{*,m} h^{k-1/2} \leq h^{k-1}$ and $\nu_{*,m} h^{k+\alpha(d)} \leq h^{k-1}$,
for some $h\leq h_{\musm,\nusm,\mum,\num}$.
Substituting these new a priori estimates \eqref{eq:ex-bbd}--\eqref{eq:ex-bbd1} back into \eqref{eq:ev-est1-s},
\begin{align}\label{eq:ev-bbd}
	&\Big\| \frac{\exM-\ehxm}{\tau} \Big\|_{H^{1/2}(\Ghsm)}
	\lesssim
	\| \eum \|_{H^{1}(\Ohsm)}
	+
	\| \ehxm \|_{H^{1/2}(\Ghsm)}
	+
	\tau
	\notag\\
	&\qquad\lesssim
	\tau^{-1/2}
	\| \nabla_\Ghsm \ehxm \|_{L^2(\Ghsm)}
	+
	\| \ehxm \|_{L^{2}(\Ghsm)}
	+
	\tau+h^{k-1}
	,
\end{align}
and, consequently, using the discrete Sobolev inequality (cf. \cite[Lemma 4.9.2]{Brenner08}), we obtain
\begin{align}\label{eq:ev-bbd1}
	&\Big\| \frac{\exM-\ehxm}{\tau} \Big\|_{L^{\infty}(\Ghsm)}
	\lesssim
	|\log h| h^{-1/2}
	\Big\| \frac{\exM-\ehxm}{\tau} \Big\|_{H^{1/2}(\Ghsm)}
	\notag\\
	&\qquad\lesssim
	|\log h| h^{-1/2}
	\Big(\tau^{-1/2}
	\| \nabla_\Ghsm \ehxm \|_{L^2(\Ghsm)}
	+
	\| \ehxm \|_{L^{2}(\Ghsm)}
	+
	\tau+h^{k-1}\Big)
	.
\end{align}
Note that \eqref{eq:ev-bbd} and \eqref{eq:ev-bbd1} contain a priori information for the position $\exM$ at the next time level, i.e., $t_{m+1}$.
The a priori estimate for $\ehxM$ can be derived from the geometric relations \eqref{eq:geo_rel_1}--\eqref{eq:geo_rel_2},
\begin{align}\label{eq:ehx-bbd}
	\| \ehxM \|_{L^\infty(\Ghsm)}
	&\leq
	\| I_h (\NsM\circ \hat X_{h,*}^{m+1} \exM) \|_{L^\infty(\Ghsm)}
	+
	\| r_h^{m+1} \|_{L^\infty(\Ghsm)}
	\notag\\
	&\lesssim
	\| I_h (\NsM\circ \hat X_{h,*}^{m+1} \exM) \|_{L^\infty(\Ghsm)}
	+
	\| I_h (\TsM\circ \hat X_{h,*}^{m+1} \exM) \|_{L^\infty(\Ghsm)}^2
	\notag\\
	&\lesssim
	\| \exM \|_{L^\infty(\Ghsm)}
	+
	\| \exM \|_{L^\infty(\Ghsm)}^2
	\notag\\
	&\lesssim
	|\log h| \| \exM \|_{H^1(\Ghsm)}
	+
	|\log h|^2
	\| \exM \|_{H^1(\Ghsm)}^2
	,
\end{align}
where we have applied the discrete Sobolev inequality in the last line.

With 
the geometric relations \eqref{eq:geo_rel_4}--\eqref{eq:geo_rel_6} and the induction hypothesis \eqref{eq:ind-hypo}, a similar argument to \cite[Eq. (4.91)]{BL24} leads to the result stated in the following lemma.
\begin{lemma}\label{lemma:hatX-W1inf}
	We have the time-marching stability estimates in $W^{1,\infty}$ norm:
	\begin{align}
		\| I_h(\nsM\circ \hat X_{h,*}^{m+1}-\nsm\circ \hat X_{h,*}^{m}) \|_{W^{1,\infty}(\Ghsm)}
		\lesssim
		\tau + \| \hat X_{h,*}^{m+1} - \hat X_{h,*}^{m} \|_{W^{1,\infty}(\Ghsm)} ,
		\notag
	\end{align}
	and
	\begin{align}
		\| \hat X_{h,*}^{m+1} - \hat X_{h,*}^{m} \|_{W^{1,\infty}(\Ghsm)}
		\lesssim
		\tau + \tau \Big\| \frac{\exM - \ehxm}{\tau} \Big\|_{W^{1,\infty}(\Ghsm)} .
		\notag
	\end{align}
	Consequently, in view of \eqref{eq:ev-bbd} and the induction hypothesis \eqref{eq:ind-hypo}, the $W^{1,q},q\in[1,\infty]$, norms on $\Ghsm$ and $\GhsM$ are equivalent (cf. \cite[Lemma 7.2]{KLL19}).
\end{lemma}
\begin{proof}
	First, we use Lipschitz continuity of $n_*(\cdot,\cdot)$ and the nonlinear super-approximation (Lemma \ref{lemma:super_conv-nonlinear}) to get
	\begin{align}\label{nsM-nsm-nodes}
		&\| I_h(\nsM\circ \hat X_{h,*}^{m+1}-\nsm\circ \hat X_{h,*}^{m}) \|_{L^\infty(\Ghsm)} 
		\notag\\
		&\leq
		\| I_h(n_*(\cdot,t_{m+1})\circ \hat X_{h,*}^{m} - n_*(\cdot,t_{m})\circ \hat X_{h,*}^{m}) \|_{L^\infty(\Ghsm)}  
		\notag\\
		&\quad
		+
		\| I_h(n_*(\cdot,t_{m+1})\circ \hat X_{h,*}^{m+1} - n_*(\cdot,t_{m+1})\circ \hat X_{h,*}^{m}) \|_{L^\infty(\Ghsm)}  \notag\\
		&\lesssim
		\tau
		+
		\| \hat X_{h,*}^{m+1} - \hat X_{h,*}^{m} \|_{L^\infty(\Ghsm)}
		,
	\end{align}
	and similarly,
	\begin{align}
		\| I_h(\nsM\circ \hat X_{h,*}^{m+1}-\nsm\circ \hat X_{h,*}^{m}) \|_{W^{1,\infty}(\Ghsm)} 
		\lesssim
		\tau
		+
		\| \hat X_{h,*}^{m+1} - \hat X_{h,*}^{m} \|_{W^{1,\infty}(\Ghsm)}
		.
		\notag
	\end{align}
	This proves the first result.
	
	Second, we estimate the numerical displacement as follows:
	\begin{align}\label{eq:hat_X_s_diff}
		&\| \hat X_{h,*}^{m+1} - \hat X_{h,*}^{m} \|_{L^{\infty}(\Ghsm)} \notag\\ 
		&\le 
		\| \hat X_{h,*}^{m+1} - X_{h}^{m+1} \|_{L^{\infty}(\Ghsm)}
		+\| X_{h}^{m+1} - X_{h}^{m} \|_{L^{\infty}(\Ghsm)}
		+ \| X_{h}^{m} - \hat X_{h,*}^{m} \|_{L^{\infty}(\Ghsm)} \notag\\ 
		&=
		\| \hat e_{x}^{m+1} \|_{L^{\infty}(\Ghsm)}
		+ \|\exM - \ehxm + \tau I_h((u^m + g^m)\circ a^m|_\Ghsm)  \|_{L^{ \infty}(\Ghsm)} 
		+\| \ehxm \|_{L^{\infty}(\Ghsm)} \notag\\
		&\lesssim
		\tau 
		+
		\| \ehxM \|_{L^{\infty}(\Ghsm)}
		+ \| \exM \|_{L^{ \infty}(\Ghsm)} 
		+\| \ehxm \|_{L^{\infty}(\Ghsm)}
		\notag\\
		&\lesssim
		\tau + |\log h| \| \ehxm \|_{H^1(\Ghsm)}
		\qquad\mbox{(using \eqref{eq:ind-hypo}, \eqref{eq:ev-bbd}--\eqref{eq:ehx-bbd})}
		.
	\end{align}
	This implies
	\begin{align}
		&\| \hat X_{h,*}^{m+1} - \hat X_{h,*}^{m} \|_{W^{1,\infty}(\Ghsm)}
		\notag\\
		&\leq
		\| I_h\Nsm (\hat X_{h,*}^{m+1} - \hat X_{h,*}^{m}) \|_{W^{1,\infty}(\Ghsm)}
		+
		\| I_h\Tsm (\hat X_{h,*}^{m+1} - \hat X_{h,*}^{m}) \|_{W^{1,\infty}(\Ghsm)}
		\notag\\
		&\lesssim
		\tau
		+
		h^{-1}
		(\tau^2 + \| I_h\Tsm (\hat X_{h,*}^{m+1} - \hat X_{h,*}^{m}) \|_{L^{\infty}(\Ghsm)}^2)
		\notag\\
		&\quad+
		\tau
		+
		\| I_h\Tsm (\exM - \ehxm) \|_{W^{1,\infty}(\Ghsm)}
		\notag\\
		&\quad
		+
		h^{-1}
		\| I_h\Tsm(\nsM\circ \hat X_{h,*}^{m+1}-\nsm\circ \hat X_{h,*}^{m}) \|_{L^\infty(\Ghsm)} 
		\| \ehxM \|_{L^\infty(\Ghsm)}
		\notag\\
		&\hspace{130pt}\mbox{(using \eqref{eq:geo_rel_3}--\eqref{eq:geo_rel_6} and the inverse inequality)}
		\notag\\
		&\lesssim
		h^{-1}(\tau + |\log h| \| \ehxm \|_{H^1(\Ghsm)})
		\| I_h\Tsm (\hat X_{h,*}^{m+1} - \hat X_{h,*}^{m}) \|_{L^{\infty}(\Ghsm)}
		\notag\\
		&\quad
		+
		\tau
		+
		\| I_h\Tsm (\exM - \ehxm) \|_{W^{1,\infty}(\Ghsm)}
		\notag\\
		&\quad
		+
		h^{-1}
		(\tau + \| \hat X_{h,*}^{m+1} - \hat X_{h,*}^{m} \|_{L^{\infty}(\Ghsm)})
		\| \ehxM \|_{L^\infty(\Ghsm)}
		\quad
		\mbox{(using \eqref{nsM-nsm-nodes} and \eqref{eq:hat_X_s_diff})} . \notag
	\end{align}
	In view of \eqref{eq:ind-hypo} and \eqref{eq:ev-bbd}--\eqref{eq:ehx-bbd}, we see that the two terms containing $\hat X_{h,*}^{m+1} - \hat X_{h,*}^{m}$ on the right-hand side can be absorbed into the left-hand side.
	Therefore, we obtain the second result:
	\begin{align}
		\| \hat X_{h,*}^{m+1} - \hat X_{h,*}^{m} \|_{W^{1,\infty}(\Ghsm)}
		\lesssim
		\tau
		+
		\| I_h\Tsm (\exM - \ehxm) \|_{W^{1,\infty}(\Ghsm)} .
		\notag
	\end{align}
\end{proof}
With this lemma, we get an $H^1$-version of \eqref{eq:ehx-bbd} using a very similar argument:
\begin{align}\label{eq:ehx-bbd1}
	\| \ehxM \|_{H^1(\Ghsm)}
	&\lesssim
	\| I_h \Nsm \exM \|_{H^1(\Ghsm)}
	+
	\| I_h (\NsM - \Nsm) \exM \|_{H^1(\Ghsm)}
	+
	\| r_h^{m+1} \|_{H^1(\Ghsm)}
	\notag\\
	&\lesssim
	\|  \exM \|_{H^1(\Ghsm)}
	+
	h^{-1}
	\| \exM \|_{L^4(\Ghsm)}^2
	\notag\\
	&\quad
	+
	(\tau + \| \exM - \ehxm \|_{W^{1,\infty}(\Ghsm)})
	\| \exM \|_{H^1(\Ghsm)}
	\qquad\mbox{(by Lemma \ref{lemma:hatX-W1inf})}
	\notag\\
	&\lesssim
	\|  \exM \|_{H^1(\Ghsm)}
	\qquad\mbox{(by induction hypothesis \eqref{eq:ind-hypo})}
	\notag\\
	&\lesssim
	\| \nabla_\Ghsm \ehxm \|_{L^2(\Ghsm)}
	+
	\tau^{1/2}
	\| \ehxm \|_{L^2(\Ghsm)}
	+
	\tau^{1/2}(\tau+h^{k-1})
	\quad\mbox{(using \eqref{eq:ex-bbd})}
	.
\end{align}
Consequently, we have another important stability estimate which will be helpful in the proof of Lemma \ref{lemma:e-convert} later.
\begin{lemma}\label{lemma:e-diff-W1inf}
	\begin{align}
		\| \exM - \ehxm \|_{W^{1,\infty}(\Ghsm)}
		\| \ehxM \|_{H^{1}(\Ghsm)}
		\lesssim
		\tau \Big\| \frac{\exM - \ehxm}{\tau} \Big\|_{H^{1/2}(\Ghsm)} .
		\notag
	\end{align}
\end{lemma}
\begin{proof}
	It holds that
	\begin{align}
		\| \exM - \ehxm \|_{W^{1,\infty}(\Ghsm)}
		\| \ehxM \|_{H^{1}(\Ghsm)}
		\lesssim
		h^{-d/2}\tau \Big\| \frac{\exM - \ehxm}{\tau} \Big\|_{H^{1/2}(\Ghsm)}
		\| \ehxM \|_{H^{1}(\Ghsm)} .
		\notag
	\end{align}
	The proof is complete in view of \eqref{eq:ehx-bbd1} and the induction hypothesis \eqref{eq:ind-hypo}.
\end{proof}

\subsection{Technical lemmas}
The following orthogonality lemma shows that under an almost-orthogonality structure, it is possible to gain regularity in the upper bound by one gradient.
\begin{lemma}\label{lemma:NT_stab_ref}
	The following bound holds for any finite element functions $f_h,g_h \in S_h(\Ghsm)$:
	\begin{align}
		&\Big| \int_{\Ghsm} \nabla_{\Ghsm} I_h \Nsm f_h \cdot  \nabla_{\Ghsm}  I_h \Tsm g_h \Big|
		\notag\\
		&
		\qquad\qquad
		\lesssim
		\min\Big\{
		\| f_h \|_{H^1(\Ghsm)} \| g_h \|_{L^2(\Ghsm)}
		,
		\| f_h \|_{L^2(\Ghsm)} \| g_h \|_{H^1(\Ghsm)}
		\Big\} .
		\notag
	\end{align}
\end{lemma}
\begin{proof}
	See Appendix \ref{sec:appndix_tan_stab}.
\end{proof}

Based on Lemma \ref{lemma:hatX-W1inf}--\ref{lemma:NT_stab_ref}, we are able to handle the norm conversion error, i.e., the first term on the right-hand side of \eqref{eq:eu-est}.
\begin{lemma}\label{lemma:e-convert}
	The following norm conversion is stable:
	\begin{align}
		&\| \ehxM \|_{H^1(\GhsM)}^2 - \| \exM \|_{H^1(\Ghsm)}^2
		\notag\\
		&
		\qquad\qquad
		\lesssim
		\tau \Big( 
		\epsilon^{-1}\| \ehxm \|_{H^1(\Ghsm)}^2 
		+ 
		\epsilon^{-1} \| \ehxM \|_{H^1(\GhsM)}^2 + 
		\epsilon\Big\| \frac{\exM - \ehxm}{\tau} \Big\|_{H^{1/2}(\Ghsm)}^2\Big) .
		\notag
	\end{align}
\end{lemma}
\begin{proof}
	See Appendix \ref{sec:e-convert}.
\end{proof}

\subsection{Main error estimates}
Substituting the norm conversion stability estimate (Lemma \ref{lemma:e-convert}), the $L^2$ interface estimate \eqref{eq:ex-est-L2-s}, the $H^{1/2}$ error estimates for $\frac{\exM-\ehxm}{\tau}$ (i.e., 	\eqref{eq:ev-est1-s}) and the estimates for $d^m$s (Lemma \ref{lemma:dm}) into the right-hand side of the main error equation \eqref{eq:eu-est}, and using the norm equivalence, we finally obtain
\begin{align}
	&\| \D_\Ohsm \eum \|_{L^2(\Ohsm)}^2
	+
	\frac{\| \ehxM \|_{H^1(\GhsM)}^2 - \| \ehxm \|_{H^1(\Ghsm)}^2}{\tau}
	\notag\\
	&\lesssim
	\epsilon\| \eum \|_{H^{1}(\Ohsm)}^2
	+
	\epsilon^{-1} \| \ehxm \|_{H^{1}(\Ghsm)}^2
	+
	\epsilon^{-1} \| \ehxM \|_{H^{1}(\Ghsm)}^2
	\notag\\
	&\quad
	+
	\epsilon^{-1} \big((1+\musm)h^{k-1/2} +(1+\nusm)h^{k+\alpha(d)} + \tau \big)^2 \notag
	,
\end{align}
where $\alpha(2)=0,\alpha(3)=-1/2$.
\begin{remark}\upshape\label{rmk:1/2}
	The exponent $-1/2$ of $h$ comes from the factor $h^{-1/2}$ in the front of the $d_p^m$ term, and
	the exponent $-1/2$ in $\alpha(3)$ comes from the $H^{-1/2}$ norm on $d_\kappa^m$.
\end{remark}
Since $\epsilon$ is a universal constant, we can use Gr\"onwall's inequality, Korn's inequality and the Poincar\'e inequality to conclude
\begin{align}\label{eq:e-est}
	&\sum_{j=0}^{m} \tau \| e_u^j \|_{H^1(\hat\Omega_{h,*}^j)}^2
	+
	\max_{j=0,\cdots,m+1}
	\| \hat e_x^j \|_{H^1(\hat\Gamma_{h,*}^j)}^2
	\notag\\
	&
	\qquad\qquad
	\leq
	C_{\bar\mu_m, \bar\nu_m}
	\sum_{j=0}^{m} 
	\tau
	\big((1+\mu_{*,j})
	h^{k-1/2} + (1+\nu_{*,j})h^{k+\alpha(d)} + \tau \big)^2
	,
\end{align}
for all $0\leq m \leq [T/\tau]$.

\section{Shape regularity analysis}\label{sec:shape-reg}

\subsection{Trajectory estimates}
\label{sec:traj-est}

In this subsection, we discuss the trajectory error of the interface approximation.
We define the trajectory error of the interface $e_{x,\#}^m:= X_h^m - X_{h,\#}^m\in V_h^k(\Ohso)^d$ where $X_{h,\#}^m = I_h (X^m\circ \Phi^0)\in S_h(\Ohso)^d$ is the Lagrange interpolation of the exact flow $X^m: \Omega^0\rightarrow\Om$.
Then, at the finite element nodes, the following relation holds
\begin{align}
	e_{x,\#}^{m+1} - e_{x,\#}^m
	&=
	(X_h^{m+1} - X_h^m) - (X_{h,\#}^{m+1} - X_{h,\#}^m)
	\notag\\
	&=
	\tau (u_h^m - I_h(u^m\circ X^m\circ \Phi^0)) - \tau \Big(\frac{X_{h,\#}^{m+1} - X_{h,\#}^m}{\tau} - I_h(u^m\circ X^m\circ \Phi^0)\Big)
	\notag\\
	&=
	\tau (u_h^m - I_h(u^m\circ \Phi^m)) 
	\notag\\
	&\quad
	+
	\tau (I_h(u^m\circ \Phi^m) - I_h(u^m\circ X^m\circ \Phi^0))  
	\notag\\
	&\quad
	- \tau \Big(\frac{X_{h,\#}^{m+1} - X_{h,\#}^m}{\tau} - I_h(u^m\circ X^m\circ \Phi^0)\Big)
	=:
	\tau \sum_{i=1}^3 E_i^m
	\notag
	.
\end{align}
By definition, $E_1^m = e_u^m$. Similar to the Taylor remainder estimate \eqref{W1infty-g}, we know 
%
at each finite element node $p\in\mathcal N(\Ohso)$ that
\begin{align}
	E_3^m(p)
	=
	\frac{X^{m+1}-X^m}{\tau}\circ \Phi^0(p) - u^m \circ X^m \circ \Phi^0 (p)
	=:
	\tilde g^m \circ \Phi^0 (p)
	\notag
\end{align}
with $\tilde g^m:\Omega\rightarrow \R^d$ satisfying $\| \tilde g^m \|_{W^{1,\infty}(\Omega)}\leq C\tau$.
Therefore,
\begin{align}
	\| E_3^m \|_{W^{1,\infty}(\Ghso)}
	\leq
	C_{\num} \tau 
	\quad\mbox{and}\quad
	\| E_3^m \|_{W^{1,\infty}(\Ohso)}
	\leq
	C_{\mum} \tau 
	.
	\notag
\end{align}

For each interface finite element node $p\in \mathcal N(\Ghso)$, we have $E_2^m(p) = u^m\circ\hat X_{h,*}^m(p) - u^m\circ X_{h,\#}^m(p)$. Therefore, using Lemma \ref{lemma:super_conv-nonlinear},
\begin{align}
	\| E_2^m \|_{H^1(\Ghso)} 
	\leq
	C
	\| e_{x,\#}^m \|_{H^1(\Ghso)}
	+
	C
	\| \ehxm \|_{H^1(\Ghso)}
	. \notag
\end{align}
In the bulk domain, we derive
\begin{align}
	&\| E_2^m \|_{H^1(\Ohso)}
	\leq
	\| (1 -I_h)(u^m\circ \Phi^m\circ \hat X_{h,*}^m) \|_{H^1(\Ohso)}
	+
	\| (1 -I_h)( u^m\circ X^m \circ \Phi^0) \|_{H^1(\Ohso)}
	\notag\\
	&\quad
	+
	C
	\| \Phi^m \circ \hat X_{h,*}^m - X^m \circ \Phi^0 \|_{H^1(\Ohso)}
	\qquad\mbox{(by Lipschitz continuity)}
	\notag\\
	&\leq
	\| (1 -I_h)(u^m\circ \Phi^m\circ \hat X_{h,*}^m) \|_{H^1(\Ohso)}
	+
	\| (1 -I_h)( u^m\circ X^m \circ \Phi^0) \|_{H^1(\Ohso)}
	\notag\\
	&\quad
	+
	C
	\| \hat X_{h,*}^m - I_h(X^m\circ\Phi^0) \|_{H^1(\Ohso)}
	+
	C
	\| \Phi^m\circ \hat X_{h,*}^m - \hat X_{h,*}^m \|_{H^1(\Ohso)}
	\notag\\
	&\quad
	+
	C
	\| (1 - I_h) (X^m\circ \Phi^0) \|_{H^1(\Ohso)}
	\qquad\mbox{(by the triangle inequality)}
	\notag\\
	&\leq
	C_{\mum}
	(1 + \musm) h^k
	+
	C_{\mum,\num}
	(1 + \nusm) h^{k+1/2}
	+
	C
	\| e_{x,\#}^m  \|_{H^{1}(\Ohso)}
	+
	C
	\| E_h \ehxm  \|_{H^{1}(\Ohso)}
	\notag\\
	&
	\hspace{200pt}\mbox{(using Lemma \ref{lemma:Ih} and Lemma \ref{lemma:Phi-approx})}
	.
	\notag
\end{align}
Collecting the estimates for $E^m s$ above, we conclude the trajectory estimate on the interface
\begin{align}
	\| e_{x,\#}^{m+1} \|_{H^1(\Ghso)}
	&\leq
	\sum_{j=0}^{m}
	\| e_{x,\#}^{j+1} - e_{x,\#}^{j} \|_{H^1(\Ghso)}
	\notag\\
	&\leq
	\sum_{j=0}^{m} \tau (C_{\nu_j} \tau + C \| e_u^j \|_{H^1(\Ghso)} 
	+ C \| e_{x,\#}^{j} \|_{H^1(\Ghso)}
	+ C \| \hat e_{x}^{j} \|_{H^1(\Ghso)}
	)
	\notag
	,
\end{align}
and in the bulk domain
\begin{align}
	&\| e_{x,\#}^{m+1} \|_{H^1(\Ohso)}
	\leq
	\sum_{j=0}^{m} \| e_{x,\#}^{j+1} - e_{x,\#}^{j} \|_{H^1(\Ohso)}
	\notag\\
	&\leq
	\sum_{j=0}^{m} \tau (C_{\mu_j} \tau + C \| e_u^j \|_{H^1(\Ohso)} 
	+ C \| e_{x,\#}^{j} \|_{H^1(\Ohso)}
	+ C \| E_h \hat e_{x}^{j} \|_{H^1(\Ohso)}
	)
	\notag\\
	&\quad+
	\sum_{j=0}^{m}
	\tau \big(C_{\mu_j}(1+\mu_{*,j})h^{k} + C_{\mu_j,\nu_j}(1+\nu_{*,j})h^{k+1/2}\big) .
	\notag
\end{align}
The $e_{x,\#}^{j}$ terms on the right-hand side can be eliminated by applying Gr\"onwall's inequality:
\begin{align}\label{eq:e-sharp-i}
	\| e_{x,\#}^{m+1} \|_{H^1(\Ghso)}
	&\leq
	\sum_{j=0}^{m} \tau (C_{\nu_j} \tau 
	+ C \| e_u^j \|_{H^1(\Ghso)} 
	+ C \| \hat e_x^j \|_{H^1(\Ghso)} 
	)
\end{align}
and
\begin{align}\label{eq:e-sharp-b}
	\| e_{x,\#}^{m+1} \|_{H^1(\Ohso)}
	&\leq
	\sum_{j=0}^{m} \tau (C_{\mu_j} \tau 
	+ C \| e_u^j \|_{H^1(\Ohso)} 
	+ C \| E_h \hat e_x^j \|_{H^1(\Ohso)} 
	)
	\notag\\
	&\quad+
	\sum_{j=0}^{m}
	\tau \big(C_{\mu_j}(1+\mu_{*,j})h^{k} + C_{\mu_j,\nu_j}(1+\nu_{*,j})h^{k+1/2}\big)
	.
\end{align}

\subsection{Shape regularity of $\GhsM$}\label{sec:shape-reg-G}

First, we observe the nodal-wise decomposition
\begin{align}
	\hat X_{h,*}^{m+1} 
	= X_{h,\#}^{m+1} + (X_h^{m+1} -X_{h,\#}^{m+1}) + (\hat X_{h,*}^{m+1} - X_h^{m+1})  
	= 
	X_{h,\#}^{m+1} + e_{x,\#}^{m+1} - E_h \hat e_x^{m+1} .
	\notag
\end{align}
Therefore, for any non-negative integer $l$, using the inverse inequality and the error estimates \eqref{eq:e-sharp-i} and \eqref{eq:e-est} successively, we get
\begin{align}
	&\| \hat X_{h,*}^{m+1}\|_{H_h^l(\Ghso)}
	\leq
	\| I_h (X^{m+1}\circ\Phi^{0}) \|_{H_h^l(\Ghso)}
	+
	\| e_{x,\#}^{m+1} \|_{H_h^l(\Ghso)}
	+
	\| \hat e_x^{m+1} \|_{H_h^l(\Ghso)}
	\notag\\
	&\leq
	C
	+
	C
	h^{-l+1}
	\| e_{x,\#}^{m+1} \|_{H^1(\Ghso)}
	+
	C
	h^{-l+1}
	\| \hat e_x^{m+1} \|_{H^1(\Ghso)}
	\notag\\
	&\leq
	C
	+
	C
	h^{-l+1}
	\| \hat e_x^{m+1} \|_{H^1(\Ghso)}
	+
	h^{-l+1} \sum_{j=0}^{m} C_{\nu_j} \tau (\tau + \| e_u^j \|_{H^1(\Ghso)}
	+ \| \hat e_x^j \|_{H^1(\Ghso)}
	)
	\notag\\
	&\leq
	C 
	+
	h^{-l+{1/2}}
	C_{\bar\mu_{m},\bar\nu_{m}}
	\Big(\sum_{j=0}^{m}
	\tau \big(\tau + (1 + \mu_{*,j}) h^{k-1/2} + (1 + \nu_{*,j}) h^{k+\alpha(d)} \big)^2 \Big)^{1/2}
	\notag .
\end{align}
Taking squares and applying Gr\"onwall's inequality, we obtain
\begin{align}\label{eq:Gron-Ghsm}
	\| \hat X_{h,*}^{m+1}\|_{H_h^l(\Ghso)}^2
	\leq
	C +
	C_{\bar\mu_{m}, \bar\nu_{m}}
	h^{-2l+1}(\tau + h^{k-1/2})^2
	+
	C_{\bar\mu_{m}, \bar\nu_{m}}
	h^{-2l}
	\sum_{j=0}^{m}
	\tau h^{2k} \mu_{*,j}^2 ,
\end{align}
for $l=0,\cdots,k$ and $h\leq h_{\bar\mu_{*,m}, \bar\nu_{*,m},\bar\mu_{m}, \bar\nu_{m}}$.

A similar argument and the triangle inequality lead to the following $W^{1,\infty}$-smallness
\begin{align}\label{eq:Gron-Ghsm1}
	&\| \hat X_{h,*}^{m+1} - I_h (X^{m+1}\circ\Phi^{0}) \|_{W^{1,\infty}(\Ghso)}
	\notag\\
	&
	\qquad\qquad
	\leq
	h^{-d/2}
	C_{\bar\mu_{m},\bar\nu_{m}}
	\Big(\sum_{j=0}^{m}
	\tau \big(\tau + (1 + \mu_{*,j}) h^{k-1/2} + (1 + \nu_{*,j}) h^{k+\alpha(d)}\big)^2\Big)^{1/2}
	.
\end{align}
The parametrization map $\hat X_{h,*}^{m+1}: \Ghso\rightarrow\GhsM$ admits the decomposition
$\hat X_{h,*}^{m+1} = \hat Y_{h,*}^{m+1} \circ I_h(X^{m+1}\circ\Phi^0)$, where $I_h(X^{m+1}\circ\Phi^0):\Ghso\rightarrow\Gamma_{h,\#}^{m+1}$ and $\hat Y_{h,*}^{m+1} := {\rm id}_{\Gamma_{h,\#}^{m+1}}
+
\hat X_{h,*}^{m+1}
-
I_h(X^{m+1}\circ\Phi^0):
\Gamma_{h,\#}^{m+1}
\rightarrow
\hat\Gamma_{h,*}^{m+1}
$. Here, $\hat X_{h,*}^{m+1}$ and $I_h(X^{m+1}\circ\Phi^0)$ are interpreted as finite element functions on $\Gamma_{h,\#}^{m+1}$. In view of the smallness condition \eqref{eq:Gron-Ghsm1} and the fact that $I_h(X^{m+1}\circ\Phi^0)$ is the interpolation of a globally continuous and piecewise $C^{k+1}$ function, we know that the $L^q$ and $W^{1,q}$, $q\in[1,\infty]$, norms on $\Ghso$, $\Gamma_{h,\#}^{m+1}$ and $\hat\Gamma_{h,*}^{m+1}$ are equivalent up to a constant that is independent of $\bar\mu_{*,m}, \bar\nu_{*,m},\bar\mu_{m}, \bar\nu_{m}$ for sufficiently small $h\leq h_{\bar\mu_{*,m}, \bar\nu_{*,m},\bar\mu_{m}, \bar\nu_{m}}$; cf. \cite[Lemma 4.3]{KLL17} and \cite[Lemma 7.2]{KLL19}.

Consequently,
\begin{align}\label{eq:Gron-Ghsm2}
	\| (\hat X_{h,*}^{m+1})^{-1}\|_{W^{1,\infty}(\GhsM)}
	\leq
	C
	\| {\rm id}_{\hat\Gamma_{h,*}^{0}} \|_{W^{1,\infty}(\hat\Gamma_{h,*}^{0})}
	\leq
	C
	,
\end{align}
where we have used a change of variables and the norm equivalence between $\Ghso$ and $\hat\Gamma_{h,*}^{m+1}$.

\subsection{Shape regularity of $\hat\Omega_{h,*}^{m+1}$}
In the bulk region, for any non-negative integer $l$, a similar argument as in Section \ref{sec:shape-reg-G} yields
\begin{align}
	&\| \hat X_{h,*}^{m+1}\|_{H_h^l(\Ohso)}
	\leq
	\| I_h (X^{m+1}\circ\Phi^{0}) \|_{H_h^l(\Ohso)}
	+
	\| e_{x,\#}^{m+1} \|_{H_h^l(\Ohso)}
	+
	\| E_h \hat e_x^{m+1} \|_{H_h^l(\Ohso)}
	\notag\\
	&\leq
	C
	+
	C
	h^{-l+1}
	\| e_{x,\#}^{m+1} \|_{H^1(\Ohso)}
	+
	C
	h^{-l+3/2}
	\| E_h \hat e_x^{m+1} \|_{H_h^{3/2}(\Ohso)}
	\notag\\
	&\leq
	C 
	+
	C
	h^{-l+3/2}
	\| \hat e_x^{m+1} \|_{H^1(\Ghso)}
	+
	h^{-l+1} \sum_{j=0}^{m} C_{\mu_j} \tau (\tau + \| e_u^j \|_{H^1(\Ohso)}
	+ \| E_h \hat e_x^j \|_{H^1(\Ohso)}
	)
	\notag\\
	&\quad+
	h^{-l+1}
	\sum_{j=0}^{m}
	\tau \big(C_{\mu_j}(1+\mu_{*,j})h^{k} + C_{\mu_j,\nu_j}(1+\nu_{*,j})h^{k+1/2}\big)
	\notag\\
	&\leq
	C 
	+
	h^{-l+1}
	C_{\bar\mu_{m},\bar\nu_{m}}
	\Big(
	\sum_{j=0}^{m}
	\tau \big(\tau + (1 + \mu_{*,j}) h^{k-1/2} + (1 + \nu_{*,j}) h^{k+\alpha(d)}\big)^2\Big)^{1/2}
	\notag\\
	&\quad+
	h^{-l+1}
	\sum_{j=0}^{m}
	\tau \big(C_{\mu_j}(1+\mu_{*,j})h^{k} + C_{\mu_j,\nu_j}(1+\nu_{*,j})h^{k+1/2}\big)
	.\notag
\end{align}
\begin{remark}\upshape
	The fractional-order inverse inequality holds due to the regular scaling behavior of the Slobodeckij seminorm; also see \cite[Lemma 4.5.3]{Brenner08}.
\end{remark}
After taking squares and applying Gr\"onwall’s inequality, we obtain
\begin{align}\label{eq:Gron-Ohsm}
	\| \hat X_{h,*}^{m+1}\|_{H_h^l(\Ohso)}^2
	&\leq
	C +
	C_{\bar\mu_{m}, \bar\nu_{m}}
	h^{-2l+2} (\tau + h^{k-1/2})^2
	+
	C_{\bar\mu_{m}, \bar\nu_{m}}
	h^{-2l+2}
	\sum_{j=0}^{m}
	\tau h^{2k+2\alpha(d)} \nu_{*,j}^2
	,
\end{align}
for $l=0,\cdots,k$ and $h\leq h_{\bar\mu_{*,m}, \bar\nu_{*,m},\bar\mu_{m}, \bar\nu_{m}}$.
Analogous to \eqref{eq:Gron-Ghsm1}--\eqref{eq:Gron-Ghsm2}, taking $l=1$ and using the inverse inequality, we derive
\begin{align}\label{eq:Gron-Ghsm11}
	&\| \hat X_{h,*}^{m+1} - I_h (X^{m+1}\circ\Phi^{0}) \|_{W^{1,\infty}(\Ohso)}
	\notag\\
	&\leq
	h^{-d/2}
	C_{\bar\mu_{m},\bar\nu_{m}}
	\Big(
	\sum_{j=0}^{m}
	\tau \big(\tau + (1 + \mu_{*,j}) h^{k-1/2} + (1 + \nu_{*,j}) h^{k+\alpha(d)}\big)^2\Big)^{1/2}
	\notag\\
	&\quad+
	h^{-d/2}
	\sum_{j=0}^{m}
	\tau \big(C_{\mu_j}(1+\mu_{*,j})h^{k} + C_{\mu_j,\nu_j}(1+\nu_{*,j})h^{k+1/2}\big)
	,
\end{align}
and similar to the derivation of \eqref{eq:Gron-Ghsm2}, we also have
\begin{align}\label{eq:Gron-Ohsm2}
	\| (\hat X_{h,*}^{m+1})^{-1}\|_{W^{1,\infty}(\OhsM)}
	\leq
	C
	.
\end{align}

%

\section{Convergence of errors}
\label{sec:conv-err}

Taking $l=k$ in \eqref{eq:Gron-Ghsm} and \eqref{eq:Gron-Ohsm} and using the equivalence relation \eqref{eq:mu-nu-equiv},
\begin{align}
\nu_{*,m+1}^2
&\leq
C +
C_{\bar\mu_{m}, \bar\nu_{m}}
h^{-2k+1}\tau^2
+
C_{\bar\mu_{m}, \bar\nu_{m}}
\sum_{j=0}^{m}
\tau \mu_{*,j}^2 ,
\notag\\
\mu_{*,m+1}^2
&\leq
C +
C_{\bar\mu_{m}, \bar\nu_{m}}
h^{-2k+2}\tau^2
+
C_{\bar\mu_{m}, \bar\nu_{m}}
h^{2+2{\alpha(d)}}
\sum_{j=0}^{m}
\tau \nu_{*,j}^2
\notag
.
\end{align}
Applying Gr\"onwall's inequality and the stepsize constraint $\tau\leq Ch^k$ to the system above, we conclude
\begin{align}
\nu_{*,m+1}
\leq C_{\bar\mu_m, \bar\nu_m} 
\quad\mbox{and}\quad
\mu_{*,m+1}
\leq C
,
\notag
\end{align}
for $h\leq h_{\bar\mu_{*,m},\bar\nu_{*,m},\bar\mu_{m},\bar\nu_{m}}$.

Consequently, from \eqref{eq:Gron-Ghsm}--\eqref{eq:Gron-Ghsm2} and \eqref{eq:Gron-Ohsm}--\eqref{eq:Gron-Ohsm2} with $l=0,\cdots,k-1$, we also obtain the boundedness
\begin{align}
\nu_{m+1}
&\sim
\| \hat X_{h,*}^{m+1} \|_{W_h^{k-1,4}(\Ghso)}
+
\| \hat X_{h,*}^{m+1} \|_{W^{1,\infty}(\Ghso)}
+
\| (\hat X_{h,*}^{m+1})^{-1} \|_{W^{1,\infty}(\GhsM)}
\leq C ,
\notag
\\
\mu_{m+1}
&\sim
\| \hat X_{h,*}^{m+1} \|_{W_h^{k-1,4}(\Ohso)}
+
\| \hat X_{h,*}^{m+1} \|_{W^{1,\infty}(\Ohso)}
+
\| (\hat X_{h,*}^{m+1})^{-1} \|_{W^{1,\infty}(\OhsM)}
\leq C ,
\notag
\end{align}
for $h\leq h_{\bar\mu_{*,m},\bar\nu_{*,m},\bar\mu_{m},\bar\nu_{m}}$.


Conversely, we can improve the boundedness and the mesh size requirement to
\begin{align}
\mu_{*,m+1} + \nu_{*,m+1}
\leq C,
\notag
\end{align}
and 
\begin{align}
h\leq h_{C,C,C,C} .
\notag
\end{align}

%
Finally, we substitute the boundedness results above back into the main error estimate \eqref{eq:e-est} and get
\begin{align}\label{eq:eux-est-fin}
&\Big(\sum_{j=0}^{m} \tau \| e_u^j \|_{H^1(\hat\Omega_{h,*}^j)}^2\Big)^{1/2}
+
\max_{j=0,\cdots,m+1}
\| \hat e_x^j \|_{H^1(\hat\Gamma_{h,*}^j)}
\leq C
(\tau + h^{k-1/2})
,
\end{align}
for all $m=0,\cdots,[T/\tau]$.
This completes the induction step and recovers the induction hypothesis \eqref{eq:ind-hypo} at $t_{m+1}$.
Plugging in \eqref{eq:eux-est-fin} and Lemma \ref{lemma:dm} into \eqref{eq:ek-est-s}, we get
\begin{align}\label{eq:ek-est-fin}
\max_{j=0,\cdots,m}
\| \kappa_h^j - I_h(H^j n^j)^{-\ell} \|_{L^{2}(\hat\Gamma_{h,*}^j)}
\leq
C h^{-1}
\max_{j=0,\cdots,m}
\| e_\kappa^j \|_{H^{-1}(\hat\Gamma_{h,*}^j)}
\leq C
h^{-1}
(\tau + h^{k-1/2})
.
\end{align}
Then, plugging \eqref{eq:eux-est-fin}, \eqref{eq:ek-est-fin} and Lemma \ref{lemma:dm} into \eqref{eq:ep-est},
\begin{align}\label{eq:ep-est-fin}
&\Big(\sum_{j=0}^{m} \tau \| e_p^j \|_{L^2(\hat\Omega_{h,*}^j)}^2\Big)^{1/2}
\leq C
h^{-1/2}
(\tau + h^{k-1/2})
.
\end{align}
From \eqref{eq:Phi-approx4}--\eqref{eq:Phi-approx5}, we see that $\bar I_h (p^m\circ\Phi^m)$ differs from $I_h (p^m\circ\Phi^m)$ by an optimal consistency error.
The proof of the main theorem (Theorem \ref{thm:main}) is now complete.

\appendix

\section{Notation}
\label{section:notation}\label{sec:notation}
The following symbols associated with the framework of the projection error are frequently used in this article. They are similar to those introduced in \cite[Section 3.1]{BL24FOCM} and are listed below for the convenience of the readers. 

\begin{longtable}{p{1.2cm}p{10cm}}
	$\Omega^m/\Gamma^m$:
	& 
	The exact smooth domain/interface at time level $t=t_m$.\\
	
	$\Omega_h^m/\Gamma_h^m$:
	&
	The numerically computed domain/interface at time level $t=t_m$.\\
	
	%
	%
	
	$\Ohsm/\Ghsm$: 
	&
	The consistency domain/interface, defined via nodal projection, at time level $t=t_m$.\\
	
	$\Omega_{h,{\rm f}}^0/\Gamma_{h,{\rm f}}^0$: 
	&
	The flat consistency domain/interface, which has the same vertices as $\hat\Omega_{h,*}^0/\hat\Gamma_{h,*}^0$.\\
	
	$\Omega_{h,\#}^m/\Gamma_{h,\#}^m$: 
	&
	The consistency domain/interface, defined via interpolating the global smooth flow, at time level $t=t_m$.\\
	
	$\Gamma_{h,*}^m$: 
	&
	The consistency interface, defined via interpolating the local smooth flow, at time level $t=t_m$.\\
	
	$\Phi^m$: 
	&
	The Lenoir-type lifting operator, mapping $\Ohsm$ to $\Omega^m$.\\
	
	$X^{m}$: 
	&
	The global flow map from $\Omega^0$ to $\Omega^{m}$. \\
	
	$Y^{m}$: 
	&
	The local flow map from $\Omega^{m-1}$ to $\Omega^{m}$. \\
	
	$X_{h}^{m}$: 
	&
	The unique (up to nodal vector identification) finite element parametrization map whose graph is $\Omega_{h}^m$. \\
	
	$\hat X_{h,*}^{m}$: 
	&
	The unique (up to nodal vector identification) finite element parametrization map whose graph is $\hat\Omega_{h,*}^m$. \\
	
	$X_{h,\#}^{m}$: 
	&
	The unique (up to nodal vector identification) finite element parametrization map whose graph is $\Omega_{h,\#}^m$. \\
	
	$X_{h,*}^m$: 
	&
	The unique (up to nodal vector identification) finite element parametrization map whose graph is $\Gamma_{h,*}^m$. \\
	
	$\hat e_{x}^m$: 
	&
	The interface finite element error function for the position variable with nodal vector determined by $X_{h}^m - \hat X_{h,*}^m$.
	\\ 
	
	$E_h \hat e_{x}^m$: 
	&
	The bulk finite element error function for the position variable with nodal vector determined by $X_{h}^m - \hat X_{h,*}^m$.
	\\ 
	
	$e_{x,\#}^m$: 
	&
	The bulk finite element error function for the position variable with nodal vector determined by $X_{h}^m - X_{h,\#}^m$.\\ 
	
	$e_{x}^{m}$: 
	&
	The interface finite element error function for the position variable with nodal vector determined by $X_{h}^m - X_{h,*}^m$.
	\\ 
	
	$n^m$: 
	&
	The outward unit normal vector on $\Gamma^m$. \\
	
	$H^m$: 
	&
	The mean curvature on $\Gamma^m$. \\
	
	$n^m_*$: 
	&
	The outward unit normal vector of $\Gamma^m$ inversely lifted to a neighborhood of $\Gamma^m$ (including $\Ghsm$), i.e., $n^m_*:=n^m\circ a^m$. \\
	
	$n_h^m$: 
	&
	The outward unit normal vector on $\Gamma_{h}^m$. \\

	$N^m$: 
	&
	The normal projection operator $N^m=n^m (n^m)^\top$ on $\Gamma^m$.  \\
	
	$\Nsm$: 
	&
	The normal projection operator $\Nsm=n^m_* (n^m_*)^\top$.
	Thus $N_*^m$ is an extension of $N^m$ to a neighborhood of $\Gm$.
	\\
	
	$T^m$: 
	&
	The tangential projection operator $T^m=I - n^m (n^m)^\top$ on $\Gamma^m$. 
	\\
	
	$T_*^m$: 
	&
	The tangential projection operator $T_*^m=I - n^m_* (n^m_*)^\top$.
	Thus $T_*^m$ is an extension of $T^m$ to a neighborhood of $\Gm$.
	
\end{longtable}

\section{Calculus on moving domains}
\label{sec:surf_calc}

Given a $(d-1)$-dimensional smooth submanifold $\Gamma$ in $\R^{d}$ (with or without boundary) and a smooth function $u\in C^\infty(\Gamma)$, we denote by $\ud_i u, i = 1,\cdots,d$, the $i$th component of the surface gradient $\nabla_\Gamma u\in C^\infty(\Gamma;T\Gamma)$. Using the Einstein summation convention, the corresponding Leibniz rule, chain rule, integration-by-parts formula, commutator formulas, and the evolution equation of the normal vector are summarized below (see \cite[Lemma 5.1]{BL24FOCM} and references therein).
\begin{lemma}\label{lemma:ud}
	Let $\Gamma$ and $ \Gamma^\prime$ be two $(d-1)$-dimensional smooth submanifolds that are possibly open, such as smooth pieces of some parametrized finite element surfaces, and let $f, h \in C^\infty(\Gamma)$ and $g\in C^\infty(\Gamma^\prime; \Gamma)$ be given functions. Then the following results hold. 
	\begin{itemize}
		\item[1.]  $\ud_i(fh) = \ud_i f h + f\ud_i h$ on $\Gamma$.
		\item[2.] $\ud_i(f\circ g) = (\ud_j f\circ g)\, \ud_i g_j$ on $\Gamma'$.
		\item[3.] $\int_{\Gamma}f \ud_i h = -\int_{\Gamma}\ud_i f h + \int_{\Gamma}f h H n_i + \int_{\partial\Gamma}f h \mu_i$ where $n, \mu$ are the normal and co-normal (tangential) direction, respectively, and $H:=\ud_i n_i$ is the mean curvature, i.e. the trace of the second fundamental form.
		\item[4.] $\ud_i \ud_j f = \ud_j \ud_i f + n_i H_{jl} \ud_l f -  n_j H_{il} \ud_l f$, where $H_{ij} := \ud_i n_j = \ud_j n_i$.
		\item[5.] If $\Gamma$ evolves under the velocity field $v$, and $G_T := \bigcup_{t\in [0, T]}\Gamma(t) \times \{t\}$, then 
		$$\md(\ud_i f) = \ud_i (\md f )- (\ud_i v_j - n_i n_l \ud_j v_l)\ud_j f \quad\forall\, f \in C^2(G_T) ,$$
		where $\md$ denotes the material derivative with respect to $v$.
		\item[6.] If $f, h \in C^2(G_T)$ then 
		$$\frac{\d}{\d t}\int_{\Gamma} f h = \int_{\Gamma} \md f h + \int_{\Gamma} f\md h + \int_{\Gamma} f h (\nabla_\Gamma\cdot v).$$
		The divergence is defined as $\nabla_\Gamma\cdot v := \ud_i v_i$, which coincides with the intrinsic divergence on $\Gamma$ if $v$ is a tangential vector field. 
	\end{itemize}
\end{lemma}
On a moving bulk domain, Items 5 and 6 in Lemma \ref{lemma:ud} immediately reduce to the following result.
\begin{lemma}\label{lemma:udb}
	Given a bulk domain $\Xi(t)\subset \R^d$ which is deforming under a velocity field $v$, let $\md$ be the material derivative along $v$. Then, for any sufficiently smooth functions $f$ and $h$ defined on $\bigcup_{t\in [0, T]}\Xi(t) \times \{t\}$, we have the pointwise commutator identity
	\begin{align}
		\md(\nabla_i f) = \nabla_i (\md f )- \nabla_i v_j \nabla_j f , \notag
	\end{align}
	where $\nabla_i f := \frac{\partial}{\partial x_i} f, i=1,\cdots,d$ is the $i$th component of the Euclidean gradient, and
	\begin{align}
		\frac{\d}{\d t}\int_{\Xi} f h = \int_{\Xi} \md f h + \int_{\Xi} f\md h + \int_{\Xi} f h (\nabla\cdot v) . \notag
	\end{align}
\end{lemma}

\section{Super-approximation estimates}
\label{sec:super}

Throughout this section, we denote by $m$ and $p$ two arbitrary numbers in $0,1,\cdots,[T/\tau]$ and $[1,\infty]$ respectively.
In the framework of the projection error, the following super-approximation results are standard (cf. \cite[Section 3.5]{BL24FOCM} and \cite[Lemma 3.6]{GLW22}).
\begin{lemma}\label{lemma:super_conv}
	The following estimates hold for any piecewise smooth function $f$ and finite element functions $\phi_h,v_h, w_h\in S_h(\Ghsm)$: 
	\begin{align*}
		\| (1 - I_h)(f \phi_h) \|_{L^p(\Ghsm)} &\lesssim \| f \|_{W_h^{k+1,\infty}(\Ghsm)} h \| \phi_h \|_{L^{p}(\Ghsm)} , \notag\\
		\| \nabla_\Ghsm (1 - I_h)(f \phi_h) \|_{L^p(\Ghsm)} &\lesssim \| f \|_{W_h^{k+1,\infty}(\Ghsm)} h \| \phi_h \|_{W^{1,p}(\Ghsm)} ,\\
		\| (1 - I_h)(v_h w_h) \|_{L^p(\Ghsm)} &\lesssim  h^2 \| v_h \|_{W^{1,q}(\Ghsm)} \| w_h \|_{W^{1,r}(\Ghsm)} ,\\
		\| \nabla_{\Ghsm}(1 - I_h)(v_h w_h) \|_{L^p(\Ghsm)} &\lesssim  h \| v_h \|_{W^{1,q}(\Ghsm)} \| w_h \|_{W^{1,r}(\Ghsm)} ,
	\end{align*}
	for any $1/p=1/q+1/r$.
\end{lemma}
As a direct application of Lemma \ref{lemma:super_conv}, we have
\begin{lemma}\label{lemma:T<=N2}
	The following estimates hold:
	\begin{align} 
		\| \Tsm \ehxm \|_{L^p(\Ghsm)}
		&\lesssim
		h\| \ehxm \|_{L^p(\Ghsm)}  , \notag\\
		\| \Tsm \ehxm \|_{W^{1,p}(\Ghsm)}
		&\lesssim
		h\| \ehxm \|_{W^{1,p}(\Ghsm)} 
		\notag .
	\end{align}
\end{lemma}
Besides, a nonlinear version of Lemma \ref{lemma:super_conv} can be proved using the chain rule and product rule of differentiation (cf. Items 1 and 2 in Lemma \ref{lemma:ud}).
\begin{lemma}\label{lemma:super_conv-nonlinear}
	Given a function $f\in W^{k+2,\infty}(D)$, defined on some open bulk region $D\subset \R^d$, and any vector-valued finite element functions $\phi_h, \psi_h \in S_h(\Ghso)^d$, whose ranges are contained in $D$, we have
	\begin{align*}
		\| (1 - I_h)(f\circ\phi_h - f\circ\psi_h) \|_{L^p(\Ghso)} &\leq C h \| \phi_h - \psi_h \|_{L^{p}(\Ghso)} , 
		\notag\\
		\| \nabla_\Ghso (1 - I_h)(f\circ\phi_h - f\circ\psi_h) \|_{L^p(\Ghso)} &\leq C h \| \phi_h - \psi_h \|_{W^{1,p}(\Ghso)}  . 
	\end{align*}
	Consequently, from the triangle inequality and Lipschitz continuity, it holds that
	\begin{align*}
		\| I_h(f\circ\phi_h - f\circ\psi_h) \|_{L^p(\Ghso)} &\leq C \| \phi_h - \psi_h \|_{L^{p}(\Ghso)} , 
		\notag\\
		\| I_h (f\circ\phi_h - f\circ\psi_h) \|_{W^{1,p}(\Ghso)} &\leq C \| \phi_h - \psi_h \|_{W^{1,p}(\Ghso)}  .
	\end{align*}
	The constants $C$ above depend on $\| f \|_{W^{k+2,\infty}(D)}$, $\| \phi_h \|_{W^{1,\infty}(\Ghso)}$ and $\| \psi_h \|_{W^{1,\infty}(\Ghso)}$.
\end{lemma}
\begin{proof}
	For simplicity of presentation, we only present the proof for the $L^p$ stability estimate. The $W^{1,p}$ estimate can be established analogously.
	
	From the chain rule,
	\begin{align}
		\nabla_\Ghso^{k+1} (f\circ \phi_h - f\circ \psi_h)
		&=
		\nabla_D^{k+1} f \circ \phi_h (\nabla_\Ghso\phi_h)^{k+1} 
		-
		\nabla_D^{k+1} f \circ \psi_h (\nabla_\Ghso\psi_h)^{k+1}
		+ \cdots
		\notag\\
		&\quad
		+
		\nabla_D^{2} f \circ \phi_h \nabla_\Ghso \phi_h \nabla_\Ghso^{k} \phi_h
		-
		\nabla_D^{2} f \circ \psi_h \nabla_\Ghso \psi_h \nabla_\Ghso^{k} \psi_h . \notag
	\end{align}
	Here, the contraction rules for gradients are conventional; cf. Fa\'a di Bruno's formula.
	The first difference on the right-hand side can be bounded as follows
	\begin{align}
		&\| \nabla_D^{k+1} f \circ \phi_h (\nabla_\Ghso\phi_h)^{k+1}
		-
		\nabla_D^{k+1} f \circ \psi_h (\nabla_\Ghso\psi_h)^{k+1} \|_{L^p(\Ghso)}
		\notag\\
		&\leq
		\| ( \nabla_D^{k+1} f \circ \phi_h - \nabla_D^{k+1} f \circ \psi_h) (\nabla_\Ghso\phi_h)^{k+1}
		\|_{L^p(\Ghso)}
		\notag\\
		&\quad
		+
		\|
		\nabla_D^{k+1} f \circ \psi_h \big((\nabla_\Ghso\psi_h)^{k+1} - (\nabla_\Ghso\phi_h)^{k+1}\big) \|_{L^p(\Ghso)}
		\notag\\
		&\leq
		C
		\| f \|_{W^{k+2,\infty}(D)}
		\| \phi_h - \psi_h \|_{L^p(\Ghso)}
		\| (\nabla_\Ghso\phi_h)^{k+1} \|_{L^\infty(\Ghso)}
		\notag\\
		&\quad
		+
		C
		\| f \|_{W^{k+1,\infty}(D)}
		\| \nabla_\Ghso (\phi_h - \psi_h) \|_{L^p(\Ghso)}
		\notag\\
		&\qquad
		\times
		(
		\| (\nabla_\Ghso\phi_h)^k \|_{L^\infty(\Ghso)}
		+
		\| (\nabla_\Ghso\psi_h)^k \|_{L^\infty(\Ghso)}
		)
		\notag\\
		&\leq
		C
		(1 + h^{-1})
		\| \phi_h - \psi_h \|_{L^p(\Ghso)}, \notag
	\end{align}
	and similarly for the last difference,
	\begin{align}
		&\| \nabla_D^{2} f \circ \phi_h \nabla_\Ghso \phi_h \nabla_\Ghso^{k} \phi_h
		-
		\nabla_D^{2} f \circ \psi_h \nabla_\Ghso \psi_h \nabla_\Ghso^{k} \psi_h \|_{L^p(\Ghso)}
		\notag\\
		&\leq
		C
		\| f \|_{W^{3,\infty}(D)}
		\| \phi_h - \psi_h \|_{L^p(\Ghso)}
		\notag\\
		&\qquad
		\times
		(
		\| \nabla_\Ghso \phi_h \nabla_\Ghso^{k} \phi_h \|_{L^\infty(\Ghso)}
		+
		\| \nabla_\Ghso \psi_h \nabla_\Ghso^{k} \psi_h \|_{L^\infty(\Ghso)}
		)
		\notag\\
		&\quad
		+
		C
		\| f \|_{W^{2,\infty}(D)}
		\| \nabla_\Ghso \phi_h - \nabla_\Ghso \psi_h \|_{L^p(\Ghso)}
		\notag\\
		&\qquad
		\times
		(
		\| \nabla_\Ghso^k \phi_h \|_{L^\infty(\Ghso)}
		+
		\| \nabla_\Ghso^k \psi_h \|_{L^\infty(\Ghso)}
		)
		\notag\\
		&\quad
		+
		C
		\| f \|_{W^{2,\infty}(D)}
		\| \nabla_\Ghso^k \phi_h - \nabla_\Ghso^k \psi_h \|_{L^p(\Ghso)}
		\notag\\
		&\qquad
		\times
		(
		\| \nabla_\Ghso \phi_h \|_{L^\infty(\Ghso)}
		+
		\| \nabla_\Ghso \psi_h \|_{L^\infty(\Ghso)}
		)
		\notag\\
		&\leq
		C
		(1 + h^{-k})
		\| \phi_h - \psi_h \|_{L^p(\Ghso)}, \notag
	\end{align}
	where, in the last inequality, we have applied the inverse inequality on $\Ghso$ $k$ times.
	
	Finally, from interpolation error estimate (Lemma \ref{lemma:Ih}), we obtain
	\begin{align}
		&\| (1 - I_h) (f\circ\phi_h - f\circ\psi_h) \|_{L^p(\Ghso)} 
		\notag\\
		&\leq
		C h^{k+1}
		\| f\circ\phi_h - f\circ\psi_h \|_{W_h^{k+1,p}(\Ghso)} 
		\notag\\
		&\leq
		C
		h^{k+1} (1 + h^{-1} + \cdots + h^{-k})
		\| \phi_h - \psi_h \|_{L^p(\Ghso)}
		\notag\\
		&\leq
		C
		h
		\| \phi_h - \psi_h \|_{L^p(\Ghso)}
		, \notag
	\end{align}
	where the constant $C$ depends on $\| f \|_{W^{k+2,\infty}(D)}$, $\| \phi_h \|_{W^{1,\infty}(\Ghso)}$ and $\| \psi_h \|_{W^{1,\infty}(\Ghso)}$.
	
	The proof is now complete.
\end{proof}

Based on the approximation property of Gauss--Lobatto quadrature rule, we have
\begin{lemma}\label{lemma:super_conv2}
	When $d=2$, let $f$ be a function which is smooth on every element $K$ of $\Ghsm$, and assume that the pull-back function $f\circ F_K $ vanishes at all the Gauss--Lobatto points of the flat segment $K_{\rm f}^0$ for every element $K$ of $\Ghsm$. Then the following estimate holds: 
	\begin{align}
		\Big|\int_\Ghsm f \d\xi \Big| 
		\lesssim h^{2k} \| f \|_{W^{2k,1}_h(\Ghsm)} , \notag
	\end{align}
	where $\|\cdot\|_{W^{2k,1}_h(\Ghsm)}$ denotes the piecewise $W^{2k,1}$ norm. 
\end{lemma}
As a corollary of Lemma \ref{lemma:super_conv2}, we also have
\begin{lemma}\label{Lemma-GLW}
	When $d=2$,
	for a smooth function $f$ on $\Gm$, the following estimate holds if the interpolation nodes of $I_h$ coincide with the Gauss--Lobatto points: 
	\begin{align*}
		\Big| \int_{\Gm} \nabla_{\Gm} (f - (I_h f^{-\ell})^\ell) \cdot  \nabla_{\Gm}  \phi_h^\ell \Big|
		\lesssim 
		h^{k+1} \|f\|_{H^{2k}(\Gm)} \|\phi_h\|_{H^1(\Ghsm)} 
		\quad\forall\,\phi_h\in S_h(\Ghsm) . 
	\end{align*}
\end{lemma}

\section{Proof of \eqref{eq:uhm-epm}}
\label{sec:H3/2}

We divide the proof of \eqref{eq:uhm-epm} into the following four steps.

Step 1: Using the moving-domain calculus formulas in Lemma \ref{lemma:udb}, we obtain
\begin{align}\label{eq:step1}
	&\int_\Ohm \nabla_\Ohm\cdot\uhm \epm
	-
	\int_\Ohsm \nabla_\Ohsm\cdot\uhm \epm
	\notag\\
	&=
	\int_0^1\frac{\d}{\d\theta}\bigg(\int_{\hat\Omega_{h,\theta}^m} \nabla_{\hat\Omega_{h,\theta}^m}\cdot\uhm \epm \bigg)\d\theta
	\notag\\
	&=
	\int_0^1
	\bigg(
	\int_{\hat\Omega_{h,\theta}^m}
	-{\rm tr}\big(\nabla_{\hat\Omega_{h,\theta}^m} E_h\ehxm \nabla_{\hat\Omega_{h,\theta}^m} \uhm\big)\epm
	+
	\nabla_{\hat\Omega_{h,\theta}^m}\cdot\uhm \epm (\nabla_{\hat\Omega_{h,\theta}^m} \cdot E_h\ehxm) 
	\bigg)
	\d\theta
	\notag\\
	&=:
	\int_{\hat\Omega_{h,*}^m}
	-{\rm tr}\big(\nabla_{\hat\Omega_{h,*}^m} E_h\ehxm \nabla_{\hat\Omega_{h,*}^m} \uhm\big)\epm
	+
	\nabla_{\hat\Omega_{h,*}^m}\cdot\uhm \epm (\nabla_{\hat\Omega_{h,*}^m} \cdot E_h\ehxm)
	+
	L_1^m ,
\end{align}
where $\hat\Omega_{h,\theta}^m$ is defined in Section \ref{sec:err-ind-hypo} and $L_1^m$ can be handled by repeating the argument of fundamental theorem of calculus, which will produce another factor of $E_h\ehxm$,
\begin{align}
	| L_1^m |
	&=:
	\bigg|
	\int_0^1 \bigg(\int_{\hat\Omega_{h,\theta}^m}
	-{\rm tr}\big(\nabla_{\hat\Omega_{h,\theta}^m} E_h\ehxm \nabla_{\hat\Omega_{h,\theta}^m} \uhm\big)\epm
	+{\rm tr}\big(\nabla_{\hat\Omega_{h,*}^m} E_h\ehxm \nabla_{\hat\Omega_{h,*}^m} \uhm\big)\epm
	\bigg)\d\theta
	\notag\\
	&\quad
	+
	\int_0^1
	\bigg(\int_{\hat\Omega_{h,*}^m}
	\nabla_{\hat\Omega_{h,\theta}^m}\cdot\uhm \epm (\nabla_{\hat\Omega_{h,\theta}^m} \cdot E_h\ehxm)
	-
	\nabla_{\hat\Omega_{h,*}^m}\cdot\uhm \epm (\nabla_{\hat\Omega_{h,*}^m} \cdot E_h\ehxm)
	\bigg) \d\theta
	\bigg|
	\notag\\
	&\lesssim
	(1 + \| \nabla_\Ohsm \eum \|_{L^\infty(\Ohsm)})
	\| \nabla_\Ohsm E_h\ehxm \|_{L^\infty(\Ohsm)}
	\| \nabla_\Ohsm E_h\ehxm \|_{L^2(\Ohsm)}
	\| \epm \|_{L^2(\Ohsm)} . \notag
\end{align}

Step 2: 
For simplicity, we only treat the second term on the right-hand side of \eqref{eq:step1}, as the first term can be handled analogously. 
Consider a linear bulk mesh transport from $\Omega_{h,\#}^m:=I_h(X^m\circ\Phi^0)$ to $\Ohsm$ via the finite element velocity $e_{x,\#}^m - E_h \ehxm$. Then, we have
\begin{align}
	&\int_{\hat\Omega_{h,*}^m}
	\nabla_{\hat\Omega_{h,*}^m}\cdot\uhm \epm (\nabla_{\hat\Omega_{h,*}^m} \cdot E_h\ehxm)
	\notag\\
	&=
	\int_{\Omega_{h,\#}^m}
	\nabla_{\Omega_{h,\#}^m}\cdot I_h(u^m\circ X^m\circ\Phi^0) \epm (\nabla_{\Omega_{h,\#}^m} \cdot E_h\ehxm)
	\notag\\
	&\quad
	+
	\int_{\Omega_{h,\#}^m}
	\nabla_{\Omega_{h,\#}^m}\cdot(\uhm - I_h(u^m\circ X^m\circ\Phi^0)) \epm (\nabla_{\Omega_{h,\#}^m} \cdot E_h\ehxm)
	\notag\\
	&\quad
	+
	\int_{\hat\Omega_{h,*}^m}
	\nabla_{\hat\Omega_{h,*}^m}\cdot\uhm \epm (\nabla_{\hat\Omega_{h,*}^m} \cdot E_h\ehxm)
	-
	\int_{\Omega_{h,\#}^m}
	\nabla_{\Omega_{h,\#}^m}\cdot\uhm \epm (\nabla_{\Omega_{h,\#}^m} \cdot E_h\ehxm)
	\notag\\
	&=:
	\int_{\Omega_{h,\#}^m}
	\nabla_{\Omega_{h,\#}^m}\cdot I_h(u^m\circ X^m\circ\Phi^0) \epm (\nabla_{\Omega_{h,\#}^m} \cdot E_h\ehxm)
	+L_2^m + L_3^m . \notag
\end{align}
In view of the decomposition $\uhm - I_h(u^m\circ X^m\circ\Phi^0) = \eum + \big(I_h(u^m\circ \Phi^m) - I_h(u^m\circ X^m\circ\Phi^0)\big)$, the second term on the right-hand side is similar to the $E_2^m$ term introduced in Section \ref{sec:traj-est}. We proceed to estimate $L_2^m$ as follows
\begin{align}
	| L_2^m |
	&\lesssim
	\Big(
	(1 + \musm) h^k
	+
	(1 + \nusm) h^{k+1/2}
	+
	\| e_{x,\#}^m  \|_{H^{1}(\Ohso)}
	+
	\| E_h \ehxm \|_{H^{1}(\Ohso)}
	\notag\\
	&\quad
	+
	\| \nabla_\Ohsm \eum \|_{L^2(\Ohsm)}
	\Big)
	\times
	\| \nabla_\Ohsm E_h\ehxm \|_{L^\infty(\Ohsm)}
	\| \epm \|_{L^2(\Ohsm)} . \notag
\end{align}
The term $L_3^m$ can be bounded using a standard argument based on the fundamental theorem of calculus:
\begin{align}
	| L_3^m |
	&\lesssim
	(1 + \| \nabla_\Ohsm \eum \|_{L^\infty(\Ohsm)})
	\| \nabla_\Ohsm (E_h\ehxm - e_{x,\#}^m) \|_{L^\infty(\Ohsm)}
	\notag\\
	&\quad
	\times
	\| \nabla_\Ohsm E_h\ehxm \|_{L^2(\Ohsm)}
	\| \epm \|_{L^2(\Ohsm)} . \notag
\end{align}

Step 3: We apply a change of variables via the pullback of $I_h(X^m\circ\Phi^0)\circ (\Phi^0)^{-1}: \Omega\rightarrow \Omega_{h,\#}^m$,
\begin{align}
	&\int_{\Omega_{h,\#}^m}
	\nabla_{\Omega_{h,\#}^m}\cdot I_h(u^m\circ X^m\circ\Phi^0) \epm (\nabla_{\Omega_{h,\#}^m} \cdot E_h\ehxm)
	\notag\\
	&=
	\int_{\Omega}
	{\rm tr}\Big((J_h^{-1})^\top \nabla_{\Omega} I_h(u^m\circ X^m\circ\Phi^0) \circ(\Phi^0)^{-1}\Big) \epm \circ (\Phi^0)^{-1} 
	\notag\\
	&\qquad
	\times
	{\rm tr}\Big( (J_h^{-1})^\top \nabla_{\Omega} (E_h\ehxm \circ (\Phi^0)^{-1})\Big)
	{\rm det}(J_h)
	\notag\\
	&=:
	\int_{\Omega}
	{\rm tr}\Big( (J^{-1})^\top \nabla_{\Omega} (u^m\circ X^m)\Big) 
	\epm \circ (\Phi^0)^{-1} 
	{\rm tr}\Big( (J^{-1})^\top \nabla_{\Omega} (E_h\ehxm \circ (\Phi^0)^{-1})\Big)
	{\rm det}(J)
	+
	L_4^m ,
	\notag
\end{align}
where the Jacobians are defined as $J := \nabla_\Omega X^m$ and $J_h := \nabla_\Omega (I_h(X^m\circ\Phi^0)\circ (\Phi^0)^{-1})$.
The term $L_4^m$ consists of errors coming from $J_h - J$, $J_h^{-1} - J^{-1}$, $(1-I_h)(\Phi^0)^{-1}$ and $(1-I_h)(u^m\circ X^m)$. First, by interpolation error estimates and the approximation property of $\Phi^0$, we have
\begin{align}
	\| J_h - J \|_{L^\infty(\Omega)} 
	+ 
	\| (1-I_h)(u^m\circ X^m) \|_{L^\infty(\Omega)} 
	+ 
	\| (1-I_h)(\Phi^0)^{-1} \|_{L^\infty(\Omega)} 
	\leq C (h^k + h^{k+1}), \notag
\end{align}
where $I_h$ should be interpreted as the lifted interpolation operator on $\Omega$ (cf. \cite[p. 12]{BL24}).
Second, note that
\begin{align}
	J_h^{-1}
	= \big(J + (J_h - J)\big)^{-1}
	=\big(I \underbrace{+ J^{-1}(J_h - J)}_{=:-e_J}\big)^{-1} J^{-1} = (I + e_J + e_J^2 + \cdots) J^{-1}. \notag
\end{align}
Consequently,
\begin{align}
	\| J_h^{-1} - J^{-1} \|_{L^\infty(\Omega)} \leq C \| e_J \|_{L^\infty(\Omega)} \leq C h^k. \notag
\end{align}
Putting these estimates together, we obtain
\begin{align}
	| L_4^m |
	&\lesssim
	h^k
	\| \nabla_\Ohsm E_h\ehxm \|_{L^2(\Ohsm)}
	\| \epm \|_{L^2(\Ohsm)} . \notag
\end{align}

Step 4:
Finally, for any smooth function $f$ in $\Omega$, by the definition of $E_h\ehxm$ (Eq. \eqref{eq:Eh}),
\begin{align}
	&\int_{\Omega}
	f \epm \circ (\Phi^0)^{-1} \nabla_{\Omega} (E_h\ehxm \circ (\Phi^0)^{-1})
	\notag\\
	&=\int_{\Omega}
	f \epm \circ (\Phi^0)^{-1} \nabla_{\Omega} (\eta E (\ehxm\circ (a^0|_\Ghso)^{-1}))
	\notag\\
	&\quad
	-
	\int_{\Omega}
	f \epm \circ (\Phi^0)^{-1} \nabla_{\Omega} \Big((1 - I_h^{\rm SZ}) \Big[\big(\eta E (\ehxm\circ (a^0|_\Ghso)^{-1})\big)\circ \Phi^0 \Big] \circ (\Phi^0)^{-1}\Big)
	\notag\\
	&=: L_5^m + L_6^m . \notag
\end{align}
Since $\eta E (\ehxm\circ (a^0|_\Ghso)^{-1})\in H_0^{3/2}(\Omega)$, H\"older's inequality and the mapping property of $E$ (cf. Section \ref{sec:Eh}) yield
\begin{align}
	| L_5^m |
	&\lesssim
	\| \ehxm \|_{H^{1/2}(\Ghsm)}
	\| \epm \|_{H^{-1/2}(\Ohsm)} . \notag
\end{align}
For $L_6^m$, we use the domain perturbation estimates \eqref{eq:Phi-approx4}--\eqref{eq:Phi-approx5}, interpolation error estimate and the inverse inequality to get
\begin{align}
	| L_6^m |
	&\lesssim
	h^k
	\| \ehxm \|_{H^{1/2}(\Ghsm)}
	\| \epm \|_{L^2(\Ohsm)} 
	+
	h^{1/2}
	\| \ehxm \|_{H^{1}(\Ghsm)}
	\| \epm \|_{L^2(\Ohsm)}
	. \notag
\end{align}

Summing up all estimates for $L_i^m, i=1,\cdots, 6$, and using the induction hypothesis \eqref{eq:ind-hypo} and the inverse inequality, we complete the proof of \eqref{eq:uhm-epm}.

\section{Proof of Lemma \ref{lemma:NT_stab_ref}}
\label{sec:appndix_tan_stab}

Given two finite element functions $f_h,g_h\in S_h(\Ghsm)$, it holds that
\begin{align}\label{eq:tan_stab1}
	&\Big| \int_{\Ghsm} \nabla_{\Ghsm} I_h \Nsm f_h \cdot  \nabla_{\Ghsm}  I_h \Tsm g_h \Big| 
	\notag\\ 
	&= \Big| \ \Big( \int_{\Ghsm} \nabla_{\Ghsm} I_h \Nsm f_h \cdot  \nabla_{\Ghsm}  I_h \Tsm g_h
	- \int_{\Gm} \nabla_{\Gm} (I_h \Nsm f_h)^{\ell} \cdot  \nabla_{\Gm}  (I_h \Tsm g_h)^{\ell} \Big)
	\notag\\ 
	&\quad+ \int_{\Gm} \nabla_{\Gm} (I_h \Nsm f_h)^{\ell} \cdot  \nabla_{\Gm}  (I_h \Tsm g_h)^{\ell} \Big|
	\notag\\ 
	&\lesssim h^{-d/2+1/2}(1 + \nusm) h^{k+1} \| \nabla_{\Ghsm} I_h \Nsm f_h \|_{L^2(\Ghsm)}  \|  \nabla_{\Ghsm}  I_h \Tsm g_h \|_{L^2(\Ghsm)} 
	\notag\\ 
	&\quad +\Big| \int_{\Gm} \nabla_{\Gm} (I_h \Nsm f_h)^{\ell} \cdot  \nabla_{\Gm} ( I_h \Tsm g_h)^{\ell} \Big|
	,
\end{align}
where we have used Lemma \ref{lemma:geo-perturb} in the last inequality.
Using the super-approximation (Lemma \ref{lemma:super_conv} with $(p,q,r)=(2,2,\infty)$ therein), the second term on the right-hand side can be treated as follows:
\begin{align}\label{eq:tan_stab212}
	&\Big| \int_{\Gm} \nabla_{\Gm} (I_h \Nsm f_h)^{\ell} \cdot  \nabla_{\Gm}  (I_h \Tsm  g_h)^{\ell} \Big|
	\notag\\ 
	&= \Big| \int_{\Gm} \nabla_{\Gm} (\Nsm f_h)^{\ell} \cdot  \nabla_{\Gm} (\Tsm g_h)^{\ell}
	- \int_{\Gm} \nabla_{\Gm} ((1 - I_h)\Nsm f_h)^{\ell} \cdot  \nabla_{\Gm} (I_h \Tsm g_h)^{\ell}
	\notag\\ 
	&\quad- \int_{\Gm} \nabla_{\Gm} (I_h\Nsm f_h)^{\ell} \cdot  \nabla_{\Gm} ((1 - I_h) \Tsm g_h)^{\ell}
	\notag\\ 
	&\quad- \int_{\Gm} \nabla_{\Gm}  ((1 - I_h) \Nsm f_h)^{\ell} \cdot  \nabla_{\Gm} ( (1 - I_h) \Tsm  g_h)^{\ell} \Big|
	\notag\\ 
	&\lesssim \Big| \int_{\Gm} \nabla_{\Gm} (\Nsm f_h)^{\ell} \cdot  \nabla_{\Gm} (\Tsm g_h)^{\ell} \Big|
	+ h \| f_h \|_{H^1(\Ghsm)}  \| \nabla_\Ghsm I_h \Tsm g_h \|_{L^2(\Ghsm)} \notag\\
	&\quad+ h \| \nabla_{\Ghsm} I_h\Nsm f_h  \|_{L^2(\Ghsm)} \| g_h \|_{H^1(\Ghsm)} 
	+ h^2 \| f_h \|_{H^1(\Ghsm)} \|  g_h \|_{H^1(\Ghsm)}
	.
\end{align}

Furthermore, using the product rule (Item 1, Lemma \ref{lemma:ud}),
\begin{align*}
	& \Big| \int_{\Gm} \nabla_{\Gm} (N^m f_h^{\ell}) \cdot  \nabla_{\Gm} (T^m g_h^{\ell}) \Big|
	\notag\\ 
	&= \Big| \int_{\Gm} (\nabla_{\Gm} N^m) f_h^{\ell} \cdot  (\nabla_{\Gm} T^m) g_h^{\ell}
	+ \int_{\Gm} N^m \nabla_{\Gm} f_h^{\ell} \cdot T^m \nabla_{\Gm} g_h^{\ell}
	\notag\\ 
	&\quad+ \int_{\Gm} (\nabla_{\Gm} N^m) f_h^{\ell} \cdot T^m  \nabla_{\Gm} g_h^{\ell}
	+ \int_{\Gm} N^m \nabla_{\Gm} f_h^{\ell} \cdot  (\nabla_{\Gm} T^m) g_h^{\ell} \Big| ,
\end{align*}
where the second term on the right-hand side is zero due to the orthogonality between the projections $N^m$ and $T^m$. For the last term on the right-hand side, we can transfer the gradient on $f_h^{\ell}$ to $g_h^{\ell}$ via integration by parts. By symmetry, we can also shift the gradient on $g_h^{\ell}$ to $f_h^{\ell}$ in the second-to-last term. 
This leads to the following estimate:
\begin{align}\label{eq:tan_stab51}
	&\Big| \int_{\Gm} \nabla_{\Gm} (N^m f_h^{\ell}) \cdot  \nabla_{\Gm} (T^m g_h^{\ell}) \Big| \notag\\ 
	&\lesssim \min\Big\{\| f_h \|_{L^2(\Ghsm)}  \| g_h \|_{H^1(\Ghsm)},
	\| f_h \|_{H^1(\Ghsm)}  \| g_h \|_{L^2(\Ghsm)}
	\Big\} .
\end{align}
The proof is complete by combining \eqref{eq:tan_stab1}--\eqref{eq:tan_stab51}.


\section{Proof of Lemma \ref{lemma:e-convert}}
\label{sec:e-convert}

For simplicity, we only prove that the conversion of $H^1$ semi-norm is stable. The proof for $L^2$ norm part is easier and therefore omitted.

To show that converting $\| \nabla_\Ghsm \exM \|_{L^2(\Ghsm)}^2$ to $\| \nabla_\GhsM \ehxM \|_{L^2(\GhsM)}^2$ is stable, we decompose their difference into the following four parts: 
\begin{align}
	&\| \nabla_\GhsM \ehxM \|_{L^2(\GhsM)}^2 
	- 
	\| \nabla_\Ghsm \exM \|_{L^2(\Ghsm)}^2
	\notag\\
	&= 
	\| \nabla_\GhsM \ehxM \|_{L^2(\GhsM)}^2  
	- 
	\| \nabla_\Ghsm \ehxM \|_{L^2(\Ghsm)}^2 
	\notag\\
	&\quad
	+ 
	\| \nabla_\Ghsm \ehxM \|_{L^2(\Ghsm)}^2 
	- 
	\| \nabla_\Ghsm I_h \Nsm \exM \|_{L^2(\Ghsm)}^2
	- 
	\| \nabla_\Ghsm I_h \Tsm \exM \|_{L^2(\Ghsm)}^2
	\notag\\
	&\quad
	-
	2 \int_\Ghsm \nabla_\Ghsm I_h\Nsm \exM \cdot \nabla_\Ghsm I_h\Tsm \exM
	\notag\\
	&\leq
	\| \nabla_\GhsM \ehxM \|_{L^2(\GhsM)}^2  
	- 
	\| \nabla_\Ghsm \ehxM \|_{L^2(\Ghsm)}^2 
	\notag\\
	&\quad
	+
	\| \nabla_\Ghsm I_h( \NsM\circ \hat X_{h,*}^{m+1} \exM) \|_{L^2(\Ghsm)}^2
	- 
	\| \nabla_\Ghsm I_h \Nsm \exM \|_{L^2(\Ghsm)}^2
	\notag\\
	&\quad
	+ 
	\| \nabla_\Ghsm \ehxM \|_{L^2(\Ghsm)}^2 
	- 
	\| \nabla_\Ghsm I_h (\NsM\circ \hat X_{h,*}^{m+1} \exM) \|_{L^2(\Ghsm)}^2
	\notag\\
	&\quad
	-
	2 \int_\Ghsm \nabla_\Ghsm I_h\Nsm \exM \cdot \nabla_\Ghsm I_h\Tsm \exM
	\notag\\
	&=: M_1^m + M_2^m + M_3^m + M_4^m .
	\notag
\end{align}
Using the fundamental theorem of calculus and the norm equivalence of $\Ghsm$ and $\hat\Gamma_{h,*}^{m+1}$ (cf. Lemma \ref{lemma:hatX-W1inf}), we derive
\begin{align}\label{eq:M1}
	M_1^m
	&=  
	\| \nabla_\GhsM \ehxM \|_{L^2(\GhsM)}^2  
	- 
	\| \nabla_\Ghsm \ehxM \|_{L^2(\Ghsm)}^2   \notag\\
	&\lesssim \| \nabla_\Ghsm (\hat X_{h,*}^{m+1} - \hat X_{h,*}^m) \|_{L^\infty(\Ghsm)} \| \nabla_\Ghsm \ehxM \|_{L^2(\Ghsm)}^2 
	\notag\\
	&\lesssim \tau \| \nabla_\Ghsm \ehxM \|_{L^2(\Ghsm)}^2
	+
	\tau \Big\| \frac{\exM - \ehxm}{\tau} \Big\|_{H^{1/2}(\Ghsm)} \| \nabla_\Ghsm \ehxM \|_{L^2(\Ghsm)}
	,
\end{align}
where we have applied Lemma \ref{lemma:hatX-W1inf} and \ref{lemma:e-diff-W1inf} in the last line. Analogous to $M_1^m$,
\begin{align}\label{eq:M2}
	M_2^m
	&\lesssim 
	\| \nabla_\Ghsm I_h(\NsM\circ \hat X_{h,*}^{m+1} - \Nsm \circ \hat X_{h,*}^{m}) \|_{L^\infty(\Ghsm)}
	\| \nabla_\Ghsm \ehxM \|_{L^2(\Ghsm)}^2
	\notag\\
	&\lesssim 
	\tau
	\| \nabla_\Ghsm \ehxM \|_{L^2(\Ghsm)}^2
	+
	\tau \Big\| \frac{\exM - \ehxm}{\tau} \Big\|_{H^{1/2}(\Ghsm)} \| \nabla_\Ghsm \ehxM \|_{L^2(\Ghsm)}
	.
\end{align}
Using the geometric relations \eqref{eq:geo_rel_1}--\eqref{eq:geo_rel_2}, we derive
\begin{align}\label{eq:M3}
	M_3^m 
	&= 
	\| \nabla_\Ghsm \ehxM \|_{L^2(\Ghsm)}^2 
	- 
	\| \nabla_\Ghsm I_h (\NsM\circ \hat X_{h,*}^{m+1} \exM) \|_{L^2(\Ghsm)}^2
	\notag\\
	&\lesssim
	\| \nabla_\Ghsm r_h^{m+1} \|_{L^2(\Ghsm)}
	(\| \nabla_\Ghsm \ehxM \|_{L^2(\Ghsm)}
	+
	\| \nabla_\Ghsm \exM \|_{L^2(\Ghsm)}
	)
	\quad\mbox{(using \eqref{eq:geo_rel_1})}
	\notag\\
	&\lesssim
	h^{-1}\| I_h(\TsM\circ \hat X_{h,*}^{m+1} \exM) \|_{L^4(\Ghsm)}^2
	(\| \nabla_\Ghsm \ehxM \|_{L^2(\Ghsm)}
	+
	\| \nabla_\Ghsm \exM \|_{L^2(\Ghsm)}
	)
	\notag\\
	&\hspace{270pt}\mbox{(using \eqref{eq:geo_rel_2})}
	\notag\\
	&\lesssim
	h^{-1}(\| I_h\Tsm(\exM - \ehxm) \|_{L^4(\Ghsm)}^2
	+
	\| I_h(\TsM\circ \hat X_{h,*}^{m+1}-\Tsm \circ \hat X_{h,*}^{m})\exM \|_{L^4(\Ghsm)}^2
	)
	\notag\\
	&\quad
	\times
	(\| \nabla_\Ghsm \ehxM \|_{L^2(\Ghsm)}
	+
	\| \nabla_\Ghsm \exM \|_{L^2(\Ghsm)}
	)
	\qquad\mbox{(by $I_h\Tsm\ehxm=0$)}
	\notag\\
	&\lesssim
	h^{-1}\tau^2
	\Big(\Big\| \frac{\exM - \ehxm}{\tau} \Big\|_{L^4(\Ghsm)}^2
	+
	\Big\| \frac{\exM - \ehxm}{\tau} \Big\|_{W^{1,\infty}(\Ghsm)}^2
	\| \exM \|_{L^4(\Ghsm)}^2
	\notag\\
	&\quad
	+
	\| \exM \|_{L^4(\Ghsm)}^2
	\Big)
	\times
	(\| \nabla_\Ghsm \ehxM \|_{L^2(\Ghsm)}
	+
	\| \nabla_\Ghsm \exM \|_{L^2(\Ghsm)}
	) 
	\notag\\
	&\hspace{210pt}
	\quad\mbox{(using Lemma \ref{lemma:hatX-W1inf})}
	,
\end{align}
which is a higher-order perturbation term according to \eqref{eq:ex-bbd}--\eqref{eq:ehx-bbd}.

Using $I_h\Tsm\ehxm=0$ and the orthogonality lemma (Lemma \ref{lemma:NT_stab_ref}), we get
\begin{align}\label{eq:M4}
	M_4^m
	&\lesssim
	\| \nabla_\Ghsm I_h\Nsm \exM \|_{L^2(\Ghsm)} \|  I_h \Tsm (\exM - \ehxm) \|_{L^2(\Ghsm)}
	\notag\\
	&\lesssim
	\tau \| \exM \|_{H^1(\Ghsm)}  \Big\| \frac{\exM - \ehxm}{\tau} \Big\|_{L^2(\Ghsm)}
	\notag\\
	&\lesssim
	\tau \Big(\| \ehxm \|_{H^1(\Ghsm)} +
	\tau \Big\| \frac{\exM - \ehxm}{\tau} \Big\|_{H^1(\Ghsm)} \Big) \Big\| \frac{\exM - \ehxm}{\tau} \Big\|_{L^2(\Ghsm)} .
\end{align}

		The proof is complete by combining \eqref{eq:M1}--\eqref{eq:M4} and using the induction hypothesis \eqref{eq:ind-hypo}.



\begin{ack}
The first author gratefully acknowledges support from the University of Michigan through Departmental Research Grant U043382, and also thanks Yupei Xie for helpful discussions. Some of the initial ideas of this work were formed while the first author was a postdoctoral fellow at The Hong Kong Polytechnic University, partially supported by the Hong Kong Research Grants Council under project no.~RFS2324-5S03.
\end{ack}



\bibliographystyle{abbrv}
\bibliography{MCF_Vt_final}

@article{ST18,
  title={A highly accurate boundary integral equation method for surfactant-laden drops in 3{D}},
  author={Sorgentone, Chiara and Tornberg, Anna-Karin},
  journal={J. Comput. Phys.},
  volume={360},
  pages={167--191},
  year={2018},
  publisher={Elsevier}
}

@article{GTZ2026,
	title={Structure-preserving parametric finite element methods for two-phase {S}tokes flow based on {L}agrange multiplier approaches},
	author={Garcke, Harald and Trautwein, Dennis and Zhang, Ganghui},
	journal={J. Comput. Phys.},
	pages={114922},
	year={2026},
	publisher={Elsevier}
}

@article{OK2005,
	title={A conservative level set method for two phase flow},
	author={Olsson, Elin and Kreiss, Gunilla},
	journal={J. Comput. Phys.},
	volume={210},
	number={1},
	pages={225--246},
	year={2005},
	publisher={Elsevier}
}

@article{OKZ2007,
	title={A conservative level set method for two phase flow {II}},
	author={Olsson, Elin and Kreiss, Gunilla and Zahedi, Sara},
	journal={J. Comput. Phys.},
	volume={225},
	number={1},
	pages={785--807},
	year={2007},
	publisher={Elsevier}
}

@book{Lunardi2018book,
	title={{I}nterpolation {T}heory},
	author={Lunardi, Alessandra},
	volume={16},
	year={2018},
	publisher={Springer}
}

@incollection{GHH2006,
	title={On the equation $div\, u= g$ and {B}ogovskii’s operator in {S}obolev spaces of negative order},
	author={Gei{\ss}ert, Matthias and Heck, Horst and Hieber, Matthias},
	booktitle={Partial Differential Equations and Functional Analysis: The Philippe Cl{\'e}ment Festschrift},
	pages={113--121},
	year={2006},
	publisher={Springer}
}

@article{FS94,
	title={Generalized resolvent estimates for the {S}tokes system in bounded and unbounded domains},
	author={Farwig, Reinhard and Sohr, Hermann},
	journal={J. Math. Soc. Japan.},
	volume={46},
	number={4},
	pages={607--643},
	year={1994},
	publisher={The Mathematical Society of Japan}
}

@article{NST16,
  title={A diffuse interface model for two-phase ferrofluid flows},
  author={Nochetto, Ricardo H and Salgado, Abner J and Tomas, Ignacio},
  journal={Comput. Methods Appl. Mech. Engrg.},
  volume={309},
  pages={497--531},
  year={2016},
  publisher={Elsevier}
}

@book{Necas11book,
	title={{D}irect {M}ethods in the {T}heory of {E}lliptic {E}quations},
	author={Ne\v{c}as, Jindrich},
	year={2011},
	publisher={Springer Science \& Business Media}
}

@book{SS10book,
	title={{B}oundary {E}lement {M}ethods},
	author={Sauter, Stefan A and Schwab, Christoph},
	booktitle={Boundary Element Methods},
	pages={183--287},
	year={2010},
	publisher={Springer}
}

@article{Li19,
	title={Analyticity, maximal regularity and maximum-norm stability of semi-discrete finite element solutions of parabolic equations in nonconvex polyhedra},
	author={Li, Buyang},
	journal={Math. Comp.},
	volume={88},
	number={315},
	pages={1--44},
	year={2019}
}

@article{DS80,
	title={Polynomial approximation of functions in {S}obolev spaces},
	author={Dupont, Todd and Scott, Ridgway},
	journal={Math. Comp.},
	volume={34},
	number={150},
	pages={441--463},
	year={1980}
}

@book{BBF13book,
	title={{M}ixed {F}inite {E}lement Methods and {A}pplications},
	author={Boffi, Daniele and Brezzi, Franco and Fortin, Michel and others},
	volume={44},
	year={2013},
	publisher={Springer}
}

@article{SZ90,
	title={Finite element interpolation of nonsmooth functions satisfying boundary conditions},
	author={Scott, L Ridgway and Zhang, Shangyou},
	journal={Math. Comp.},
	volume={54},
	number={190},
	pages={483--493},
	year={1990}
}

@book{Brenner08,
	title={The {M}athematical {T}heory of {F}inite {E}lement {M}ethods},
	author={Brenner, Susanne C and L. Ridgway Scott},
	year={2008},
	publisher={Springer}
}

@article{BGN15a,
	title={A stable parametric finite element discretization of two-phase {N}avier--{S}tokes flow},
	author={Barrett, John W and Garcke, Harald and N{\"u}rnberg, Robert},
	journal={J. Sci. Comput.},
	volume={63},
	pages={78--117},
	year={2015},
	publisher={Springer}
}

@article{BGN15b,
	title={Stable finite element approximations of two-phase flow with soluble surfactant},
	author={Barrett, John W and Garcke, Harald and N{\"u}rnberg, Robert},
	journal={J. Comput. Phys.},
	volume={297},
	pages={530--564},
	year={2015},
	publisher={Elsevier}
}

@article{Lenoir86,
	title={Optimal isoparametric finite elements and error estimates for domains involving curved boundaries},
	author={Lenoir, Marc},
	journal={SIAM J. Numer. Anal.},
	volume={23},
	number={3},
	pages={562--580},
	year={1986},
	publisher={SIAM}
}

@article{BGN16,
	title={A stable numerical method for the dynamics of fluidic membranes},
	author={Barrett, John W and Garcke, Harald and N{\"u}rnberg, Robert},
	journal={Numer. Math.},
	volume={134},
	pages={783--822},
	year={2016},
	publisher={Springer}
}

@article{DL24,
	title={New artificial tangential motions for parametric finite element approximation of surface evolution},
	author={Duan, Beiping and Li, Buyang},
	journal={SIAM J. Sci. Comput.},
	volume={46},
	number={1},
	pages={A587--A608},
	year={2024},
	publisher={SIAM}
}

@book{GT2001,
  title = {{Elliptic Partial Differential Equations of Second Order}},
  author = {D. Gilbarg and N. S. Trudinger},
  series = {Classics in mathematics},
  publisher = {{Springer, Second Edition}},
  location = {},
  year = {2001}
}

@book{Lee18,
  title = {{I}ntroduction to {R}iemannian {M}anifolds},
  author = {J. Lee},
  series = {},
  publisher = {{Springer}},
  location = {},
  year = {2018}
}

@article{Bansch01,
	title={Finite element discretization of the {N}avier--{S}tokes equations with a free capillary surface},
	author={B{\"a}nsch, Eberhard},
	journal={Numer. Math.},
	volume={88},
	number={2},
	pages={203--235},
	year={2001},
	publisher={Springer}
}

@article{GLW22,
author = {X. Gui and B. Li and J. Wang},
title = {{Convergence of renormalized finite element methods for heat flow of harmonic maps}},
journal = {SIAM J. Numer. Anal.},
volume = {60},
number = {},
pages = {312--338},
year = {2022},
}

@incollection{BGN20,
	AUTHOR={Barrett, J.W. and Garcke, H. and N{\"u}rnberg, R.},
	TITLE={Parametric finite element approximations of curvature driven interface evolutions},
	booktitle={Handb. Numer. Anal.},
	volume={21},
	pages={275--423},
	year={2020},
}

@article{EKL24,
	title={Numerical analysis of an evolving bulk--surface model of tumour growth},
	author={Edelmann, Dominik and Kov{\'a}cs, Bal{\'a}zs and Lubich, Christian},
	journal={IMA J. Numer. Anal.},
	volume={45},
	number={5},
	pages={2581--2627},
	year={2025},
	publisher={Oxford University Press}
}

@article{Garcke13,
	title={Curvature driven interface evolution},
	author={Garcke, Harald},
	journal={Jahresber. Dtsch. Math.-Ver.},
	volume={115},
	number={2},
	pages={63--100},
	year={2013},
	publisher={Springer}
}

@book{QM2013book,
	title={{A}n {I}ntroduction to the {R}egularity {T}heory for {E}lliptic {S}ystems, {H}armonic {M}aps and {M}inimal {G}raphs},
	author={Giaquinta, Mariano and Martinazzi, Luca},
	year={2013},
	publisher={Springer Science \& Business Media}
}

@article{KLL-Willmore,
  title = {{A convergent evolving finite element algorithm for Willmore flow of closed surfaces}},
  author = {Bal\'azs Kov\'acs and Buyang Li and Christian Lubich},
  volume = {149},
  number = {},
  pages = {595--643},
  journal = {Numer. Math.},
  year = {2021}
  }

@article{EKS19,
	title={A tractable mathematical model for tissue growth},
	author={Eyles, Joe and King, John R and Styles, Vanessa},
	journal={Interface Free. Bound.},
	volume={21},
	number={4},
	pages={463--493},
	year={2019}
}

@article{BL22A,
    author = {Genming Bai and Buyang Li},
    title = "{Erratum: Convergence of Dziuk's semidiscrete finite element method for mean curvature flow of closed surfaces with high-order finite elements}",
    journal = {SIAM J. Numer. Anal.},
    volume = {61},
    number = {3},
    pages = {1609-1612 },
    year = {2023},
}

@article{BL24FOCM,
	title={A new approach to the analysis of parametric finite element approximations to mean curvature flow},
	author={Bai, Genming and Li, Buyang},
	journal={Found. Comput. Math.},
	volume={24},
	number={5},
	pages={1673--1737},
	year={2024},
	publisher={Springer}
}

@article{BGV26,
	author = {Genming Bai and Harald Garcke and Shravan Veerapeneni},
	title = "{Convergence analysis for the {B}arrett–{G}arcke–{N}ürnberg method of transport type for evolving curves.}",
	journal = {Numer. Math.},
	volume = {158},
	year = {2026},
	pages = {361--410},
}

@article{Abels07,
	title={On generalized solutions of two-phase flows for viscous incompressible fluids},
	author={Abels, Helmut},
	journal={Interface Free. Bound.},
	volume={9},
	number={1},
	pages={31--65},
	year={2007}
}

@article{BL24,
	title={Convergence of a stabilized parametric finite element method of the {B}arrett--{G}arcke--{N}{\"u}rnberg type for curve shortening flow},
	author={Bai, Genming and Li, Buyang},
	journal={Math. Comp.},
	volume={94},
	number={355},
	pages={2151--2220},
	year={2025}
}

@article{TK96,
	title={Large-time existence of surface waves in incompressible viscous fluids with or without surface tension},
	author={Tani, Atusi and Tanaka, Naoto},
	journal={Arch. Ration. Mech. Anal.},
	volume={130},
	number={4},
	pages={303--314},
	year={1995},
	publisher={Springer}
}

@article{Beale81,
	title={The initial value problem for the {N}avier--{S}tokes equations with a free surface},
	author={Beale, J Thomas},
	journal={Commun. Pure Appl. Math.},
	volume={34},
	number={3},
	pages={359--392},
	year={1981},
	publisher={Wiley Online Library}
}

@article{Beale84,
	title={Large-time regularity of viscous surface waves},
	author={Beale, J Thomas},
	journal={Arch. Ration. Mech. Anal.},
	volume={84},
	number={4},
	pages={307--352},
	year={1984},
	publisher={Springer}
}

@article{LS90,
	title={Large time existence of small viscous surface waves without surface tension},
	author={Lynn, Donna and Sylvester, Gates},
	journal={Commun. Partial. Differ. Equ.},
	volume={15},
	number={6},
	pages={823--903},
	year={1990},
	publisher={Taylor \& Francis}
}

@article{ER21,
	title={A unified theory for continuous-in-time evolving finite element space approximations to partial differential equations in evolving domains},
	author={Elliott, Charles M and Ranner, Thomas},
	journal={IMA J. Numer. Anal.},
	volume={41},
	number={3},
	pages={1696--1845},
	year={2021},
	publisher={Oxford University Press}
}

@article{Fu20,
	title={Arbitrary {L}agrangian--{E}ulerian hybridizable discontinuous {G}alerkin methods for incompressible flow with moving boundaries and interfaces},
	author={Fu, Guosheng},
	journal={Comput. Methods Appl. Mech. Engrg.},
	volume={367},
	pages={113158},
	year={2020},
	publisher={Elsevier}
}

@book{LM2012book,
	title={Non-homogeneous {B}oundary {V}alue {P}roblems and {A}pplications: Vol. 1},
	author={Lions, Jacques Louis and Magenes, Enrico},
	volume={1},
	year={2012},
	publisher={Springer Science \& Business Media}
}

@article{DLY22,
	title={An energy diminishing arbitrary {L}agrangian--{E}ulerian finite element method for two-phase {N}avier--{S}tokes flow},
	author={Duan, Beiping and Li, Buyang and Yang, Zongze},
	journal={J. Comput. Phys.},
	volume={461},
	pages={111215},
	year={2022},
	publisher={Elsevier}
}

@article{GRR06,
	title={A finite element based level set method for two-phase incompressible flows},
	author={Gro{\ss}, Sven and Reichelt, Volker and Reusken, Arnold},
	journal={Comput. Vis. Sci.},
	volume={9},
	number={4},
	pages={239--257},
	year={2006},
	publisher={Springer}
}

@article{GNZ24,
	title={Arbitrary {L}agrangian--{E}ulerian finite element approximations for axisymmetric two-phase flow},
	author={Garcke, Harald and N{\"u}rnberg, Robert and Zhao, Quan},
	journal={Comput. Math. Appl.},
	volume={155},
	pages={209--223},
	year={2024},
	publisher={Elsevier}
}

@article{ZWT10,
	title={Simulating two-phase viscoelastic flows using moving finite element methods},
	author={Zhang, Yubo and Wang, Heyu and Tang, Tao},
	journal={Commun. Comput. Phys.},
	volume={7},
	number={2},
	pages={333--349},
	year={2010}
}

@article{GNZ23,
	title={Unfitted finite element methods for axisymmetric two-phase flow},
	author={Garcke, Harald and N{\"u}rnberg, Robert and Zhao, Quan},
	journal={J. Sci. Comput.},
	volume={97},
	number={1},
	pages={14},
	year={2023},
	publisher={Springer}
}

@article{OQS22,
	title={A finite element method for two-phase flow with material viscous interface},
	author={Olshanskii, Maxim and Quaini, Annalisa and Sun, Qi},
	journal={Comput. Methods Appl. Math.},
	volume={22},
	number={2},
	pages={443--464},
	year={2022},
	publisher={De Gruyter}
}

@article{FZ23,
	title={A cut finite element method for two-phase flows with insoluble surfactants},
	author={Frachon, Thomas and Zahedi, Sara},
	journal={J. Comput. Phys.},
	volume={473},
	pages={111734},
	year={2023},
	publisher={Elsevier}
}

@article{XBS13,
	title={Analytical and computational methods for two-phase flow with soluble surfactant},
	author={Xu, Kuan and Booty, MR and Siegel, Michael},
	journal={SIAM J. Appl. Math.},
	volume={73},
	number={1},
	pages={523--548},
	year={2013},
	publisher={SIAM}
}

@article{VGZB09,
	title={A boundary integral method for simulating the dynamics of inextensible vesicles suspended in a viscous fluid in 2{D}},
	author={Veerapaneni, Shravan K and Gueyffier, Denis and Zorin, Denis and Biros, George},
	journal={J. Comput. Phys.},
	volume={228},
	number={7},
	pages={2334--2353},
	year={2009},
	publisher={Elsevier}
}

@article{VRBZ11,
	title={A fast algorithm for simulating vesicle flows in three dimensions},
	author={Veerapaneni, Shravan K and Rahimian, Abtin and Biros, George and Zorin, Denis},
	journal={J. Comput. Phys.},
	volume={230},
	number={14},
	pages={5610--5634},
	year={2011},
	publisher={Elsevier}
}

@article {KLL17,
    author = {Bal\'azs Kov{\'a}cs and Buyang Li and Christian Lubich and Christian A. Power Guerra},
     title = {{Convergence of finite elements on an evolving surface driven by diffusion on the surface}},
   journal = {Numer. Math.},
    volume = {137},
      year = {2017},
    number = {},
     pages = {643--689},
 }

@article {KLL19,
    author = {Bal\'azs Kov{\'a}cs and Buyang Li and Christian Lubich},
     title = {{A convergent evolving finite element algorithm for mean curvature flow of closed surfaces}},
   journal = {Numer. Math.},
    volume = {143},
      year = {2019},
    number = {},
     pages = {797--853},
 }

@article{BGN2007JCP,
  title = {A Parametric Finite Element Method for Fourth Order Geometric Evolution Equations},
  author = {John W. Barrett and Harald Garcke and Robert N\"urnberg},
  volume = {222},
  number = {},
  pages = {441--467},
  journal = {J. Comput. Phys.},
  year = {2007}
}

@article{BGN2008JCP,
  author = {John W. Barrett and Harald Garcke and Robert N\"urnberg},
  date = {2008-04},
  journaltitle = {J. Comput. Phys.},
  volume = {227},
  number = {},
  pages = {4281--4307},
  langid = {english},
  title = {On the Parametric Finite Element Approximation of Evolving Hypersurfaces in {$\mathbb{R}^3$}},
  journal = {J. Comput. Phys.},
  year = {2008}
}

@article{DDE05,
  title = {Computation of Geometric Partial Differential Equations and Mean Curvature Flow},
  author = {Deckelnick, Klaus and Dziuk, Gerhard and Elliott, Charles M.},
  date = {2005-05},
  journaltitle = {Acta Numer.},
  volume = {14},
  pages = {139--232},
  publisher = {{Cambridge University Press}},
  langid = {english},
  journal = {Acta Numer.},
  year = {2005}
}

@article{DE13,
  title = {Finite Element Methods for Surface {{PDEs}}},
  author = {Gerhard Dziuk and Charles M. Elliott},
  date = {2013-05},
  journaltitle = {Acta Numer.},
  volume = {22},
  pages = {289--396},
  publisher = {{Cambridge University Press}},
  langid = {english},
  journal = {Acta Numer.},
  year = {2013}
}

@article{LQ25,
  title={Optimal convergence of the arbitrary {L}agrangian--{E}ulerian interface tracking method for two-phase {N}avier--{S}tokes flow without surface tension},
  author={Li, Buyang and Ma, Shu and Qiu, Weifeng},
  journal={IMA J. Numer. Anal.},
  pages={draf003},
  year={2025},
  publisher={Oxford University Press}
}

@article{Kov18,
  title = {High-Order Evolving Surface Finite Element Method for Parabolic Problems on Evolving Surfaces},
  author = {Kov\'acs, Bal\'azs},
  date = {2018-01-25},
  journaltitle = {IMA J. Numer. Anal.},
  volume = {38},
  number = {1},
  pages = {430--459},
  journal = {IMA J. Numer. Anal.},
  year = {2018}
}
%
%
%
%
%
%
%

\end{document}